\documentclass[english,11pt, twoside]{article}
\usepackage[latin1]{inputenc}
\usepackage[T1]{fontenc}
\usepackage{lmodern}
\usepackage[a4paper]{geometry}
\usepackage{mathtools}
\usepackage{amssymb}
\usepackage{amsthm}
\usepackage{mathrsfs}
\usepackage{xcolor}
\usepackage{latexsym}
\usepackage{stmaryrd}
\usepackage{soul}
\usepackage{booktabs}
\usepackage{multirow}
\usepackage{placeins}
\usepackage{tikz}
\usetikzlibrary{calc,arrows}  
\usepackage{pgfmath}
\usepackage{xargs}
\usepackage{dsfont} 
\usepackage{appendix}
\usepackage{ifthen}
\usepackage{supertabular}
\usepackage{csquotes}
\usepackage{babel}
\usetikzlibrary{babel}  
\usepackage[autolanguage,np]{numprint}
\usepackage[breaklinks]{hyperref}
\usepackage{enumitem}
\usepackage{listings}

\newcommand{\po}{\Bigg (}
\newcommand{\pf}{\Bigg )}

\newcommand{\ud}[1]{\tikzset{every picture/.style={scale=#1}}}
\newcommand{\arrowdist}{0.075}

\newcommandx{\ngone}[5][3=black, 4=0, 5=0]{
	\pgfmathsetmacro{\var}{#2+1};
	\ifthenelse{#1<6}{\pgfmathsetmacro{\diagradius}{1}}{\pgfmathsetmacro{\diagradius}{1.5}}
	\pgfmathsetmacro{\diagsmallradius}{0.075};
	\pgfmathsetmacro{\diagdegree}{#1};
	\pgfmathsetmacro{\diaglabeldist}{\diagradius/3.5};
	\pgfmathsetmacro{\diagn}{\diagdegree+1};
	\ifthenelse{\equal{#2}{}}{\draw[#3] (0+#4,0+#5) coordinate (C) circle (\diagradius);}{\draw[white] (0,0) coordinate (C) circle (\diagradius);}
	\phantom{ \draw (0+#4,0+#5)circle (\diaglabeldist/2+\diagradius);}
	\foreach \diagi in {1,...,\diagdegree}{
		\filldraw (C) + (270-\diagi*360/\diagn:\diagradius) coordinate (A\diagi) circle(\diagsmallradius);}
	\ifthenelse{\equal{#2}{}}{}{\ifthenelse{\equal{#2}{\diagdegree}}{\draw (A0) arc (270:270-\diagdegree*360/\diagn:\diagradius);}{	\draw (A\var) arc (270-\var*360/\diagn:270-(\var+\diagdegree)*360/\diagn:\diagradius);}}
	\filldraw[fill=white] (C) + (270:\diagradius) coordinate (A0)  circle(\diagsmallradius);
}
\newcommand{\arrow}[3]{
	\draw[#3, style={shorten  >= \arrowdist cm ,shorten  <= \arrowdist cm},line width=1,>=stealth,line width=0.5,->] (A#1) -- (A#2);}

\newcommand{\aretes}[2]{
	\foreach  \i/\j in #1{%
		\arrow{\i}{\j}{#2};
	}
}

\newcommand{\numbers}{
	\foreach \diagi in {1,...,\diagdegree}{
		\draw (A\diagi) ++ (270-\diagi*360/\diagn:\diaglabeldist) node{\tiny $\diagi$};
	}
}

\newcommand{\colsom}[1]{
	\foreach \i in #1{
		\ifthenelse{\equal{\i}{0}}{\filldraw[fill=yellow] (C) + (270:\diagradius) coordinate (A0)  circle(\diagsmallradius);}{	\filldraw[fill=red] (C) + (270-\i*360/\diagn:\diagradius) circle(\diagsmallradius);}
	}		
}

\newcommandx{\defDiag}[3][3=->]{\begin{tikzpicture}[baseline,x=1cm,y=1cm]
	\pgfmathsetmacro{\diagsmallradius}{0.075};
	\draw(0,0) coordinate (C) circle [radius=1.5]; %cercle du diagramme
	\draw(0,-1.5) circle [radius=\diagsmallradius]; %racine
	\draw[fill](0,-1.5)circle [radius=\diagsmallradius]; %racine
	\filldraw (C) + (270:1.5) coordinate (A0);
	\foreach \diagi in {1,...,12}{
		\filldraw (C) + (270-\diagi*360/13:1.5) coordinate (A\diagi);}
	\filldraw ($(A6)!0.5!(A7)$) coordinate (A13);
	\foreach \i/\j in #1{
		\draw[fill] (A\i) circle [radius=\diagsmallradius];
		\draw[fill] (A\j) circle [radius=\diagsmallradius];
		\draw[black, style={shorten  >=0.06  cm ,shorten  <= 0.06  cm},line width=1,>=stealth,line width=0.5,#3] (A\i) -- (A\j);
	}
	\foreach \u/\v/\nom in #2{
		\filldraw ($(A\u)!0.55!(A\v)$) coordinate (B);
		\draw[black, style={shorten  >= 0.06 cm ,shorten  <= 0.06 cm},line width=1,>=stealth,line width=0.5,#3] (B) -- (A\v);
		\draw[fill=white] (B) circle [radius=0.3]; %cercle pour D_{i}
		\draw (B) node{\tiny$\nom$}; % nom cercle pour D_{i}
	}
	\draw[fill=white](0,-1.5)circle [radius=\diagsmallradius];
	\phantom{ \draw (0,0) circle [radius=12/7];}
	\end{tikzpicture}}

\newcommand{\InsDeuxDiag}[2]{\begin{tikzpicture}[baseline,x=1cm,y=1cm]
	\pgfmathsetmacro{\diagsmallradius}{0.075};
	\pgfmathsetmacro{\radiusdiag}{2.7};
	\draw(0,0) coordinate (C) circle [radius=\radiusdiag]; %cercle du diagramme
	\draw(0,-\radiusdiag) circle [radius=\diagsmallradius]; %racine
	\draw[fill](0,-\radiusdiag)circle [radius=\diagsmallradius]; %racine
	\filldraw (C) + (270:\radiusdiag) coordinate (A0);
	\foreach \diagi in {1,...,15}{
		\filldraw (C) + (270-\diagi*360/16:\radiusdiag) coordinate (A\diagi);}
	\foreach \i/\j in #1{
		\draw[fill] (A\i) circle [radius=\diagsmallradius];
		\draw[fill] (A\j) circle [radius=\diagsmallradius];
		\draw[black, style={shorten  >=0.06  cm ,shorten  <= 0.06  cm},line width=1,>=stealth,line width=0.5,->] (A\i) -- (A\j);
	}
	\foreach \u/\v/\nom in #2{
		\filldraw ($(A\u)!0.65!(A\v)$) coordinate (B);
		\draw[black, style={shorten  >= 0.06 cm ,shorten  <= 0.06 cm},line width=1,>=stealth,line width=0.5,->] (B) -- (A\v);
		\draw[fill=white] (B) circle [radius=0.3]; %cercle pour D_{i}
		\draw (B) node{\tiny$\nom$}; % nom cercle pour D_{i}
	}
	\draw[fill=white](0,-\radiusdiag)circle [radius=\diagsmallradius];
	\phantom{ \draw (0,0) circle [radius=8*\radiusdiag/7];}
	\end{tikzpicture}}

\newcommand{\exempleintro}{
	\begin{tikzpicture}[baseline,x=1cm,y=1cm]
	\ngone{3}{}[white];
	\numbers;
	\aretes{{1/0,2/1,3/1}}{red};
	\foreach \i/\j [count = \k] in {1/0,2/1,3/1}{
		\draw ($(A\i)!0.5!(A\j)$) node{\tiny \textcolor{red}{$\k$}};
	}
	\foreach \i in {0,...,3}{
		\draw (225-\i*360/\diagn:\diagradius+\diaglabeldist) node{\tiny\textcolor{blue}{$\i$}};
	}
		\foreach \u [count = \k] in {(0,-1),(-1,0),(0,1),(1,0)}{
		\draw[blue,style={shorten  >= 0.075cm ,shorten  <= 0.075cm},line width=1,>=stealth,line width=0.5,->] \u arc (270-(\k-1)*360/\diagn:270-\k*360/\diagn:\diagradius);
	}
\end{tikzpicture}}

\newcommand{\diagramme}[8]{\begin{tikzpicture}[baseline,x=1cm,y=1cm]
	\ngone{#1}{#8};
	\ifthenelse{\equal{#2}{n}}{\numbers}{};
	\aretes{#3}{black};
	\aretes{#4}{red};
	\aretes{#5}{blue};
	\aretes{#6}{green};
	\colsom{#7};
	\end{tikzpicture}}

\newcommand{\diagnot}[3]{\begin{tikzpicture}[baseline,x=1cm,y=1cm]
	\ngone{9}{};
	\numbers;
	\aretes{#1}{black};
	\aretes{#2}{red};
	\colsom{#3};
	\foreach \i/\j [count = \k] in #2{
		\draw coordinate (M\k) ++ ($(A\i)!0.5!(A\j)$) node{\tiny\textcolor{red}{$\k$}};
	}	
\end{tikzpicture}}

\newcommand{\colldiag}[1]{\begin{tikzpicture}[baseline,x=1cm,y=1cm]
		\pgfmathsetmacro{\dep}{2};
		\pgfmathsetmacro{\cradius}{1};
		\pgfmathsetmacro{\csmallradius}{0.075};
		\draw[white] (-0.15,1.15) coordinate (C1) circle (\cradius);
		\draw[white] (1.101,-0.459) coordinate (C2) circle (\cradius);
		\draw (-0.15,0.15) coordinate (U1) arc (270:-18:\cradius);
		\draw (0.15,-0.15) coordinate (U2) arc (-198:90:\cradius);
		\foreach \i in {0,...,9}{
			\ifthenelse{\i<5}{\filldraw (C1) + (270-\i*72:1) coordinate (A\i) circle(\csmallradius);}{\filldraw (C2) + (90-\i*72:1) coordinate (A\i) circle(\csmallradius);}
		}
		\ifthenelse{\equal{#1}{2}}{\pgfmathsetmacro{\dep}{0}}{};
		\ifthenelse{\equal{#1}{3}}{\pgfmathsetmacro{\dep}{4}}{};
		\foreach \i in {0,1,2,3}{
			\pgfmathparse{int{mod(\i+\dep,8)}};
			\pgfmathsetmacro{\s}{\pgfmathresult};
			\draw (A\i) ++ (270-\i*72:0.286) node{\tiny $\s$};
		}
			
		\foreach \i in {5,6,7,8}{
			\pgfmathparse{int{mod(\i+\dep-1,8)}};
			\pgfmathsetmacro{\s}{\pgfmathresult};
			\draw (A\i) ++ (90-\i*72:0.286) node{\tiny $\s$};	
		}
		\ifthenelse{\equal{#1}{1}}{
			\aretes{{1/0,2/1,3/4,6/5,8/7,9/7}}{black};
			\filldraw[fill=red]  (A0) circle (\csmallradius);\filldraw[fill=red]  (A9) circle (\csmallradius);\filldraw[fill=white]  (A4) circle (\csmallradius);\filldraw[fill=red]  (A5) circle (\csmallradius);\filldraw[fill=yellow]  (A7) circle (\csmallradius);\draw[red, style={dashed},line width=1,>=stealth,line width=0.5,->] ($(A4)!0.5!(A5)$) -- ($(A0)!0.5!(A9)$);}{}
		\ifthenelse{\equal{#1}{2}}{
			\aretes{{1/3,2/1,3/4,6/5,7/6,8/9}}{black};
			\filldraw[fill=yellow]  (A0) circle (\csmallradius);\filldraw[fill=yellow]  (A9) circle (\csmallradius);\filldraw[fill=red]  (A4) circle (\csmallradius);\filldraw[fill=red]  (A5) circle (\csmallradius);\draw[red, style={dashed},line width=1,>=stealth,line width=0.5,->] ($(A4)!0.5!(A5)$) -- ($(A0)!0.5!(A9)$);}{}
		\ifthenelse{\equal{#1}{3}}{
			\aretes{{1/3,2/1,3/4,6/5,7/6,8/9}}{black};
			\filldraw[fill=red]  (A0) circle (\csmallradius);\filldraw[fill=red]  (A9) circle (\csmallradius);\filldraw[fill=yellow]  (A4) circle (\csmallradius);\filldraw[fill=yellow]  (A5) circle (\csmallradius);\draw[red, style={dashed},line width=1,>=stealth,line width=0.5,->] ($(A0)!0.5!(A9)$) -- ($(A4)!0.5!(A5)$);}{}
	\end{tikzpicture}
}

\newcommand{\tun}{\begin{picture}(5,0)(-2,-1)
\put(0,0){\circle*{2}}
\end{picture}}

\newcommand{\tdeux}{\begin{picture}(7,7)(0,-1)
\put(3,0){\circle*{2}}
\put(3,0){\line(0,1){5}}
\put(3,5){\circle*{2}}
\end{picture}}

\newcommand{\ttroisun}{\begin{picture}(15,8)(-5,-1)
\put(3,0){\circle*{2}}
\put(-0.65,0){$\vee$}
\put(6,7){\circle*{2}}
\put(0,7){\circle*{2}}
\end{picture}}
\newcommand{\ttroisdeux}{\begin{picture}(5,12)(-2,-1)
\put(0,0){\circle*{2}}
\put(0,0){\line(0,1){5}}
\put(0,5){\circle*{2}}
\put(0,5){\line(0,1){5}}
\put(0,10){\circle*{2}}
\end{picture}}

\newcommand{\tquatreun}{\begin{picture}(15,12)(-5,-1)
\put(3,0){\circle*{2}}
\put(-0.65,0){$\vee$}
\put(6,7){\circle*{2}}
\put(0,7){\circle*{2}}
\put(3,7){\circle*{2}}
\put(3,0){\line(0,1){7}}
\end{picture}}

\newcommand{\tquatretrois}{\begin{picture}(15,18)(-5,-1)
\put(3,0){\circle*{2}}
\put(-0.65,0){$\vee$}
\put(6,7){\circle*{2}}
\put(0,7){\circle*{2}}
\put(6,14){\circle*{2}}
\put(6,7){\line(0,1){7}}
\end{picture}}

\newcommand{\tquatrequatre}{\begin{picture}(15,18)(-5,-1)
\put(3,5){\circle*{2}}
\put(-0.65,5){$\vee$}
\put(6,12){\circle*{2}}
\put(0,12){\circle*{2}}
\put(3,0){\circle*{2}}
\put(3,0){\line(0,1){5}}
\end{picture}}

\newcommand{\tcinqcinq}{\begin{picture}(15,19)(-5,-1)
\put(3,0){\circle*{2}}
\put(-0.65,0){$\vee$}
\put(6,7){\circle*{2}}
\put(0,7){\circle*{2}}
\put(6,14){\circle*{2}}
\put(6,7){\line(0,1){7}}
\put(0,14){\circle*{2}}
\put(0,7){\line(0,1){7}}
\end{picture}}
\newcommand{\tcinqsix}{\begin{picture}(15,20)(-7,-1)
\put(3,0){\circle*{2}}
\put(-0.65,0){$\vee$}
\put(6,7){\circle*{2}}
\put(0,7){\circle*{2}}
\put(-3.65,7){$\vee$}
\put(3,14){\circle*{2}}
\put(-3,14){\circle*{2}}
\end{picture}}
\newcommand{\tcinqhuit}{\begin{picture}(15,26)(-5,-1)
\put(3,0){\circle*{2}}
\put(-0.65,0){$\vee$}
\put(6,7){\circle*{2}}
\put(0,7){\circle*{2}}
\put(0,14){\circle*{2}}
\put(0,7){\line(0,1){7}}
\put(0,21){\circle*{2}}
\put(0,14){\line(0,1){7}}
\end{picture}}
\newcommand{\tcinqonze}{\begin{picture}(15,26)(-5,-1)
\put(3,5){\circle*{2}}
\put(-0.65,5){$\vee$}
\put(6,12){\circle*{2}}
\put(0,12){\circle*{2}}
\put(3,0){\circle*{2}}
\put(3,0){\line(0,1){5}}
\put(0,12){\line(0,1){7}}
\put(0,19){\circle*{2}}
\end{picture}}
\newcommand{\tcinqtreize}{\begin{picture}(5,26)(-2,-1)
\put(0,0){\circle*{2}}
\put(0,0){\line(0,1){7}}
\put(0,7){\circle*{2}}
\put(0,7){\line(0,1){7}}
\put(0,14){\circle*{2}}
\put(-3.65,14){$\vee$}
\put(-3,21){\circle*{2}}
\put(3,21){\circle*{2}}
\end{picture}}
\newcommand{\tcinqquatorze}{\begin{picture}(9,26)(-5,-1)
\put(0,0){\circle*{2}}
\put(0,0){\line(0,1){5}}
\put(0,5){\circle*{2}}
\put(0,5){\line(0,1){5}}
\put(0,10){\circle*{2}}
\put(0,10){\line(0,1){5}}
\put(0,15){\circle*{2}}
\put(0,15){\line(0,1){5}}
\put(0,20){\circle*{2}}
\end{picture}}

\newcommand{\tdtroisunb}[3]{\begin{picture}(20,12)(-5,-1)
	\put(3,0){\circle*{2}}
	\put(6,7){\circle*{2}}
	\put(0,7){\circle*{2}}
	\put(-0.65,0){$\vee$}
	\put(5,-2){\tiny #1}
	\put(9,5){\tiny #2}
	\put(-8,5){\tiny #3}
	\end{picture}}

\newcommand{\N}{\mathbb{N}} % entiers naturels
\newcommand{\R}{\mathbb{R}} % corps de base réels
\newcommand{\K}{\mathbb{K}} % corps de base
\DeclareMathOperator{\Prm}{Prim}
\newcommand{\Prim}[2]{\ifthenelse{\equal{#2}{}}{\Prm(#1)}{{\Prm(#1)}_{#2}}} % éléments primitifs de #1 de degré #2
\DeclareMathOperator{\Gn}{Gen}
\newcommand{\Gen}[2]{\ifthenelse{\equal{#2}{}}{\Gn(#1)}{{\Gn(#1)}_{#2}}} % éléments générateurs de #1 de degré #2
\newcommand{\Idl}[1]{\mathcal{#1}} % Noter un idéal de Hopf d'une algèbre de Hopf
\newcommand{\Sym}{Sym} % fonctions symétriques
\newcommand{\TR}{\mathcal{T}_{\textsc{R}}} % ensemble des arbres enracinés
\newcommand{\FRt}{\mathcal{F}_{\textsc{R}}} % ensemble des forets enracinés
\newcommand{\sym}[1]{\mathcal{S}(#1)} % écrire une algèbre symétrique
\newcommand{\g}[1]{\mathfrak{g}_{#1}} % algèbre de Lie g(#1)
\newcommand{\Ug}[1]{\mathcal{U}(\g{#1})} % algèbre enveloppante de l'algèbre de Lie g(#1)
\newcommand{\Hc}{\mathcal{H}_{\tiny\textsc{CK}}} % algèbre de Connes-Kreimer ie algèbre de Hopf des arbres enracinés 
\newcommand{\Hgl}{\mathcal{H}_{\tiny\textsc{GL}}} % algèbre de Grossman-Larson ie algèbre de Hopf des arbres enracinés dual 
\newcommand{\dual}[1]{#1^{\circledast}} % écriture du dual gradué de #1
\newcommand{\ot}{\otimes} % raccourci pour écrire le produit tensoriel
\newcommand{\De}{\Delta} %  raccourci pour écrire le delta du coproduit
\newcommand{\Rq}[1]{\ifthenelse{\equal{#1}{}}{\subparagraph{Remarks.}}{\subparagraph{Remark.}}} % remarque(s)
\newcommand{\Ex}[1]{\ifthenelse{\equal{#1}{}}{\subparagraph{Examples.}}{\subparagraph{Example.}}} % exemple(s)
\newcommand{\CQMM}{Cartier-Quillen-Milnor-Moore theorem} % pour ne pas à avoir à écrire les nom à la main à chaque fois.
\newcommand{\PH}{Poincaré-Hilbert} % pour ne pas à avoir à écrire les nom à la main à chaque fois
\newcommand{\HFdB}{\mathcal{H}_{FdB}} % notation algèbre de Hopf de Faà di Bruno
\newcommand{\Com}{\mathcal{C}om}
\newcommand{\PreLie}{\mathcal{P}re\mathcal{L}ie}
\newcommand{\PostLie}{\mathcal{P}ost\mathcal{L}ie}

\newtheorem{defi}{\indent Definition}
\newtheorem{lemma}[defi]{\indent Lemma}
\newtheorem{cor}[defi]{\indent Corollary}
\newtheorem{theo}[defi]{\indent Theorem}
\newtheorem{prop}[defi]{\indent Proposition}
\newtheorem{conj}[defi]{\indent Conjecture}

\newcommand{\Hd}{\mathcal{H}_{\tiny\textsc{D}}}
\newcommand{\D}{\mathcal{D}}
\newcommand{\Qd}{\mathcal{Q}}
\newcommand{\Ch}{\mathscr{C}} 
\newcommand{\Sp}{\mathscr{S}}
\newcommand{\sbT}{\mathcal{T}{\tiny\textsc{SB}}} % arbres enracinés sous-binaires
\newcommand{\sbF}{\mathcal{F}{\tiny\textsc{SB}}} % forêts enracinées sous-binaires
\newcommand{\HsbT}{\mathbf{SBT}} % algèbre de Hopf des arbres enracinés sous-binaires

\DeclareMathOperator{\interne}{int} % nb sommets internes

\begin{document}
\title{On the combinatorics of the Hopf algebra of dissection diagrams}
\date{}
\author{Cécile Mammez\\
	\small{\it Univ. Littoral Côte d'Opale,}\\
	\small{\it EA 2597 - LMPA Laboratoire de Mathématiques Pures et Appliquées Joseph Liouville,} \\ 
	\small{\it F-62228 Calais, France and CNRS, FR 2956, France}\\
	\small{e-mail : mammez@lmpa.univ-littoral.fr}
}
	
\maketitle

\textbf{Abstract.} In this article, we are interested in the Hopf algebra $\Hd$ of dissection diagrams introduced by Dupont in his thesis. We use the version with a parameter $x\in\K$. We want to study its underlying coalgebra. We conjecture it is cofree, except for a countable subset of $\K$. If $x=-1$ then we know there is no cofreedom. We easily see that $\Hd$ is a free commutative right-sided combinatorial Hopf algebra according to Loday and Ronco. So, there exists a pre-Lie structure on its graded dual. Furthermore $\dual{\Hd}$ and the enveloping algebra of its primitive elements are isomorphic. Thus, we can equip $\dual{\Hd}$ with a structure of Oudom and Guin. We focus on the pre-Lie structure on dissection diagrams and in particular on the pre-Lie algebra generated by the dissection diagram of degree $1$. We prove that it is not free.  We express a Hopf algebra morphism between the Grossman and Larson Hopf algebra and $\dual{\Hd}$ by using pre-Lie and Oudom and Guin structures. \\


\textbf{Résumé.} Dans cet article, nous nous intéressons à l'algèbre de Hopf à paramètre $\Hd$ des diagrammes de dissection introduite par Dupont dans sa thèse de doctorat. Nous cherchons plus particulièrement à étudier sa cogèbre sous-jacente. Nous conjecturons qu'elle est colibre excepté pour un ensemble dénombrable de paramètres. Il n'y a pas de coliberté lorsque le paramètre vaut $-1$. L'algèbre de Hopf $\Hd$ est un exemple d'algèbre de Hopf combinatoire commutative-associative-libre droite selon Loday et Ronco. Ceci implique l'existence d'une structure pré-Lie sur son dual gradué dont on sait qu'il est isomorphe à l'algèbre enveloppante de ses primitifs. Il est alors possible de munir $\dual{\Hd}$ d'une structure de Oudom et Guin. Nous nous intéréssons à la structure pré-Lie des diagrammes de dissection et plus particulièrement à la sous-algèbre pré-Lie non triviale engendrée par le diagramme de dissection de degré $1$. Nous  montrons que cette dernière n'est pas libre. Nous explicitons un morphisme d'algèbres de Hopf entre celle de Grossman et Larson et $\dual{\Hd}$ grâce aux structures pré-Lie et aux structures de Oudom et Guin.   \\

\textbf{Mots-clés.} Algèbres de Hopf combinatoires, diagrammes de dissection, arbres enracinés, algèbres pré-Lie, algèbres enveloppantes, morphisme, procédé d'insertion. \\


\tableofcontents
	
\allowdisplaybreaks
\ud{0.7}
\section*{Introduction} 

The Hopf algebra of dissection diagrams comes from the number theory and was introducted by Dupont \cite{Dupont2014a} \cite[chapter 2]{Dupont2014}. He considers the coproduct computation problem in the fundamental Hopf algebra of the category of mixed Hodge-Tate structures. Thanks to dissection diagrams, he aims at computing this coproduct for motivic dissection polylogorithms.

A dissection diagram of degree $n$ is a $(n+1)$-gon with a set of $n$ non-intersecting chords forming a planar rooted tree. Dupont builds a coalgebra structure on the symmetric algebra of dissection diagrams which makes it become a Hopf algebra denoted by $\Hd$. The coproduct is given with a parameter $x$ selection-quotient process  \cite[sections 2.1.2 and 2.1.3]{Dupont2014}. Dupont considers the case $x=-1$, defines then a decorated version and explains compatibilities between decorations and chords contraction in an oriented graph  \cite[section 2.1.4]{Dupont2014}.  The set of decorations must be a group and in this case, Dupont uses the complex numbers group. To each dissection diagram $D$, he defines dissection polylogarithms $I(D)$ as an absolutely convergent integral on a simplex of a differential form \cite[definition 2.3.5]{Dupont2014}. The differential form depends on decorated chords and the simplex on decorated sides of $D$ \cite[section 2.3.2]{Dupont2014}. He gives then a motivic version $I^{\mathcal{H}}(D)$ \cite[section 2.4.1]{Dupont2014} and computes the coproduct  \cite[section 2.4.2, theorem 2.4.9]{Dupont2014}.

In this article, we aims at understanding the dissection diagrams Hopf algebra combinatorics. We first recall the Hopf algebra construction and two particular families of dissection diagrams stressed by Dupont, the path trees' one and the corollas' one, because they generate Hopf subalgebras. We express their Hopf algebraic structures as group coordinate algebras. The path trees subalgebra and the Hopf algebra of symmetric functions are isomorphic. The corollas subalgebra is the dissection diagrams version of the Faà di Bruno Hopf algebra. Then, we discuss cofreedom. In the case of degree three, we give a basis of the primitive element vectors space but, their dimensions don't allow to conclude. Computations, made in joint work with Jean Fromentin, attest that it is not cofree with the parameter $x=-1$. We conclude this first section with computations of the antipode.

The second part consists in studying $\Hd$ by using rooted trees Hopf algebras. The choice consisting in sending a dissection diagram to its underlying rooted tree is not efficient, because it doesn't respect the coalgebra structure. Dupont \cite[remark 2.1.15]{Dupont2014} makes an allusion to a pre-Lie structure on dissection diagrams. The Hopf algebra $\dual{\Hd}$ is graded, connected, cocommutative, so is isomorphic to the enveloping algebra $\Ug{\g{\D}}$ of its primitive elements $\g{\D}$. According to Loday and Ronco, $\Hd$ is a free-commutative right-sided combinatorial Hopf algebra. This implies that there exists a pre-Lie structure on $\g{\D}$. In that way, we can use the  Oudom and Guin structural theorem \cite[proposition 2.7 - theorem 2.12]{Oudom2008} about enveloping algebras of pre-Lie algebras. The goal is then to study the pre-Lie algebra with parameter of dissection diagrams. Indeed, let $x$ be a parameter. If  the pre-Lie algebra of dissection diagrams associated to $x$ is free then its enveloping algebra is free too and, by duality, the dissection Hopf algebra of dissection diagrams $\Hd$ associated to $x$ is cofree. After a brief recall about pre-Lie algebras and the Grossman and Larson Hopf algebra, we present the Hopf algebra  $\dual{\Hd}$ and the pre-Lie structure on $\g{\D}$. We describe the unique morphism of Hopf algebras $\varphi$ respecting Oudom and Guin structures of  $\Hgl$ and $\dual{\Hd}$ and sending the rooted tree $t=\tun$ of degree $1$  to  \ud{0.5} $D=\diagramme{1}{}{{1/0}}{}{}{}{}{}$. This morphism relies on a inserting process of chords (propositions \ref{insertionuna}, \ref{valLundiagramme}, \ref{insertionunb} and \ref{valLdeuxdiagrammes}). \ud{0.7} We prove that the pre-Lie algebra generated by \ud{0.5} $\diagramme{1}{}{{1/0}}{}{}{}{}{}$ \ud{0.7} is not free (corollary \ref{absliberte}) and it is a strict sub-pre-Lie algebra of  $\g{\D}$ (proposition \ref{engendrement}).  Then the pre-Lie algebra of dissection diagrams is not free and it doesn't answer the question of cofreedom. We conjecture the kernel of $\varphi$ is the Hopf biideal generated by rooted trees with at least a vertex of degree at least three (conjecture \ref{conjnoyau}). Propositions \ref{etapeun}, \ref{etapedeux} and \ref{triangulaire} are first steps to solve this conjecture.

This article is a shortened version of chapter 3 of my thesis \cite[chapitre 3]{Mammezb}. 

\section{Hopf algebra of dissection diagrams}
\subsection{Recalls}
Let $\K$ be a commutative base field of characteristic 0. We recall some notations and the construction of the Hopf algebra of dissection diagrams. For all non negative integer $n$, we consider a regular oriented $(n+1)$-gone $\Pi_{n}$ with a special vertex called root. We draw $\Pi_{n}$ as a circle and put the root at the bottom. An arc between two vertices is a side and a line between two different vertices is a chord.

\begin{defi}
	A dissection diagram of degree $n$ is a set of $n$ non-intersecting chords of $\Pi_{n}$ such that the graph formed by the chords is acyclic. So the $n$ chords form a planar rooted tree whose root is the root of $\Pi_{n}$. For all dissection diagram $D$, we denote by $\Ch(D)$ the set of its chords. 
\end{defi}
\Ex{1} We consider the dissection diagram
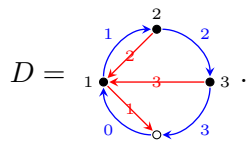
\[D=\exempleintro.\]
Its sides are colored in blue and its chords are colored in red.

\begin{lemma}
	The number of dissection diagrams of degree $n\in\N$ is given by 
	\[d_{n}=\frac{1}{2n+1}\binom{3n}{n}\]
	and satisfied the following reccursive relation
	\[\forall n\geq 1\text{, } d_{n}=\sum_{\substack{i_{1},i_{2},i_{3}\geq 0 \\ i_{1}+i_{2}+i_{3}=n-1}}d_{i_{1}}d_{i_{2}}d_{i_{3}}.\] 
\end{lemma}

\begin{proof}
	Let recall ideas of the proof explained by Dupont \cite[lemme 2.1.1]{Dupont2014}. For all dissection diagram $D$, there exists an unique triple of dissection diagrams  $(D_{1},D_{2},D_{3})$ and an unique triple of integers $(i_{1},i_{2},i_{3})$ such that   $i_{1}+i_{2}+i_{3}=n-1$, for all $j\in\{1,2,3\}$, the dissection diagram $D_{j}$ is of degree $i_{j}$ and $D$ is the dissection diagram 
	\[\defDiag{{6/0}}{{2/6/D_{1},10/6/D_{2},9/0/D_{3}}}[-].\]
	The black vertex is the first vertex, in the clockwise orientation, connected with the root. So, if 
	$d(h)=\sum\limits_{n\geq 0}d_{n}h^{n}$, then  \[d(h)=1+hd(h)^{3}.\]
	By the Lagrange inversion formula, for all integer $n\in\N$, $d_{n}$ is given by \[\displaystyle d_{n}=\frac{1}{n}\langle(1+h)^{3n},h^{n-1}\rangle=\frac{1}{2n+1}\binom{3n}{n}\] where $\langle(1+h)^{3n},h^{n-1}\rangle$ is the coefficient of $h^{n-1}$ in  $(1+h)^{3n}$. 
\end{proof}

Now we denote by $\D$ the vector space spanned by dissection diagrams. The formal series $d$ (\cite[séquence A001764 ]{Sloane}) recalled in the proceding proof is the \PH{} series of $\D$.

As a dissection diagram $D$ is clockwise oriented, it is possible to label the sides of $\Pi_{n}$ and the chords of $D$. If it is necessary, we label the root by 0.   
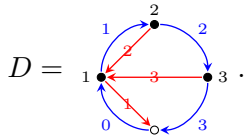
\Ex{1} 
\[D=\exempleintro.\]

So, it is possible to do the following identification 
\[\Ch(D)\simeq\{1,\dots,n\}\simeq\Sp^{+} \text{ where } \Sp^{+}=\{\text{sides of } \Pi_{n}\}\setminus\{\text{side } 0\}.\]

As it is natural, sometimes we won't write any label. For instance, the dissection diagram $D=\diagramme{3}{n}{{1/0,2/1,3/2}}{}{}{}{}{}$ becomes $D=\diagramme{3}{}{{1/0,2/1,3/2}}{}{}{}{}{}$.

Let $D$ be a dissection diagram and $C$ be a subset of $\Ch(D)$. We assume that the cardinality of $C$ equals $p$. Chords in $C$ give a partition of $\Pi_{\deg(D)}$ in $p+1$ faces. For each face $\alpha$, $\Sp_{C}(\alpha)$ is the set of sides of $\Pi_{\deg(D)}$ which are in the face $\alpha$. We shall consider the set $\Sp_{C}^{+}(D)=\sqcup_{\alpha}\Sp_{C}^{+}(\alpha)$ where $\Sp_{C}^{+}(\alpha)=\Sp_{C}(\alpha)\setminus\{\min(\Sp_{C}(\alpha))\}$.

\begin{prop}
	We define two maps over $\Hd$, which is, as vector space, the symmetric algebra generated by $\D$ (\emph{i.e.} as vector space $\Hd=\sym{\D}$), a product $m$ which is the disjoint union, and a coproduct $\De$ given by a selection-quotient process:
	\[m:\left\{
	\begin{array}{rcl}
	\Hd\otimes\Hd &\longrightarrow & \Hd\\
	D_{1}\ot D_{2} &\longrightarrow & D_{1}D_{2}
	\end{array}
	\right.\]
	and
	\[\De:\left\{
	\begin{array}{rcl}
	\Hd &\longrightarrow & \Hd\ot\Hd\\
	D &\longrightarrow & \sum\limits_{C\subset\Ch(D)}x^{k_{C}(D)}q_{C}(D)\ot r_{C}(D)
	\end{array}
	\right.\]
	where
	\begin{enumerate}
		\item $x\in\K$ is scalar,
		\item $q_{C}(D)$ is the disjoint union obtained by contracting chords in $C$,
		\item $r_{C}(D)$ is the dissection diagram obtained by keeping chords in $C$ and contracting sides of $\Pi_{\deg(D)}$ which are in the set $\Sp_{C}^{+}(D)$,
		\item $k_{C}(D)$ is the number of changes of direction we need to point properly chords of $r_{C}(D)$.
	\end{enumerate}
	We call $1_{\Hd}$ the unit of $\Hd$ algebra  and  $\varepsilon$ its counit. The space $(\Hd,m,1_{\Hd},\De,\varepsilon)$ is a Hopf algebra.
\end{prop}

\begin{proof}
Dupont proves it in \cite[proposition 2.1.11]{Dupont2014}.
\end{proof}

\Ex{1}\ud{0.5} As example, we use the dissection diagram $D=\diagramme{3}{n}{{1/0,2/0,3/2}}{}{}{}{}{}$.
\begin{align*}
	\De\po\diagramme{3}{n}{{1/0,2/0,3/2}}{}{}{}{}{}\pf=&\diagramme{3}{}{{1/0,2/0,3/2}}{}{}{}{}{}\ot 1 + \po x\diagramme{2}{}{{1/0,2/0}}{}{}{}{}{}+\diagramme{2}{}{{1/0,2/1}}{}{}{}{}{}+\diagramme{1}{}{{1/0}}{}{}{}{}{}\diagramme{1}{}{{1/0}}{}{}{}{}{}\pf\ot\diagramme{1}{}{{1/0}}{}{}{}{}{}\\
	+&\diagramme{1}{}{{1/0}}{}{}{}{}{}\ot\po\diagramme{2}{}{{1/0,2/1}}{}{}{}{}{}+(1+x)\diagramme{2}{}{{1/0,2/0}}{}{}{}{}{}\pf+1\ot\diagramme{3}{}{{1/0,2/0,3/2}}{}{}{}{}{}.
\end{align*}
\ud{0.7}

\subsection{Path trees and corollas: two Hopf subalgebras of $\Hd$}\label{echellescorolles}

Dupont \cite[examples 2.1.14]{Dupont2014} notices two special families: path trees and corollas. In fact, they form two Hopf subalgebras which are isomorphic to group coordinate Hopf algebras.  We recall the two families and explain the isomorphisms.

\subsubsection{Path trees and symmetric functions}
Let $n$ be a positive integer. The path tree of degree $n$ is the dissection diagramm $Y_{n}$ of degree $n$ such that for all  $i\in\llbracket 1,n-1\rrbracket$, the chord $i$ comes from the vertex $i$ to the vertex $i+1$ and the chord $n$ connects the vertex $n$ and the root. The path tree of degree 0 is the empty dissection diagram  \emph{i.e.} $Y_{0}=1$.
\Ex{} Path trees of degree 1 to 4:
\begin{align*}
Y_{1}=&\diagramme{1}{n}{{1/0}}{}{}{}{}{}, & Y_{3}=&\diagramme{3}{n}{{1/2,2/3,3/0}}{}{}{}{}{},\\
Y_{2}=&\diagramme{2}{n}{{1/2,2/0}}{}{}{}{}{}, & Y_{4}=&\diagramme{4}{n}{{1/2,2/3,3/4,4/0}}{}{}{}{}{}.
\end{align*}
For all non negative integer $n$ we have \[\De(Y_{n})=\sum_{k=0}^{n}\binom{n}{k}Y_{k}\ot Y_{n-k}.\]
Trivially, the vectorial space of disjoint union of path trees is a Hopf subalgebra of $\Hd$. We denote it by $\mathcal{E}_{Y}$. Dupont \cite[examples 2.1.14, 1]{Dupont2014} notices that the  $\mathcal{E}_{Y}$ coproduct and the symmetric functions one are reminiscent. We give an isomorphism between both.

Let $\displaystyle G_{0}=\langle1+\sum_{n=1}^{\infty}q_{n}h^{n}\in\K[[h]]\rangle$ be the multiplicative group of formal series with constant equals 1. For all positive integer $n$ we call $n$-th coordinate map the following map $\Sigma_{n}$ defined by $\Sigma_{n}:\left\{\begin{array}{rcl}\K[[h]] &\longrightarrow &\K \\\displaystyle Q=1+\sum_{n=1}^{\infty}q_{n}h^{n} &\longrightarrow & q_{n}.\end{array}\right.$ We consider the Hopf algebra $\Sym=\K[\Sigma_{1},\dots,\Sigma_{n},\dots]$ with its usual product and the coproduct given by: for all map $f\in\Sym$, for all elements $P$ and $Q$ of $G_{0}$, $\De_{G_{0}}(f)(P\ot Q)=f(PQ)$.

\begin{prop}
	The Hopf algebras $\mathcal{E}_{Y}$ and $\Sym$ are isomorphic.
\end{prop}
\begin{proof}
	It is sufficient to consider the following Hopf algebra morphism
	\[\omega_{1}:\left\{\begin{array}{rcl}\Sym &\longrightarrow &\mathcal{E}_{Y} \\ \Sigma_{n} &\longrightarrow & \frac{1}{n!}Y_{n}.\end{array}\right.\]
\end{proof}

\subsubsection{Corollas and the Faà di Bruno Hopf algebra}
The second prominent family noticed by Dupont is the corollas' one. For all positive integer $n$, the corolla of degree $n$ is the dissection diagram such that for all $i\in\llbracket 1,n\rrbracket$ the chord $i$ comes from vertex $i$ to the root. The corolla of degree 0 is the empty dissection diagram \emph{i.e.} $X_{0}=1$.

\Ex{} Non-empty corollas of degree $n$ less than or equal to 4:
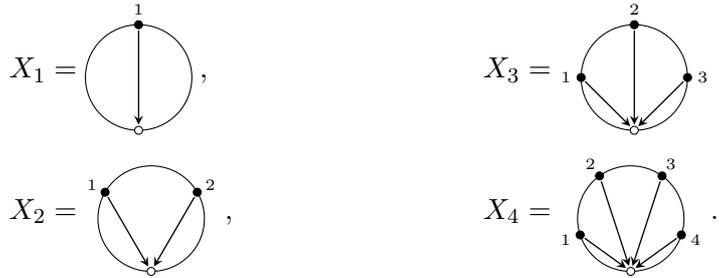
\begin{align*}
X_{1}=&\diagramme{1}{n}{{1/0}}{}{}{}{}{}, & X_{3}=&\diagramme{3}{n}{{1/0,2/0,3/0}}{}{}{}{}{},\\
X_{2}=&\diagramme{2}{n}{{1/0,2/0}}{}{}{}{}{}, & X_{4}=&\diagramme{4}{n}{{1/0,2/0,3/0,4/0}}{}{}{}{}{}.
\end{align*}
For all non negative integer $n$, we have:
\[\De(X_{n})=\sum_{k=0}^{n}\bigg(\sum_{\substack{i_{0}+\dots+i_{k}=n-k \\ i_{j}\geq 0}}X_{i_{0}}\dots X_{i_{k}}\bigg)\ot X_{k}.\]
Trivially, the vector space spanned by disjoint union of corollas is a Hopf subalgebra of $\Hd$ that we denote by $\mathcal{C}_{X}$.

Let us recall now the Faà di Bruno Hopf algebra construction. Let us consider the set $\displaystyle G_{1}=\langle h+\sum_{n=1}^{\infty}q_{n}h^{n+1}\in\K[[h]]\rangle$. Provided with the natural formal series composition $\circ$, it is the group of formal diffeomorphisms. For all positive integer $n$ we still call $n$-th coordinate map the linear map $\Sigma_{n}$ defined by $\Sigma_{n}:\left\{\begin{array}{rcl}\K[[h]] &\longrightarrow &\K \\\displaystyle Q=h+\sum_{n=1}^{\infty}q_{n}h^{n+1} &\longrightarrow & q_{n}. \end{array}\right.$ We consider the Hopf algebra $\HFdB=\K[\Sigma_{1},\dots,\Sigma_{n},\dots]$ with its usual product and the coproduct given by: for all map $f\in\HFdB$, for all elements $P$ and $Q$ of $G_{1}$,  $\De_{G_{1}}(f)(P\ot Q)=f(Q\circ P)$. It is the Faà di Bruno Hopf algebra.
\begin{prop}
	The Hopf algebras $\mathcal{C}_{X}$ and $\HFdB$ are isomorphic.
\end{prop}
\begin{proof}
	It is sufficient to consider the morphism
	\[\omega_{2}:\left\{\begin{array}{rcl}\HFdB &\longrightarrow & \mathcal{C}_{X} \\ \Sigma_{n} &\longrightarrow &X_{n}.\end{array}\right.\]
\end{proof}

\subsection{Primitive elements of degree less than or equal to 3}
\ud{0.5}
We want to study the $\Hd$ coalgebra to determine if it is cofree or not. We start by giving for all parameter $x\in\K$ a basis of the vector space $\Prim{\Hd}{2}$ (respectively $\Prim{\Hd}{3}$) spanned by degree 2 (respectively degree 3) primitive elements. Unfortunately, those two cases only don't allow to answer the question.

Let $F_{\Hd}$ the \PH{} formal series of $\Hd$. By definition, $F_{\Hd}(h)=\displaystyle\prod_{n=1}^{\infty}\frac{1}{(1-h^{n})^{d_{n}}}$, where $d_{n}$ is the number of dissection diagrams of degree $n$ for all non negative integer. If $\Hd$ is cofree, we have $\displaystyle F_{\Prim{\Hd}{}}=1-\frac{1}{F_{\Hd}}$ \emph{i.e.}  \begin{equation}\label{serieprimcolib}
F_{\Prim{\Hd}{}}(h)=h+3h^{2}+9h^{3}+40h^{4}+185h^{5}+\dots
\end{equation}

A basis of $\Prim{\Hd}{2}$ is given by the linearly independant vectors triple $(V_{1},V_{2},V_{3})$ where \begin{align*}
V_{1}=&(1+x)\diagramme{2}{n}{{1/2,2/0}}{}{}{}{}{}-2\diagramme{2}{n}{{1/0,2/1}}{}{}{}{}{},\\
V_{2}=&\diagramme{2}{n}{{1/2,2/0}}{}{}{}{}{}-\diagramme{2}{n}{{1/0,2/0}}{}{}{}{}{},\\
V_{3}=&\diagramme{2}{n}{{1/2,2/0}}{}{}{}{}{}-\diagramme{1}{n}{{1/0}}{}{}{}{}{}\diagramme{1}{n}{{1/0}}{}{}{}{}{}.
\end{align*}

For the third degree case, we add the following notations. 
\begin{defi}
		We call $\D^{+}$ the vector space spanned by non-empty dissection diagrams.  
\end{defi}
\begin{defi}
	 Let $U$ be in $\Hd$.
	\begin{enumerate}
		\item We denote by $l(U)$ its projection over $\D^{+}$, by $q(U)$ its projection over $(\D^{+})^{2}$, $t(U)$ its projection over $(\D^{+})^{3}$ and by $r(U)$ the sum $U-l(U)-q(U)-t(U)$.
		\item We call linear part of $\De(U)$ the projection of $\De(U)$ over $\D^{+}\ot\D^{+}$ denoted by $\delta(U)$. 
		\item The opposite linear part of $\De(U)$ is the linear part of $\De^{op}(U)$ denoted by $\delta^{op}(U)$.
		\item We call quadratic part of  $\De(U)$ the projection of $\De(U)$ over $(\D^{+})^{2}\ot\D^{+}$ denoted by $\Qd(U)$.
		\item The opposite quadratic part of $\De(U)$ is the quadratic part $\De^{op}(U)$ denoted by $\Qd^{op}(U)$.		
	\end{enumerate}
\end{defi}

\begin{lemma}
	Let $p$ be a primitive element of degree 3. We can write it as $p=l(p)+q(p)+t(p)$ and we have: 
	\begin{align*}
	m\circ\delta(l(p))=& -2q(p),\\
	\Qd(l(p))=&0,\\
	t(p)=&k\diagramme{1}{n}{{1/0}}{}{}{}{}{}\diagramme{1}{n}{{1/0}}{}{}{}{}{}\diagramme{1}{n}{{1/0}}{}{}{}{}{} \text{ où } k\in\K,\\
	m\circ\Qd(t(p))=& -3t(p).
	\end{align*}
\end{lemma}

\begin{proof}
	Let $p$ be a primitive element of degree 3. By commutativity, we have $m\circ\delta(l(p))=-2q(p)$. By definition of $t(p)$, as $p$ is of degree 3, there exists a scalar $k\in\K$ such that $t(p)=k\diagramme{1}{n}{{1/0}}{}{}{}{}{}\diagramme{1}{n}{{1/0}}{}{}{}{}{}\diagramme{1}{n}{{1/0}}{}{}{}{}{}$. 
	Similarly, there exists scalars $k_{1}$, $k_{2}$ and $k_{3}$ in  $\K$ such that $q(p)=k_{1}\diagramme{2}{n}{{1/2,2/0}}{}{}{}{}{}\diagramme{1}{n}{{1/0}}{}{}{}{}{}+k_{2}\diagramme{2}{n}{{1/0,2/0}}{}{}{}{}{}\diagramme{1}{n}{{1/0}}{}{}{}{}{}+k_{3}\diagramme{2}{n}{{1/0,2/1}}{}{}{}{}{}\diagramme{1}{n}{{1/0}}{}{}{}{}{}$.
	So $\Qd(t(p))=\Qd^{op}(t(p))$ and $\Qd(q(p))=\Qd^{op}(q(p))$. So $\Qd(l(p))=0$.
\end{proof}

\begin{lemma}
	Let $p$ be a primitive element of degree 3. The part $l(p)$ is a linear combination of the independant vectors $U_{1}$, \dots, $U_{9}$ where
	\begin{align*}
	U_{1}=&\diagramme{3}{}{{1/0,2/1,3/0}}{}{}{}{}{}-x\diagramme{3}{}{{1/2,2/0,3/0}}{}{}{}{}{}+x\diagramme{3}{}{{1/0,2/0,3/2}}{}{}{}{}{},  
	&U_{2}=&\diagramme{3}{}{{1/0,2/1,3/2}}{}{}{}{}{}-x^{2}\diagramme{3}{}{{1/2,2/0,3/0}}{}{}{}{}{}+x\diagramme{3}{}{{1/0,2/1,3/1}}{}{}{}{}{},\\
	U_{3}=&\diagramme{3}{}{{1/0,2/0,3/2}}{}{}{}{}{}-\diagramme{3}{}{{1/0,2/0,3/0}}{}{}{}{}{},
	&U_{4}=&\diagramme{3}{}{{1/2,2/0,3/2}}{}{}{}{}{}-\diagramme{3}{}{{1/3,2/3,3/0}}{}{}{}{}{},\\
	U_{5}=&\diagramme{3}{}{{1/0,2/3,3/1}}{}{}{}{}{}-(1+x)\diagramme{3}{}{{1/0,2/0,3/0}}{}{}{}{}{}+\diagramme{3}{}{{1/2,2/0,3/2}}{}{}{}{}{},
	&U_{6}=&\diagramme{3}{}{{1/3,2/1,3/0}}{}{}{}{}{}-\diagramme{3}{}{{1/2,2/0,3/2}}{}{}{}{}{},\\
	U_{7}=&\diagramme{3}{}{{1/2,2/0,3/0}}{}{}{}{}{}-\diagramme{3}{}{{1/3,2/1,3/0}}{}{}{}{}{},
	&U_{8}=&\diagramme{3}{}{{1/0,2/3,3/0}}{}{}{}{}{},
	~U_{9}=\diagramme{3}{}{{1/2,2/3,3/0}}{}{}{}{}{}.
	\end{align*}
\end{lemma}

\begin{prop}
	A basis of  $\Prim{\Hd}{3}$ is given by the vector family  $\left\{V_{1},\dots,V_{9}\right\}$ where
	\begin{align*}
	V_{1}=&\diagramme{3}{}{{1/0,2/1,3/0}}{}{}{}{}{}-x\diagramme{3}{}{{1/2,2/0,3/0}}{}{}{}{}{}+x\diagramme{3}{}{{1/0,2/0,3/2}}{}{}{}{}{}-(1+x)\diagramme{2}{}{{1/0,2/1}}{}{}{}{}{}\diagramme{1}{}{{1/0}}{}{}{}{}{}-(1+x^{2})\diagramme{2}{}{{1/0,2/0}}{}{}{}{}{}\diagramme{1}{}{{1/0}}{}{}{}{}{}\\
	+&x\diagramme{2}{}{{1/2,2/0}}{}{}{}{}{}\diagramme{1}{}{{1/0}}{}{}{}{}{}+(x^{2}+1)\diagramme{1}{}{{1/0}}{}{}{}{}{}\diagramme{1}{}{{1/0}}{}{}{}{}{}\diagramme{1}{}{{1/0}}{}{}{}{}{},\\  
	V_{2}=&\diagramme{3}{}{{1/0,2/1,3/2}}{}{}{}{}{}-x^{2}\diagramme{3}{}{{1/2,2/0,3/0}}{}{}{}{}{}+x\diagramme{3}{}{{1/0,2/1,3/1}}{}{}{}{}{}-(1+x)^{2}\diagramme{2}{}{{1/0,2/1}}{}{}{}{}{}\diagramme{1}{}{{1/0}}{}{}{}{}{}+x(x-1)\diagramme{2}{}{{1/0,2/0}}{}{}{}{}{}\diagramme{1}{}{{1/0}}{}{}{}{}{}\\
	+&x^{2}\diagramme{2}{}{{1/2,2/0}}{}{}{}{}{}\diagramme{1}{}{{1/0}}{}{}{}{}{}+\frac{x^{3}-x^{2}+5x+1}{3}\diagramme{1}{}{{1/0}}{}{}{}{}{}\diagramme{1}{}{{1/0}}{}{}{}{}{}\diagramme{1}{}{{1/0}}{}{}{}{}{},\\  
	V_{3}=&\diagramme{3}{}{{1/0,2/0,3/2}}{}{}{}{}{}-\diagramme{3}{}{{1/0,2/0,3/0}}{}{}{}{}{}-\diagramme{2}{}{{1/0,2/1}}{}{}{}{}{}\diagramme{1}{}{{1/0}}{}{}{}{}{}-(x-2)\diagramme{2}{}{{1/0,2/0}}{}{}{}{}{}\diagramme{1}{}{{1/0}}{}{}{}{}{}+(x-1)\diagramme{1}{}{{1/0}}{}{}{}{}{}\diagramme{1}{}{{1/0}}{}{}{}{}{}\diagramme{1}{}{{1/0}}{}{}{}{}{},\\  
	V_{4}=&\diagramme{3}{}{{1/2,2/0,3/2}}{}{}{}{}{}-\diagramme{3}{}{{1/3,2/3,3/0}}{}{}{}{}{}-\diagramme{2}{}{{1/0,2/1}}{}{}{}{}{}\diagramme{1}{}{{1/0}}{}{}{}{}{}+\diagramme{2}{}{{1/0,2/0}}{}{}{}{}{}\diagramme{1}{}{{1/0}}{}{}{}{}{}-(x-1)\diagramme{2}{}{{1/2,2/0}}{}{}{}{}{}\diagramme{1}{}{{1/0}}{}{}{}{}{}\\
	+&(x-1)\diagramme{1}{}{{1/0}}{}{}{}{}{}\diagramme{1}{}{{1/0}}{}{}{}{}{}\diagramme{1}{}{{1/0}}{}{}{}{}{},\\  
	V_{5}=&\diagramme{3}{}{{1/0,2/3,3/1}}{}{}{}{}{}-(1+x)\diagramme{3}{}{{1/0,2/0,3/0}}{}{}{}{}{}+\diagramme{3}{}{{1/2,2/0,3/2}}{}{}{}{}{}-2\diagramme{2}{}{{1/0,2/1}}{}{}{}{}{}\diagramme{1}{}{{1/0}}{}{}{}{}{}+2(1+x)\diagramme{2}{}{{1/0,2/0}}{}{}{}{}{}\diagramme{1}{}{{1/0}}{}{}{}{}{}\\
	-&(1+x)\diagramme{2}{}{{1/2,2/0}}{}{}{}{}{}\diagramme{1}{}{{1/0}}{}{}{}{}{},\\  
	V_{6}=&\diagramme{3}{}{{1/3,2/1,3/0}}{}{}{}{}{}-\diagramme{3}{}{{1/2,2/0,3/2}}{}{}{}{}{},\\  
	V_{7}=&\diagramme{3}{}{{1/2,2/0,3/0}}{}{}{}{}{}-\diagramme{3}{}{{1/3,2/1,3/0}}{}{}{}{}{}+\diagramme{2}{}{{1/0,2/1}}{}{}{}{}{}\diagramme{1}{}{{1/0}}{}{}{}{}{}-\diagramme{2}{}{{1/0,2/0}}{}{}{}{}{}\diagramme{1}{}{{1/0}}{}{}{}{}{}+(x-1)\diagramme{2}{}{{1/2,2/0}}{}{}{}{}{}\diagramme{1}{}{{1/0}}{}{}{}{}{}\\
	-&(x-1)\diagramme{1}{}{{1/0}}{}{}{}{}{}\diagramme{1}{}{{1/0}}{}{}{}{}{}\diagramme{1}{}{{1/0}}{}{}{}{}{},\\  
	V_{8}=&\diagramme{3}{}{{1/0,2/3,3/0}}{}{}{}{}{}-2\diagramme{2}{}{{1/0,2/0}}{}{}{}{}{}\diagramme{1}{}{{1/0}}{}{}{}{}{}-\diagramme{2}{}{{1/2,2/0}}{}{}{}{}{}\diagramme{1}{}{{1/0}}{}{}{}{}{}+2\diagramme{1}{}{{1/0}}{}{}{}{}{}\diagramme{1}{}{{1/0}}{}{}{}{}{}\diagramme{1}{}{{1/0}}{}{}{}{}{},\\
	V_{9}=&\diagramme{3}{}{{1/2,2/3,3/0}}{}{}{}{}{}-3\diagramme{2}{}{{1/2,2/0}}{}{}{}{}{}\diagramme{1}{}{{1/0}}{}{}{}{}{}+2\diagramme{1}{}{{1/0}}{}{}{}{}{}\diagramme{1}{}{{1/0}}{}{}{}{}{}\diagramme{1}{}{{1/0}}{}{}{}{}{}.
	\end{align*}
\end{prop}

The basis given before are compatible with formula (\ref{serieprimcolib}). In a joint work with Jean Fromentin, we compute the Hopf algebra $\Hd$ in {\verb!C++!}. With this computation it is possible to explain $\Hd$ is not cofree if $x=-1$. Indeed, in degree 5, if $x=-1$, the primitive space dimension equals  $187$ which is different from $185$ the dimension compatible with cofreedom. If $x\in\llbracket-10^4,10^4\rrbracket$ the dimensions of the primitive spaces of degree $d\in\llbracket1,4\rrbracket$ are compatible with cofreedom. In degree $5$, because of a long time of computation, we compute the dimension of the vector space for $x\in\llbracket-100,100\rrbracket$. Except for $x=-1$, dimensions founds all equal $185$.

\begin{conj}
	The Hopf algebra $\Hd$ is cofree, except if $x$ belongs to a countable subset of $\K$.
\end{conj} 

\subsection{The antipode computation}\label{renumcoproduititere}
We want to give a formula for antipode. We have to understand iterated coproduct \emph{ i.e.} morphisms $\De^{k}=(\De\otimes\underbrace{Id\ot\dots\ot Id}_{k-1 \text{ fois}})\circ\dots\circ(\De\ot Id)\circ\De$ where $k$ is a positive integer.
Let $D$ be a dissection diagram and let $C$ be a subset of $\Ch(D)$. By contracting $C$, we change labels in $q_{C}(D)$ and $r_{C}(D)$. So we have to define maps to change labels in a compatible way. As the coproduct is coassociative and we can contract $C$ without order it, it is sufficient to change labels in $q_{C}(D)$ where $C$ is a singleton.

So, let $D$ be a dissection diagram of degree $n$, let $a$ be a chord of $D$. The map we choose depends on the fact that  $q_{\{a\}}(D)$ is a dissection diagram or the disjoint union of two dissection diagrams. 
\begin{enumerate}
	\item We assume first $q_{\{a\}}(D)$ is a dissection diagram, so the chord $a$ connects the vertex $n$ and the root or there exists a non negative integer $i\in\llbracket0,n-1\rrbracket$  such that $a$ connects vertices $i$ and $i+1$. Let $s$ be the label of the chord $a$ given by its orientation. We define the map:
	\[N_{s,D} :\left\{\begin{array}{rcl} \llbracket1,n-1\rrbracket &\longrightarrow &\llbracket1,n\rrbracket \\ \alpha &\longrightarrow & \begin{cases}
	\alpha & \mbox{if } \alpha<s,\\
	\alpha+1 & \mbox{if } \alpha\geq s.
	\end{cases}. \end{array}\right.\]
	We decide to keep the same notation for $q_{\{a\}}(D)$ and for the dissection diagram $q_{\{a\}}(D)$ with new labels given by $N_{s,D}$.
	
	\item We assume now $q_{\{a\}}(D)$ is the disjoint union of two dissection diagrams  $D_{1}$ (whose side 0 is that of $D$) and 
	 $D_{2}$ with respective positive degree  $n_{1}$ and $n_{2}$. So we have two possibilities: either there exists a positive integer $j$ in $\llbracket2,n-1\rrbracket$ such that $a$ connects the vertex $j$ with the root, either there exists a positive integer $i$ in $\llbracket 1,n-2\rrbracket$ and another positive integer $u$ in $\llbracket 2,n-i\rrbracket$ such that $a$ connects vertices $i$ and $i+u$. Then there exists three cases: 
	 \begin{enumerate}
		\item The chord $a$ makes vertices $i$ and $i+u$ neighbours and $a$ is labeled by $i$ (so $i$ is positive). We have $n_{1}=n-u$ and $n_{2}=u-1$. We define the two following maps:
		\[N_{i,i+u,D_{1}} :\left\{\begin{array}{rcl} \llbracket1,n-u\rrbracket &\longrightarrow &\llbracket1,n\rrbracket \\ \alpha &\longrightarrow & \begin{cases}
		\alpha & \mbox{if } \alpha<i,\\
		\alpha+u & \mbox{if } \alpha\geq i
		\end{cases}. \end{array}\right.\]
		and
		\[N_{i,i+u,D_{2}} :\left\{\begin{array}{rcl} \llbracket1,u-1\rrbracket &\longrightarrow &\llbracket1,n\rrbracket \\ \alpha &\longrightarrow & \alpha+i. \end{array}\right.\]	
		
		\item The chord $a$ makes vertices $i$ and $i+u$ neighbours, $a$ is labelled by $i+u$ and $i$ is different from the root. We have $n_{1}=n-u$ and $n_{2}=u-1$. We define the two following maps:
		\[N_{i+u,i,D_{1}} :\left\{\begin{array}{rcl} \llbracket1,n-u\rrbracket &\longrightarrow &\llbracket1,n\rrbracket \\ \alpha &\longrightarrow & \begin{cases}
		\alpha & \mbox{if } \alpha\leq i,\\
		\alpha+u & \mbox{if } \alpha\geq i+1
		\end{cases}. \end{array}\right.\]
		and 
		\[N_{i+u,i,D_{2}} :\left\{\begin{array}{rcl} \llbracket1,u-1\rrbracket &\longrightarrow &\llbracket1,n\rrbracket \\ \alpha 
		&\longrightarrow & \alpha+i. \end{array}\right.\]
		
		\item The chord $a$ makes vertex $j$ and the root neighbours. As a consequence, $a$ is labelled  by $j$. We have $n_{1}=j-1$ and $n_{2}=n-j$. We define the two following maps:
		\[N_{0,j,D_{1}} :\left\{\begin{array}{rcl} \llbracket1,j-1\rrbracket &\longrightarrow &\llbracket1,n\rrbracket \\ \alpha &\longrightarrow & \alpha \end{array}\right.\]
		and 
		\[N_{0,j,D_{2}} :\left\{\begin{array}{rcl} \llbracket1,n-j\rrbracket &\longrightarrow &\llbracket1,n\rrbracket \\ \alpha &\longrightarrow & \alpha+u. \end{array}\right.\]
	\end{enumerate}
		We decide to keep the same notation for $q_{\{a\}}(D)$ and for the dissection diagram $q_{\{a\}}(D)$ with new labels given by	
	 $N_{i,i+u,D_{1}}$ and $N_{i,i+u,D_{2}}$ or, $N_{i+u,i,D_{1}}$ and $N_{i+u,i,D_{2}}$ or, $N_{0,j,D_{1}}$ and $N_{0,j,D_{2}}$ according to the label of the chord $a$.
\end{enumerate}

We consider now a dissection diagram $D$ of positive degree $n$ and $C=\{i_{1},\dots,i_{p}\}$ a nonempty subset of $\Ch(D)=\{1,\dots,n\}$. Contracting all chords in $C$ gives the same result as contracting in $D$ successively each chord. So, we keep the same notation for $q_{C}(D)$ and for the element  $q_{C}(D)$ with new labels, successively given by previous maps.

\begin{prop}
	Let $n$ be a positive integer and $D$ be a dissection diagram of degree $n$. The value of the antipode $S$ computed for $D$ is given by:  
	\[S(D)=\sum_{s=1}^{n}(-1)^{s}\sum_{P\in\Pi_{D}(s)}x^{k_{P}(D)}\prod_{i=1}^{s}r_{C_{i}}(q_{P_{i-1}}(D))\]
	where  
	\begin{itemize}
		\item[] $\Pi_{D}(s)$ is the set of $s$-tuples $(C_{1},\dots,C_{s})$ of no-empty sets which are a partition of $\Ch(D)$,
		\item[] $P_{0}=\emptyset$ for each $P$ in $\Pi_{D}(s)$,
		\item[] $P_{i}=\bigcup\limits_{u=1}^{i} C_{u}$  for each $P$ in $\Pi_{D}(s)$ and each integer $i$ in $\llbracket 1,s\rrbracket$,
		\item[] $k_{P}(D)=k_{C_{1}}(q_{P_{0}}(D))+\dots+k_{C_{s}}(q_{P_{s-1}}(D))$.
	\end{itemize} 
	
	If, for each non negative integer $p$, we denote by $X_{p}$ (respectively by $Y_{p}$) the corolla (respectively the path tree) of degree $n$, we have in particular:
	\begin{align*}
		S(X_{n})=&\sum_{k=1}^{n}(-1)^{k}\sum_{(\alpha_{1},\dots,\alpha_{k})\models n}\sum_{\substack{i_{j,0}+\dots+i_{j,p_{j}}=\alpha_{j} \\ p_{j}=\alpha_{j+1}+\dots+\alpha_{k} \\ j\in\{1,\dots,k-1\} \\ \forall m,~i_{u,m}\geq 0}} X_{\alpha_{k}}\prod_{j=1}^{k-1}X_{i_{j,0}}\dots X_{i_{j,p_{j}}},\\
		S(Y_{n})=&\sum_{k=1}^{n}(-1)^{k}\sum_{\alpha=(\alpha_{1},\dots,\alpha_{k})\models n}\frac{n!}{\alpha_{1}!\dots\alpha_{k}!}Y_{\alpha_{1}}\dots Y_{\alpha_{k}}\\
		=&\sum_{k=1}^{n}(-1)^{k}\sum_{\alpha=(\alpha_{1},\dots,\alpha_{k})\vdash n}\frac{n!}{\alpha_{1}!\dots\alpha_{k}!}\frac{k!}{u_{1}!\dots u_{s}!}Y_{\alpha_{u_{1}}}^{u_{1}}\dots Y_{\alpha_{u_{s}}}^{u_{s}},	
	\end{align*}
	where, for each $k$-tuple $(\alpha_{1},\dots,\alpha_{k})$, partition of $n$
	\begin{align*}
		s=&Card\{\alpha_{1},\dots,\alpha_{k}\},\\
		m_{1}=&\max\big\{i\in\llbracket 1,k\rrbracket, \forall j\leq i,\alpha_{j}=\alpha_{1}\big\},\\
		m_{l+1}=&\max\big\{i\in\llbracket 1,k\rrbracket, \forall m_{l}+1\leq j\leq i,\alpha_{j}=\alpha_{(u_{l}+1)}\big\} \text{ for } 
		0\leq l\leq s-1,\\
		u_{1}=&m_{1} \text{ and } u_{l+1}=m_{l+1}-m_{l}.
	\end{align*}
\end{prop}

\begin{proof}
	Let $n$ be a positive integer and $D$ be a dissection diagram of degree $n$. For all integer $s$ in $\llbracket1,n-1\rrbracket$ we define  \[\tilde{\Pi}_{D}(s)=\{(C_{1},\dots,C_{s+1})\in\Ch(D)^{s+1},~(C_{1},\dots,C_{s})\in\Pi_{D}(s)\text{ or } (C_{1},\dots,C_{s+1})\in\Pi_{D}(s+1)\}.\] We use the fact that the antipode reverse the identity map for the convolution product and on each step we use the  dialing process described before. Then we have:
	\begin{align*}
	S(D)=&-\sum_{\substack{C_{1}\subset\Ch(D) \\ C_{1}\neq \emptyset}}x^{k_{C_{1}}(D)}S(q_{C_{1}}(D))r_{C_{1}(D)}\\
	=&\sum_{\substack{C_{1}\subset\Ch(D) \\ C_{1}\neq \emptyset}}x^{k_{C_{1}}(D)}\sum_{\substack{C_{2}\subset\Ch(D)\setminus C_{1} \\ C_{2}\neq \emptyset}}x^{k_{C_{2}}(q_{C_{1}}(D))}S(q_{C_{1}\cup C_{2}}(D))r_{C_{2}}(q_{C_{1}}(D))r_{C_{1}(D)}\\
	=&(-1)^{u}\sum_{P\in\tilde{\Pi}_{D}(u)}x^{k_{P}(D)}S(q_{P_{u}(D)})\prod_{i=1}^{u}r_{C_{i}}(q_{P_{i-1}}(D))\\
	=&\sum_{s=1}^{n}(-1)^{s}\sum_{P\in\Pi_{D}(s)}x^{k_{P}(D)}\prod_{i=1}^{s}r_{C_{i}}(q_{P_{i-1}}(D)).
	\end{align*}
\end{proof}

\Ex{}
\begin{align*}
	S\po\diagramme{3}{}{{1/0,2/3,3/1}}{}{}{}{}{}\pf=&-\diagramme{3}{}{{1/0,2/3,3/1}}{}{}{}{}{}+(1+x)\diagramme{1}{}{{1/0}}{}{}{}{}{}\diagramme{2}{}{{1/0,2/0}}{}{}{}{}{}+2\diagramme{1}{}{{1/0}}{}{}{}{}{}\diagramme{2}{}{{1/0,2/1}}{}{}{}{}{}+\diagramme{1}{}{{1/0}}{}{}{}{}{}\diagramme{2}{}{{1/2,2/0}}{}{}{}{}{}\\
	&-(2x+3)\diagramme{1}{}{{1/0}}{}{}{}{}{}\diagramme{1}{}{{1/0}}{}{}{}{}{}\diagramme{1}{}{{1/0}}{}{}{}{}{},\\
	S\po\diagramme{3}{}{{1/0,2/0,3/0}}{}{}{}{}{}\pf=&-\diagramme{3}{}{{1/0,2/0,3/0}}{}{}{}{}{}+5\diagramme{1}{}{{1/0}}{}{}{}{}{}\diagramme{2}{}{{1/0,2/0}}{}{}{}{}{}-5\diagramme{1}{}{{1/0}}{}{}{}{}{}\diagramme{1}{}{{1/0}}{}{}{}{}{}\diagramme{1}{}{{1/0}}{}{}{}{}{}.
\end{align*}

\Rq{1} Considering the Faà di Bruno Hopf algebra as the Hopf algebra of coordinate of the group of formal diffeomorphims $G_{1}$, for each positive integer $n$, the antipode $S$ computes in the coordinate map  $\Sigma_{n}$ is given by: \[\forall P\in G_{1},~S(\Sigma_{n})(P)=\Sigma_{n}(P^{-1})\]
where $P^{-1}$ is the reverse of $P$ for the formal series composition \cite[section 4.3]{Frabetti2015}.

\section{Dissection diagrams, rooted trees, theorem of Oudom and Guin}
The aim of this section is to study $\Hd$ by using rooted trees Hopf algebras. We can easily associate a rooted tree to any dissection diagram. It is sufficient to consider the underlying non planar rooted tree of its regular $(n+1)$-gone. It may seem better to consider the underlying planar rooted tree of its regular $(n+1)$-gone, but the non planar case allows to conclude. We have to consider a Hopf algebra of rooted trees graded by the number of edges and not the number of vertices. We can think about the Hopf algebra of rooted trees with the contraction coproduct introduced in \cite{Calaque2011}. Unfortunately, associating a dissection diagram with its underlying rooted tree is not a good idea, because this kind of morphisms doesn't respect the coalgebra structure. %The given counter-example is a counter-example too for the planar case.

It is wiser to be interested in the graded dual $\dual{\Hd}$ of the Hopf algebra of dissection diagrams. Dupont makes allusion to a pre-Lie structure over its primitive element space \cite[remark 2.1.15]{Dupont2014}. As $\dual{\Hd}$ and the enveloping algebra of its primitive elements $\g{\D}$ are isomorphic, it is possible to apply the theorem of  Oudom and Guin \cite[proposition 2.7 - théorème 2.12]{Oudom2008}. We build the unique pre-Lie morphism $\gamma$ from the free pre-Lie algebra $\g{\TR}$ generated by $t=\tun$ to $\g{\D}$ sending $t$ to the dissection diagram of degree 1. We can deduce that the pre-Lie algebra generated by the dissection diagram of degree 1 is not free and do not generate the whole pre-Lie algebra $\g{\D}$.We extand $\gamma$ as a morphism of Hopf algebras and give a conjecture about its kernel. 

\subsection{pre-Lie algebras and theorem of Oudom and Guin}
Pre-Lie algebras (also called Vinberg algebras) were introduced in 1963 by Vinberg \cite{Vinberg1963} in the theory of homogeneous convex cones and by Gerstenhaber \cite[section 2]{Gerstenhaber1963} in deformation theory. The $\PreLie$ operad defining pre-Lie algebras was introduced and described by Chapoton and Livernet \cite[theorem 1.9]{Chapoton2001}. They describe the free pre-Lie algebra generated by one or several generators in terms of rooted trees too \cite[corollaire 1.10]{Chapoton2001}. It is another way to prove the isomorphism  between the dual of the Connes and Kreimer Hopf algebra \cite{Connes1999} and the Grossman and Larson Hopf algebra \cite{Grossman1989}. Foissy proves that the free pre-Lie algebra with one generator is free as a Lie algebra \cite[theorem 8.4]{Foissy2002c}. Chapoton, thanks to  operad theory, proves this result for each free pre-Lie algebra	\cite[corollary 5.3]{Chapoton2010}. He proves that the operad $\PreLie$ is anticyclic \cite{Chapoton2005}. Livernet \cite[theorem 3.4]{Livernet2006} determines the freedom of \enquote{Hopf pre-Lie algebras} equipped the relation $\De(x\circ y)=\De(x)\circ y+ x\ot y =x^{(1)}\circ y\ot x^{(2)} + x^{(1)}\ot x^{(2)}\circ y + x\ot y$. Oudom and Guin \cite[proposition 2.7 - theorem 2.12]{Oudom2008}  build for each pre-Lie algebra $\g{}$ a Hopf algebra which is isomorphic to $\Ug{}$. It is a generalization of the construction of the enveloping algebra of the free pre-Lie algebra of rooted trees generated by $\tun$ (Grossman and Larson Hopf algebra). Loday and Ronco \cite[theorems 5.3 and 5.8]{Loday2010} explain that there exists a pre-Lie structure for free-commutative right-sided combinatorial Hopf algebras. It is possible to equip an operad with a pre-Lie structure.  Burgunder, Delcroix-Oger and Manchon \cite[theorem 3.1]{Burgunder170} attest that an operad cannot be free as pre-Lie algebra.
New structures can be defined too. Mansuy builds the quadratic operad $\Com-\PreLie$  and gives as example the rooted trees algebra equipped with the grafting product and the grafting product on the root \cite[section 4.2]{Mansuy2014}. Foissy \cite[definition 17 and theorem 20]{Foissy2015a} explains the free Com-PreLie algebra with one generator as the partitioned trees algebra equipped with the disjoint union product and with the grafting product. As another structure, we can notice the quadratic operad $\PostLie$ introduced by Vallette \cite[section A.2]{Vallette2007}. Une post-Lie algebra $\mathcal{A}$ is equipped with a binary map $\circ$ and with a Lie bracket $\{-,-\}$ which are compatible. If $(\mathcal{A},\{-,-\})$ is abelian then $(\mathcal{A},\circ)$ is a pre-Lie algebra. The post-Lie algebra notion is a generalization of the pre-Lie algebra notion. To fill up knowledge about origins and applications about pre-Lie algebra, there exists the survey of Burde \cite{Burde2006}.

\begin{defi}
	A left pre-Lie algebra is couple $(\g{},\circ)$ where $\g{}$ is a vector space and $\circ:	\g{}\otimes\g{} \longrightarrow \g{}$ is an internal product with the following relation: for all $x,y,z\in\g{}$,
	\[x\circ (y\circ z)-(x\circ y)\circ z=y\circ (x\circ z)-(y\circ x)\circ z.\]
\end{defi}
\Ex{1} We consider the classical example of $\g{}=\{P(X)\partial, P(X)\in\K[X]\}$, the derivation algebra of $\K[X]$, where $\partial$ is the derivation sending $X$ to $1$. We define the product $\circ$ by 
\[\circ:\left\{
\begin{array}{rcl}
\g{}\ot\g{} &\longrightarrow & \g{} \\
P(X)\partial\ot Q(X)\partial &\longrightarrow & (P(X)\partial Q(X))\partial.\end{array}\right.\]
The product of two derivations $P(X)\partial$ and $Q(X)\partial$ is not the usual composition of maps ($\g{}$ is not stable under this product) but it is the unique derivation sending $X$ to $P(X)Q'(X)$. Let $P(X)$, $Q(X)$ and $R(X)$ be polynomials. We have 
\[P(X)\partial\circ (Q(X)\partial\circ R(X)\partial)-(P(X)\partial\circ Q(X)\partial)\circ R(X)\partial=P(X)Q(X)\partial^{2}R(X)\partial.\]
This relation is symmetric over $P(X)$ and $Q(X)$ so the pre-Lie relation is satisfied.
\begin{prop}
	Let $(\g{,\circ})$ be a pre-Lie algebra. We define the bracket $\{-,-\}$ by:
	\[\{-,-\}:\left\{
	\begin{array}{rcl}
	\g{}\ot\g{} &\longrightarrow & \g{} \\
	x\ot y &\longrightarrow & x\circ y -y\circ x. 
	\end{array}\right.\]
	With this bracket, $\g{}$ is a Lie algebra denoted by $\g{Lie}$. 
\end{prop}

\begin{proof} Direct.
\end{proof}

\begin{defi}\label{defOG}
	Let $(\g{},\circ)$ a pre-Lie algebra. We consider the Hopf symmetric algebra $\sym{\g{}}$ equipped with its usual coproduct $\De$.  We extend the product $\circ$ to $\sym{\g{}}$. Let  $a$, $b$, $c$ and $x$ be elements such that $a,b,c\in\sym{\g{}}$ and $x\in\g{}$. We put: 
	\[\left\{
	\begin{array}{rcl}
	1\circ a&=&a,\\
	a\circ 1&=&\varepsilon(a)1,\\
	(xa)\circ b&=&x\circ(a\circ b)-(x\circ a)\circ b,\\
	a\circ(bc)&=&\sum\limits_{a}(a^{(1)}\circ b)(a^{(2)}\circ c).
	\end{array}\right.\]
	On $\mathcal{S}(\g{})$, we define a product $\star$ by:
	\[\star\left\{
	\begin{array}{rcl}
	\mathcal{S}(\g{})\ot\mathcal{S}(\g{}) &\longrightarrow & \mathcal{S}(\g{}) \\
	a\ot b &\longrightarrow &a\star b=\sum\limits_{a}a^{(1)}(a^{(2)}\circ b).
	\end{array}\right.\]
\end{defi}

\begin{theo}\label{structureOG}
	The space $(\sym{\g{}},\star,\De)$ is a Hopf algebra which is isomorphic to the enveloping Hopf algebra $\Ug{Lie}$ of the Lie algebra $\g{Lie}$ generated by primitive elements.
\end{theo}
\begin{proof}
	This theorem was proved by Oudom and Guin in \cite[proposition 2.7 - theorem 2.12]{Oudom2008}.
\end{proof}

\subsection{Hopf algebras of rooted trees}

\subsubsection{Hopf algebra $\Hgl$ of Grossman and Larson rooted trees}
The Grossman and Larson Hopf algebra, written  $\Hgl$, was introduced in \cite{Grossman1989} as a tool in the theory of differential operators \cite{Grossman1990,Grossman1992}. It is graded, connected, cocommutative and not commutative. By the \CQMM{}, it is isomorphic  to the enveloping algebra of its primitive elements. Panaite \cite{Panaite2000} proves that there exists a connection between $\Hgl$ and the graded dual $\dual{\Hc}$ of the Connes and Kreimer Hopf algebra \cite{Connes1999}. He is corrected by 
Hoffman \cite{Hoffman2003}. The two Hopf algebras are not equal but isomorphic in characteristic $0$. In $\dual{\Hc}$, because of grafts, there are symmetry coefficients which don't appear in $\Hgl$. Chapoton and Livernet \cite[corollary 1.10]{Chapoton2001} prove that the pre-Lie algebra of rooted trees in Grossman and Larson Hopf algebra is free and give another proof of the isomorphism between $\Hgl$ and $\dual{\Hc}$.  Oudom and Guin \cite[proposition 2.7 - theorem 2.12]{Oudom2008} use the $\Hgl$ case as model to prove a structural theorem for the enveloping algebra of a pre-Lie algebra.

We recall the definition of $\Hgl$ with the point of vue of Oudom and Guin. Let $\g{\TR}$ the vector space $\g{\TR}=Vect(t,~t\in\TR)$.  

\begin{defi} On $\g{\TR}$ we define the following product:
	\[\circ:\left\{\begin{array}{rcl}\g{\TR}\ot\g{\TR} &\longrightarrow & \g{\TR} \\ t_{1}\ot t_{2} &\longrightarrow & \sum\limits_{s\in V(t_{2})} \tau_{t_{1},t_{2},s} \end{array}\right.\]
	where $\tau_{t_{1},t_{2},s}$ is the rooted tree obtained by grafting $t_{1}$ on the vertex $s$ of $t_{2}$.
\end{defi}

\begin{theo}\label{prelielibreun}
	The algebra $(\g{\TR},\circ)$ is the free left pre-Lie algebra generated by the rooted tree of degree $1$ (\emph{i.e. $t=\tun$}).
\end{theo}

\begin{proof}
	It is corollary 1.10 proved by Chapoton and Livernet in \cite{Chapoton2001}.
\end{proof}

\Ex{}
\begin{align*}
\tdeux\circ \ttroisdeux=&\tcinqcinq+\tcinqonze+\tcinqquatorze,&
\ttroisdeux\circ \tdeux=&\tcinqhuit+\tcinqquatorze.
\end{align*}

\begin{prop}
	On $\sym{\g{\TR}}$, the product $\star$, built with the pre-Lie structure and the theorem of Oudom and Guin, is given by:
	\[\star:\left\{\begin{array}{rcl}\Hgl\ot \Hgl &\longrightarrow &\Hgl \\ t_{1}\dots t_{n}\ot t_{n+1}\dots t_{n+m} &\longrightarrow & \displaystyle\sum_{\sigma:I\subseteq\llbracket1,n\rrbracket\longrightarrow V(t_{n+1}\dots t_{n+m})}(t_{1}\dots t_{n},t_{n+1}\dots t_{n+m},\sigma) \end{array}\right.\]
	where 
	\begin{itemize}
		\item[] $V(t_{n+1}\dots t_{n+m})$ is the set of the vertices of the forest $t_{n+1}\dots t_{n+m}$,
		\item[] $(t_{1}\dots t_{n},t_{n+1}\dots t_{n+m},\sigma)$ is the rooted forest obtained by grafting the tree $t_{i}$ on the vertex  $\sigma(i)$ for all $i$ in $I$.
	\end{itemize} 
\end{prop}

\Ex{}
\begin{align*}
\ttroisun\star\tdeux=&\tcinqsix+\tcinqtreize+\ttroisun\tdeux,\\
\tun\tun\star\tdeux=&2\ttroisun\tun+2\tun\ttroisdeux+\tquatreun+\tquatrequatre+2\tquatretrois+\tun\tun\tdeux,\\
\tdeux\star\tun\tun=&2\tun\ttroisdeux+\tun\tun\tdeux.
\end{align*}

\subsubsection{Quotient Hopf algebra of sub-binary trees}
\begin{defi}
	\begin{enumerate}
		\item A rooted tree $t$ is a sub-binary tree if all its vertices have a fertility less than or equal to 2. The set of rooted sub-binary trees is denoted by $\sbT$.
		\item A rooted forest $F$ is sub-binary if all of its trees are sub-binary trees. We denote the set of rooted sub-binary forests by $\sbF$.
	\end{enumerate}
\end{defi}
\begin{prop}
	We consider the vector space $\Idl{I}=Vect( F, F\in\FRt\setminus\sbF)$. It is a Hopf biideal of $\Hgl$.
\end{prop}

\begin{proof} Direct checking.
\end{proof}

\begin{prop}
	Le vector space $\HsbT=Vect(F, F\in\sbF)$ is a quotient Hopf algebra of $\Hgl$.
\end{prop}
\begin{proof}
	It is sufficient to consider the canonical surjection
	\[\Pi_{\HsbT}:\left\{\begin{array}{rcl}\Hgl &\longrightarrow & \HsbT \\ F\in\FRt &\longrightarrow & \begin{cases}
	F &\mbox{ if } F\in\sbF, \\
	0 &\mbox{ else.}
	\end{cases} \end{array}\right.\]
	The biideal $\Idl{I}$ is the kernel of $\Pi_{\HsbT}$. 
\end{proof}

\Rq{1} By duality, we can consider $\dual{\HsbT}$ as a Hopf subalgebra of the Connes and Kreimer Hopf algebra.

\subsection{Dual of the dissection diagrams Hopf algebra}
As the Hopf algebra $\Hd$ is graded and connected, its graded dual $\dual{\Hd}$ is a graded connected Hopf algebra. We aim at studying its structure. Till corollary \ref{similitudetravailLoic}, we use the same approach as in \cite[chapter 7, section 7.3]{Foissy}.

\begin{prop}\label{algEnvp}
	The Hopf algebra $\dual{\Hd}$ is isomorphic to the enveloping algebra $\Ug{\D}$ where $\g{\D}$ is the Lie algebra $\Prim{\dual{\Hd}}{}$ of the primitive elements of $\dual{\Hd}$.
\end{prop}

\begin{proof}
	We apply the \CQMM.
\end{proof}

We denote by $(Z_{F_{D}})_{F_{D}\in\Hd}$ the dual basis of disjoint union of dissection diagrams. We know that generators of the Lie algebra $\Prim{\dual{\Hd}}{}$ are $\g{\D}=((1)+Ker(\varepsilon)^{2})^{\perp}$, so a basis of $\g{\D}$ is given by $(Z_{D})_{D\in\D}$. In order to describe operations of the Hopf algebra $\dual{\Hd}$, it is sufficient to define the product on $Z_{D}$, with $D$ a dissection diagram. Each dissection diagram $D$ is primitive. Concerning the product, we have the following proposition:

\begin{prop}
	Let $D_{1}$ and $D_{2}$ be two nonempty dissection diagrams of respective degree $n_{1}$ and $n_{2}$ and $x$ a scalar. the product of $Z_{D_{1}}$ and $Z_{D_{2}}$ in this order is given by:
	\[Z_{D_{1}}Z_{D_{2}}=(1+\delta_{D_{1},D_{2}})Z_{D_{1}D_{2}}+\sum_{D\in\D}c(D_{1},D_{2};D)Z_{D}\]
	where for all dissection diagram $D$, the coefficient $c(D_{1},D_{2};D)$ is a polynomial $Q_{D}$, depending on $D$ and evaluated in $x$. We have: 
	\[Q_{D}(x)=(Z_{D_{1}}\ot Z_{D_{2}})\circ \De(D).\] 
\end{prop}

\ud{0.5}
\Ex{1}
\begin{align*}
	Z_{\diagramme{1}{}{{1/0}}{}{}{}{}{}}Z_{\diagramme{1}{}{{1/0}}{}{}{}{}{}}=&2Z_{\diagramme{1}{}{{1/0}}{}{}{}{}{}\diagramme{1}{}{{1/0}}{}{}{}{}{}}+2Z_{\diagramme{2}{}{{1/0,2/0}}{}{}{}{}{}}+2Z_{\diagramme{2}{}{{1/2,2/0}}{}{}{}{}{}}+(1+x)Z_{\diagramme{2}{}{{1/0,2/1}}{}{}{}{}{}}.
\end{align*}
\ud{0.7}

\begin{cor}\label{similitudetravailLoic}
	The Lie algebra $\g{\D}$ is the vector space $\g{\D}=\Prim{\dual{\Hd}}{}$ equipped with the bracket $[-,-]$ defined by:
	for all dissection diagrams $D_{1}$ and $D_{2}$,
	\[[Z_{D_{1}},Z_{D_{2}}]=\sum\limits_{D\in\D}\bigg(c(D_{1},D_{2};D)-c(D_{2},D_{1};D)\bigg)Z_{D}.\]
\end{cor}

\ud{0.5}
\Ex{1}
\begin{align*}
	\Bigg[Z_{\diagramme{2}{}{{1/2,2/0}}{}{}{}{}{}},Z_{\diagramme{1}{}{{1/0}}{}{}{}{}{}}\Bigg]=&-Z_{\diagramme{3}{}{{1/2,2/0,3/2}}{}{}{}{}{}}-Z_{\diagramme{3}{}{{1/3,2/3,3/0}}{}{}{}{}{}}-Z_{\diagramme{3}{}{{1/2,2/0,3/0}}{}{}{}{}{}}-Z_{\diagramme{3}{}{{1/3,2/1,3/0}}{}{}{}{}{}}\\
	+&Z_{\diagramme{3}{}{{1/0,2/3,3/1}}{}{}{}{}{}}+xZ_{\diagramme{3}{}{{1/0,2/1,3/0}}{}{}{}{}{}}+x^{2}Z_{\diagramme{3}{}{{1/0,2/1,3/2}}{}{}{}{}{}}.
\end{align*}

Now we aim at defining a Hopf algebra isomorphic to $\dual{\Hd}$ by providing it with a structure of Oudom and Guin. We create then a morphism between $\Hgl$ and $\dual{\Hd}$ by using pre-Lie structures.

\begin{prop}
	The vector space $\g{\D}=\Prim{\g{\D}}{}$, equipped with $\circ$ defined by: 
	\[\circ:\left\{\begin{array}{rcl}\g{\D}\ot\g{\D} &\longrightarrow & \g{\D} \\ Z_{D_{1}}\ot Z_{D_{2}} &\longrightarrow & \sum\limits_{D\in\D}c(D_{1},D_{2};D)Z_{D}, \end{array}\right.\] is a left pre-Lie algebra.
\end{prop}

\begin{proof}
	Let us first recall the notion of free-commutative right-sided combinatorial Hopf algebra introduced by Loday and Ronco in \cite[definitions 3.16, 4.1, 4.16 and 5.7]{Loday2010}. A free-commutative right-sided combinatorial Hopf algebra $H$ is an associative commutative free Hopf algebra $(H,m,\De)$ generated by $\Gen{H}{}$ and such that, for $h\in \Gen{H}{}$, we have $\De(h)=\sum\limits_{h}h^{(1)}\ot h^{(2)}$ with $h^{(2)}\in\Gen{H}{}$. The Hopf algebra $\Hd$ is a free-commutative right-sided combinatorial Hopf algebra. According to Loday and Ronco \cite[theorems 5.3 and 5.8]{Loday2010}, the couple $(\g{\D},\circ)$ is a left pre-Lie algebra.
\end{proof}

\begin{prop}
	We consider the Hopf algebra $\mathcal{S}=(\sym{\g{\D}},\star,\De)$ where $\sym{\g{\D}}$ is the symmetric algebra of dissection diagrams, the coproduct $\De$ is the usual coproduct of $\sym{\g{\D}}$ and the product $\star$ is induced by the pre-Lie product $\circ$ defined on $\g{\D}$ and the definition \ref{defOG}. The Hopf algebras $\mathcal{S}$ and $\dual{\Hd}$ are isomorphic.
\end{prop}

\begin{proof}
	It is a direct application of theorem \ref{structureOG} and proposition \ref{algEnvp}. Now we identify $\dual{\Hd}$ with the Hopf algebra obtained by the theorem of Oudom and Guin.
\end{proof}

\begin{lemma}\label{engendrement}
	The pre-Lie generated by the dissection diagram $\diagramme{1}{}{{1/0}}{}{}{}{}{}$ is a non trivial sub-object of $\g{\D}$.
\end{lemma}
\begin{proof}
	The vector subspace of elements of degree $2$ underlying the pre-Lie algebra generated by $\diagramme{1}{}{{1/0}}{}{}{}{}{}$ is generated by 	\[\diagramme{1}{}{{1/0}}{}{}{}{}{}\circ\diagramme{1}{}{{1/0}}{}{}{}{}{}=2\diagramme{2}{}{{1/2,2/0}}{}{}{}{}{}+2\diagramme{2}{}{{1/0,2/0}}{}{}{}{}{}+(1+x)\diagramme{2}{}{{1/0,2/1}}{}{}{}{}{}.\] And yet, the vector space of elements of degree $2$ underlying $\g{\D}$ is genetated by the three dissection diagrams of degree $2$. So, the pre-Lie generated by the dissection diagram $\diagramme{1}{}{{1/0}}{}{}{}{}{}$ is a non trivial sub-object of $\g{\D}$.
\end{proof}

\section{Definition of a Hopf algebra morphism from $\Hgl$ to $\dual{\Hd}$}
We want to describe a Hopf algebra morphism between the  Grossman and Larson rooted forests Hopf algebra and the graded dual of the Hopf algebra of dissection diagrams. To do that we use the underlying pre-Lie structure of the two algebras, the Hopf algebra structure and a chords inserting process. To simplify the writing, for each disjoint union $U$, the element $Z_{U}$ of $\dual{\Hd}$ is denoted by $U$.

\Rq{1} If  $f:\Hgl\longrightarrow\dual{\Hd}$ is a graded Hopf morphism homogeneous of degree $k\geq0$ then it is not surjective. Indeed, for all positive integers different from $1$ there are less rooted forests of degree $n$ than dissection diagrams of degree $n$. Let $n$ bet a positive integer. We denote by $C_n$ the $n$-th Catalan number. There exist at most $C_{n}$ rooted forests of degree $n$. We denote by $d_{n}$ the number of dissection diagrams of degree $n$. Then, $C_1=d_1$ and, if $n$ is greater than $2$ then we have:
\[\frac{d_{n}}{C_{n}}=\frac{(n+1)\displaystyle\binom{3n}{n}}{(2n+1)\displaystyle\binom{2n}{n}}=\frac{(3n)!(n+1)!}{(2n+1)!(2n)!}=\frac{(2n+2)\dots(3n)}{(n+2)\dots(2n)}.\]

\begin{defi}
	We call $\gamma$ the following pre-Lie morphism:
	\[\gamma:\left\{\begin{array}{rcl}
	\g{\TR} &\longrightarrow & \g{\D}\\
	\tun &\longrightarrow &\diagramme{1}{}{{1/0}}{}{}{}{}{}
	\end{array}
	\right.\]
\end{defi}

Now we just have to extend this pre-Lie morphism $\gamma$ as a Hopf algebra morphism. Let us recall first the definition of the grafting operator.
\begin{defi}
	The grafting operator, denoted by $B$, is given by:
	\[B:\left\{
	\begin{array}{rcl}
	\K[\TR] &\longrightarrow & \g{\TR}\\
	t_{1},\dots,t_{n}\in\TR &\longrightarrow & \text{the rooted tree obtained by grafting $t_{1},\dots,t_{n}$ on a common root.} 
	\end{array}
	\right.\]
\end{defi}
\Ex{1} $B(t_{1},t_{2})=\tdtroisunb{}{$t_{2}$}{$t_{1}$}=\tdtroisunb{}{$t_{1}$}{$t_{2}$}=B(t_{2},t_{1})$.

\begin{prop}
	The unique extension $\varphi$ of $\gamma$, built with the Oudom and Guin structure of $\Hgl$, is defined by:
	\[\varphi:\left\{
	\begin{array}{rcl}
	\Hgl &\longrightarrow & \dual{\Hd}\\
	t_{1}\dots t_{n} &\longrightarrow & \varphi(t_{1})\dots\varphi(t_{n})\\
	t &\longrightarrow & L(\varphi(t_{1})\dots\varphi(t_{k}))\\
	\end{array}
	\right.\]
	where $t_{1}\dots t_{k}$ is the rooted forest such that $t=B(t_{1}\dots t_{k})$ and $L$ is the following linear map:
	\[L:\left\{
	\begin{array}{rcl}
	\Hd &\longrightarrow & \D\\
	D_{1}\dots D_{n} &\longrightarrow & \sum\limits_{\substack{G\in\D \\ G \text{ diagram}}}Z_{D_{1}\dots D_{n}}\circ Z_{\diagramme{1}{}{{1/0}}{}{}{}{}{}}(G)G. 
	\end{array}
	\right.\]
\end{prop}
\begin{proof}
	
	Let $t_1t_2$ be a rooted forest of $2$ rooted trees. We have $t_1t_2=t_1\ast t_2- t_1\circ t_2$ so, $\varphi(t_1t_2)=\varphi(t_1)\varphi(t_2)$. We assume now there exists an positive integer $n\geq2$ such as for all rooted forest $t_1\dots t_n$ of $n$ rooted trees we have $\varphi(t_1\dots t_n)=\varphi(t_{1})\dots\varphi(t_{n})$. Let $t_1\dots t_{n+1}$ a rooted forest of $n+1$ rooted trees. We have:
	\[t_1\dots t_{n+1}=t_1\ast t_2\dots t_{n+1}-(t_1\circ t_2)t_3\dots t_{n+1}-\sum_{i=1}^{n-2}t_2\dots t_{2+i-1}(t_1\circ t_{2+i})t_{2+i+1}\dots t_{n+1}.\]
	Then, by using the fact that $\phi$ respects the Oudom and Guin structure and the recurrence hypothesis, $\varphi(t_1\dots t_{n+1})=\varphi(t_{1})\dots\varphi(t_{n+1})$.
	
	Let $t$ be a rooted tree of degree at least 2. There exists an unique positive integer $k$ and a $k$-uple $(t_1,\dots,t_k)$ such as $t=B(t_1,\dots, t_k)$. Actually, in words of Oudom and Guin, the grafting operator $B$ corresponds to the extended pre-Lie product of a rooted forest on the tree of degree $1$. So,  $t=B(t_1,\dots, t_k)=t_1\dots t_k \circ \tun$ and $\varphi(t)=L(\varphi(t_{1})\dots\varphi(t_{k}))$.
\end{proof}
\begin{cor}\label{corprovenotfree}
	Let $t$ be a rooted tree. If $t$ has at least one vertex of fertility strictly greater than two, then $\gamma(Z_{t})=0$ so $\varphi$ is not an injective morphism.
\end{cor}

\begin{proof}
	It is sufficient to prove the statement:
	\[\forall n\in\N^{*}\setminus \{1,2\},~\forall (D_1,\dots,D_n)\in(\D^{+})^n,~Z_{D_1\dots D_n}\circ Z_{\diagramme{1}{}{{1/0}}{}{}{}{}{}}=0.\]
	We prove the result by induction. Let $U$ and $V$ be two dissection diagrams. We recall that $Id$ is the identity morphism and $l$ is the projection on $\D^{+}$. By definition, 
	\ud{0.5}
	\begin{align*}
	L(UV)=&\sum\limits_{\substack{G\in\D \\ G \text{ diagram}}}\bigg(Z_{UV}\circ Z_{\diagramme{1}{}{{1/0}}{}{}{}{}{}}\bigg)(G)G\\
	=&\sum\limits_{\substack{G\in\D \\ G \text{ diagram}}}\bigg(Z_{U}\circ (Z_{V}\circ Z_{\diagramme{1}{}{{1/0}}{}{}{}{}{}})-\left(Z_{U}\circ Z_{V}\right)\circ Z_{\diagramme{1}{}{{1/0}}{}{}{}{}{}}\bigg)(G)G\\
	=&\sum\limits_{\substack{G\in\D \\ G \text{ diagram}}}\bigg((Z_{U}\ot Z_{V}\ot Z_{\diagramme{1}{}{{1/0}}{}{}{}{}{}}) \circ(Id\ot \De)\circ\De\bigg)(G)G\\
	-&\sum\limits_{\substack{G\in\D \\ G \text{ diagram}}}\bigg((Z_{U}\ot Z_{V}\ot Z_{\diagramme{1}{}{{1/0}}{}{}{}{}{}}) \circ(\De\ot Id)\circ(l\ot Id)\circ\De\bigg)(G)G\\
	=&\begin{cases}
	\sum\limits_{\substack{G\in\D \\ G \text{ diagram}}}\bigg(Z_{UV}\ot Z_{\diagramme{1}{}{{1/0}}{}{}{}{}{}}\bigg)\bigg(\De(G)\bigg)G& \mbox{if } U\neq V,\\ 
	2\sum\limits_{\substack{G\in\D \\ G \text{ diagram}}}\bigg(Z_{UV}\ot Z_{\diagramme{1}{}{{1/0}}{}{}{}{}{}}\bigg)\bigg(\De(G)\bigg)G& \mbox{else.}
	\end{cases}
	\end{align*}
	Let now $D_1$, $D_2$ and $D_{3}$ three dissection diagrams. We have:
	\[Z_{D_1D_2D_3}\circ Z_{\diagramme{1}{}{{1/0}}{}{}{}{}{}}=Z_{D_1}\circ \bigg( Z_{D_2D_{3}} \circ Z_{\diagramme{1}{}{{1/0}}{}{}{}{}{}} \bigg)- Z_{\left(D_1\circ D_2\right)D_3}\circ Z_{\diagramme{1}{}{{1/0}}{}{}{}{}{}} -Z_{D_2\left( D_1\circ D_{3}\right)} \circ Z_{\diagramme{1}{}{{1/0}}{}{}{}{}{}}.\]
	By the previous computation, $Z_{D_1D_2D_3}\circ Z_{\diagramme{1}{}{{1/0}}{}{}{}{}{}}=0$.
	
	We assume now there exists a positive integer $n$ greater than or equal to $3$ such as for all  $n$-uple $(D_1,\dots,D_n)$ of dissection diagrams $Z_{D_1\dots D_n}\circ Z_{\diagramme{1}{}{{1/0}}{}{}{}{}{}}=0$. Let $D_1$, \dots, $D_{n+1}$ be $n+1$ dissection diagrams. We have:
	\begin{align*}
	Z_{D_1\dots D_{n+1}}\circ Z_{\diagramme{1}{}{{1/0}}{}{}{}{}{}}=&Z_{D_1}\circ \bigg( Z_{D_2\dots D_{n+1}} \circ Z_{\diagramme{1}{}{{1/0}}{}{}{}{}{}} \bigg)- \bigg( Z_{D_1}\circ Z_{D_2\dots D_{n+1}} \bigg) \circ Z_{\diagramme{1}{}{{1/0}}{}{}{}{}{}}\\
	=&Z_{D_1}\circ \bigg( Z_{D_2\dots D_{n+1}} \circ Z_{\diagramme{1}{}{{1/0}}{}{}{}{}{}} \bigg)- Z_{\left(D_1\circ D_2\right)D_3\dots D_{n+1}}\circ Z_{\diagramme{1}{}{{1/0}}{}{}{}{}{}}\\
	-&\sum_{i=3}^{n-1}Z_{D_2\dots Z_{D_i}\left( D_1\circ D_{i+1}\right)D_{i+2}\dots D_{n+1}} \circ Z_{\diagramme{1}{}{{1/0}}{}{}{}{}{}}\\
	-&Z_{D_2\dots Z_{D_n}\left( D_1\circ D_{n+1}\right)} \circ Z_{\diagramme{1}{}{{1/0}}{}{}{}{}{}}\\
	=&0
	\end{align*}
\end{proof}

\begin{cor}\label{absliberte}
	The pre-Lie algebra generated by $\diagramme{1}{}{{1/0}}{}{}{}{}{}$ is not free, the pre-Lie algebra $\g{\D}$ is not free either. 
\end{cor}

\begin{proof}
	The first point is a direct consequence of corollary \ref{corprovenotfree}. For the second point we assume $\g{\D}$ is free as pre-Lie algebra. Let $V$ be the vector space such that $\g{\D}$ is the pre-Lie algebra freely generated by $V$. Let $\mathcal{B}$ be a basis of $V$. By Chapoton and Livernet \cite[corollary 1.10]{Chapoton2001} $\g{\D}$ is isomorphic to the pre-Lie algebra of rooted trees decorated by $\mathcal{B}$ with the grafting product. As there is just one dissection diagram of degre 1, this one is an element of $\mathcal{B}$. As a consequence, the pre-Lie algebra generated by $\diagramme{1}{}{{1/0}}{}{}{}{}{}$ is the pre-Lie algebra of no decorated rooted trees with the grafting product so, it is a free pre-Lie algebra, which is impossible.  
\end{proof}

\begin{conj}\label{conjnoyau}
	We call $\mathcal{N}$ the vector space defined by:
	\[\mathcal{N}=Vect(F\in\FRt,~\exists s\in V(F), fert(s)\geq 3)\]
	\emph{i.e.} the vector space generated by the non sub-binary forests.
	Then, $\mathcal{N}$ is the kernel of the Hopf algebra morphism $\varphi$.
\end{conj}

We consider now, as a departure space, the quotient Hopf algebra $\HsbT$ of sub-binary rooted trees. We keep the same notation $\varphi$ for this new morphism. We have:
\[\varphi:\left\{
\begin{array}{rcl}
\HsbT &\longrightarrow & \dual{\Hd}\\ 
B(t)\in\sbT &\longrightarrow & L(\varphi(t)),\\
B(t_{1}t_{2})\in\sbT &\longrightarrow & L(\varphi(t_{1})\varphi(t_{2})),\\
t_{1}\dots t_{n}\in\sbF &\longrightarrow & \varphi(t_{1})\dots\varphi(t_{n}),\\
\end{array}
\right.\]
where $L$ is the following linear map:
\[L:\left\{
\begin{array}{rcl}
\D^{+}\ot\K \oplus \D^{+}\ot\D^{+} &\longrightarrow & \D^{+}\\
D\ot1 &\longrightarrow & \sum\limits_{\substack{G\in\D \\ G \text{ diagram}}}\bigg(Z_{D}\circ Z_{\diagramme{1}{}{{1/0}}{}{}{}{}{}}\bigg)(G)G,\\ 
D_{1}\ot D_{2} &\longrightarrow & \sum\limits_{\substack{G\in\D \\ G \text{ diagram}}}\bigg(Z_{D_{1}D_{2}}\circ Z_{\diagramme{1}{}{{1/0}}{}{}{}{}{}}\bigg)(G)G. 
\end{array}
\right.\]

\Ex{}
We consider the rooted tree $t=\tdeux$. We know that $t=B(\tun)$ so we have:
\[\varphi(t)=L\po\diagramme{1}{}{{1/0}}{}{}{}{}{}\pf=2\diagramme{2}{}{{1/2,2/0}}{}{}{}{}{}+2\diagramme{2}{}{{1/0,2/0}}{}{}{}{}{}+(1+x)\diagramme{2}{}{{1/0,2/1}}{}{}{}{}{}.\]

For the rooted tree $t=\ttroisun=B(\tun\tun)$ we have:
\begin{align*}
	\varphi(t)=&L\po\diagramme{1}{}{{1/0}}{}{}{}{}{}\diagramme{1}{}{{1/0}}{}{}{}{}{}\pf\\
	=&2\diagramme{3}{}{{1/0,2/0,3/0}}{}{}{}{}{}+2x\diagramme{3}{}{{1/0,2/1,3/1}}{}{}{}{}{}+2\diagramme{3}{}{{1/2,2/0,3/2}}{}{}{}{}{}+2\diagramme{3}{}{{1/3,2/3,3/0}}{}{}{}{}{}\\
	+&2\diagramme{3}{}{{1/3,2/1,3/0}}{}{}{}{}{}+2x\diagramme{3}{}{{1/0,2/3,3/1}}{}{}{}{}{}+2\diagramme{3}{}{{1/2,2/0,3/0}}{}{}{}{}{}+2\diagramme{3}{}{{1/0,2/0,3/2}}{}{}{}{}{}.
\end{align*}  

\begin{lemma}
	Let $t$ be a sub-binary rooted tree of positive degree $n$. Its image $\varphi(t)$ is homogeneous of degree $n$.
\end{lemma}

\begin{proof}
	It is true by construction.
\end{proof}

\ud{0.7}
We aim at now determining $L(D)$ and $L(D_{1}D_{2})$ for all dissection diagrams $D$, $D_{1}$ and $D_{2}$. We formalize for this the dissection diagram construction by an inserting process of a chord in one or two dissection diagrams. We need more notations. Let $D$ a dissection diagram of degree $n\geq1$ and $i$ an integer in the interval $\llbracket0,n\rrbracket$.
\begin{itemize}
	\item The integer $i$ is the vertex $i$ (or the root, if $i=0$). Sometimes the vertex $i$ is denoted by $S_{D,i}$.
	\item We call fertility of the vertex $i$ the number $f_{D}(i)$ of chords of $D$ such that the vertex $i$ be one of its extremities. 
	\item The set of the chords of $D$ with the vertex $i$ in common labeled by respecting counter-clockwise is defined by: $A_{D,i}=\{a^{i}_{D,1},\dots,a^{i}_{D,f_{D}(i)}\}$.
	\item Let $u<v$ integers between 0 and $n$. We denote by $\{u,v\}$ a chord connecting the vertices $u$ and $v$. As the dissection diagram has a natural orientation we don't recall the orientation, of the chord $\{u,v\}$ in its notation.
	\item We consider an integer $t\in\llbracket0,f_{D}(i)\rrbracket$. We define two subsets of chords: $A^{t,1}_{D,i}=\{a^{i}_{D,1},\dots,a^{i}_{D,t}\}$ and $A^{t,2}_{D,i}=\{a^{i}_{D,t+1},\dots,a^{i}_{D,f_{D}(i)}\}$.
\end{itemize}
If there is no ambiguity on the considered dissection diagram, we forget its name in notations.\\

\Ex{} We want to illutrate the way of labeling chords with a common vertex. In the first example, we consider a vertex different from the root; we color it in red. In the second one, we consider the root and we color it in yellow. In the two cases, chords with the considered common vertex are colored in red. To relieve notations, the $s^{th}$ chords from the vertex $i$ is written $s$ instead of $a^{i}_{D,s}$. 
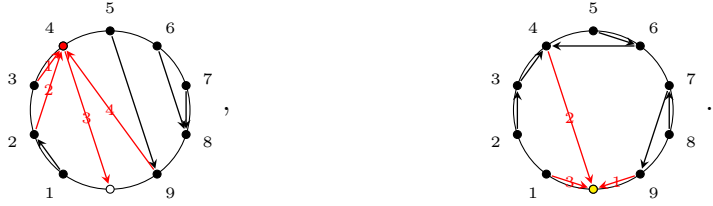
\begin{align*}
&\diagnot{{1/2,5/9,6/8,7/8}}{{3/4,2/4,4/0,9/4}}{4} , 
&\diagnot{{2/3,3/4,5/6,6/4,7/9,8/7}}{{9/0,4/0,1/0}}{0} .
\end{align*}

\subsection{Insertion of a chord in a vertex of a dissection diagram $D$}\label{fonctionouverture}
We consider $D$ a dissection diagram of degree $n\geq1$. We want to insert a new chord in $D$. To do that we start by choosing a vertex $i$ in $D$, we split $i$ in two vertices $s_{1}$ and $s_{2}$ and the chords of $D$ with $i$ in common too. The new object is not a dissection diagram. It is sufficient to build the chord  between $s_{1}$ and $s_{2}$ to obtain again a dissection diagram. By this way, we can build all dissection diagrams $G$ of degree  $n+1$ with at least one chord $a$ such that  $q_{\{a\}}(G)=D$.

\subsubsection{The chosen vertex of $D$ is different from the root.}
We choose a vertex of $D$ different from its root, so we just consider a positive integer $i$ in $\llbracket1,n\rrbracket$. Let  $t\in\llbracket0,f(i)\rrbracket$ be a integer, used to split the chord of $D$ with the common vertex $i$. We define the following map:
\[\phi_{i,t}:\left\{
\begin{array}{rcl}
\Ch(D) &\longrightarrow &\{\{u,v\} \mbox{, } 0\leq u<v\leq n+1\} \\
\{u,v\} &\longrightarrow & \begin{cases} 
\{u,v\}& \mbox{ if } (u\leq i-1  \mbox{ and } v\leq i-1 ) \\
& \mbox{ or } (u\leq i-1,~ v=i  \mbox{ and } \{u,v\}\in A^{t,1}_{i}),\\
\{u,v+1\}& \mbox{ if }(u\leq i-1,~ v=i \mbox{ and } \{u,v\}\in A^{t,2}_{i})\\
& \mbox{ or } (u=i \mbox{ and }\{u,v\}\in A^{t,1}_{i}),\\
\{u+1,v+1\}& \mbox{ if }(u=i \mbox{ and } \{u,v\}\in A^{t,2}_{i}) \mbox{ or } u\geq i+1 .
\end{cases}
\end{array}
\right.\]
With this map, we consider a diagram $\tilde{G}_{D,i,t}$ of degree $n+1$ which is \emph{open} between the vertices $i$ and $i+1$ with $\Ch(\tilde{G}_{D,i,t})=\phi_{i,t}(\Ch(D))$. This new diagram is not a dissection diagram, but an intermediate object in the definition of the insertion process.  
\Ex{}
We consider some dissection diagrams $D$ and we give their diagram $\tilde{G}_{D,i,t}$. The vertex $i$ to split is colored in red, the chords of $D$ in $A^{t,1}_{D,i}$ are colored in blue and the chords in $A^{t,2}_{D,i}$ are colored in green.   
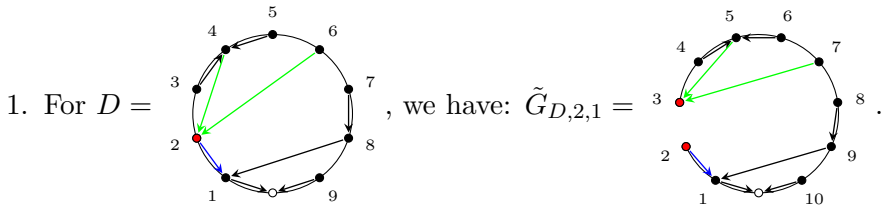
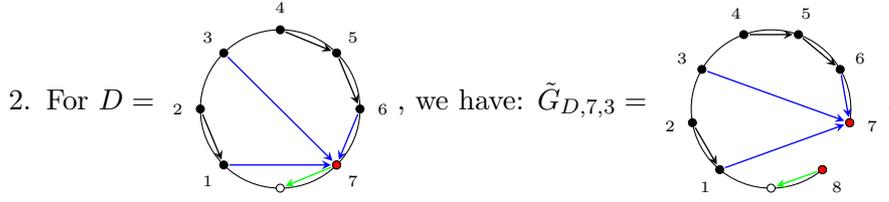
\begin{enumerate}
	\item For $D=\diagramme{9}{n}{{1/0,3/4,5/4,7/8,8/1,9/0}}{}{{2/1}}{{4/2,6/2}}{{2}}{}$, we have: $\tilde{G}_{D,2,1}=\diagramme{10}{n}{{1/0,4/5,6/5,8/9,9/1,10/0}}{}{{2/1}}{{5/3,7/3}}{{2,3}}{2}$. 
	\item For $D=\diagramme{7}{n}{{2/1,4/5,5/6}}{}{{1/7,3/7,6/7}}{{7/0}}{{7}}{}$, we have: $\tilde{G}_{D,7,3}=\diagramme{8}{n}{{2/1,4/5,5/6}}{}{{1/7,3/7,6/7}}{{8/0}}{{7,8}}{7}$. 
\end{enumerate}

It is then possible to define an insertion endomorphism homogeneous of degree 1 of the vector space of dissection diagrams. We consider a positive integer $i$ (choice of the vertex to plit to make insertion) and a non negative integer $t$ (partition of the chord with the vertex $i$ in common). We define the map $\Gamma_{i,t}$ by:
\[\Phi_{i,t}:\left\{
\begin{array}{rcl}
(\D)_{n}&\longrightarrow&(\D)_{n+1} \\
D&\longrightarrow& \begin{cases}
G_{D,i,t} \mbox{ with } \Ch(G_{D,i,t})=\phi_{i,t}(\Ch(D))\cup\{i,i+1\} &\mbox{if } i\leq n \mbox{ and } t\leq f_{D}(i),\\
0 & \mbox{else.}
\end{cases} 
\end{array}
\right.\]
\Rq{}  Let $D$ be a dissection diagram, $i$ be a vertex and $t$ be an integer in $\llbracket0, f_{D}(i)\rrbracket$.
\begin{enumerate}
	\item We easily know the orientation of the chord $\{i,i+1\}$ in $G_{D,i,t}=\Phi_{i,t}(D)$. Indeed, there exists an unique integer $l(i)$ such that the chord labeled by $i$ is the element $a^{i}_{D,l(i)}$ of $A_{D,i}$. If  $t\leq l(i)-1$ then  $\{i,i+1\}$ is labeled by $i$ else $\{i,i+1\}$ is labeled by $i+1$. 
	\item  The sum of maps $\Phi_{i,t}$ is called operation 1. 
\end{enumerate}

\Ex{} We use the two previous examples by keeping the same color code. The chord inserted is colored in red.
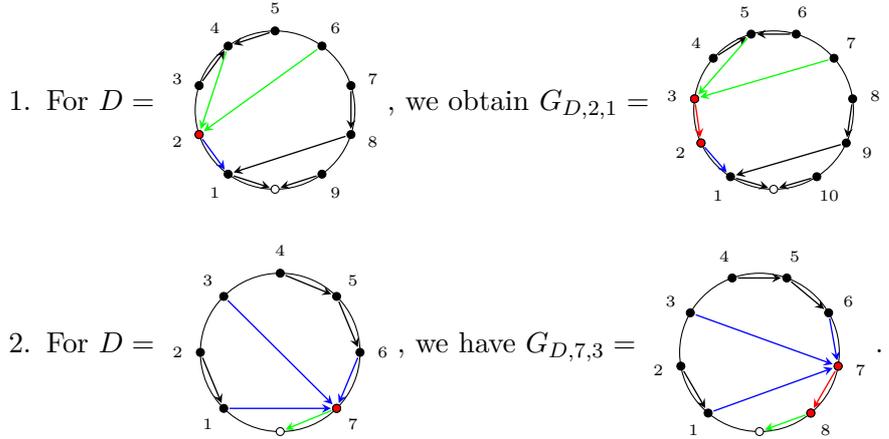
\begin{enumerate}
	\item For $D=\diagramme{9}{n}{{1/0,3/4,5/4,7/8,8/1,9/0}}{}{{2/1}}{{4/2,6/2}}{{2}}{}$, we obtain $G_{D,2,1}=\diagramme{10}{n}{{1/0,4/5,6/5,8/9,9/1,10/0}}{{3/2}}{{2/1}}{{5/3,7/3}}{{2,3}}{}$. 
	\item For $D=\diagramme{7}{n}{{2/1,4/5,5/6}}{}{{1/7,3/7,6/7}}{{7/0}}{{7}}{}$, we have $G_{D,7,3}=\diagramme{8}{n}{{2/1,4/5,5/6}}{{7/8}}{{1/7,3/7,6/7}}{{8/0}}{{7,8}}{}$. 	
\end{enumerate}

\subsubsection{The chosen vertex of $D$ is its root.}
We consider now the root of $D$. Let $\tau\in\llbracket0,f(0)\rrbracket$ be an integer to split the chords connected with the root. There exists an unique integer $s$ in $\llbracket1,n\rrbracket$ such that  $a^{0}_{D,\tau}=\{0,s\}$. We choose an element $\lambda\in\{0,1\}$ and define two maps.

\paragraph{Case 1: $\lambda=0$.} The map $\phi^{\lambda}_{0,\tau}$ is defined by:
\[\phi^{\lambda}_{0,\tau}:\left\{
\begin{array}{rcl}
\Ch(D) &\longrightarrow &\{\{u,v\} \mbox{, } 0\leq u<v\leq n+1\} \\
\{u,v\} &\longrightarrow & \begin{cases} 
\{u,v\}& \mbox{ if } u\geq 1 \mbox{ or } (u=0  \mbox{ and } v\leq s-1), \\
\{v,n+1\}& \mbox{ if } u=0 \mbox{ and } v\geq s.
\end{cases}
\end{array}
\right.\]
We call $\tilde{G}^{\lambda}_{D,0,\tau}$ the diagram of degree $n+1$, \emph{open} between the root and the vertex $n+1$ such that $\Ch(\tilde{G}^{\lambda}_{D,0,\tau})=\phi^{\lambda}_{0,\tau}(\Ch(D))$. The map $\phi^{\lambda}_{0,\tau}$ builds an \emph{open} diagram by creating the vertex $n+1$. 

\Ex{1} We choose the dissection diagram $D=\diagramme{9}{n}{{2/1,4/5,6/7,8/9}}{}{{1/0,3/0,5/0}}{{7/0,9/0}}{{0}}{}$. We consider $\tau=2$ and $\lambda=0$. We have: $\tilde{G}^{0}_{D,0,2}=\diagramme{10}{n}{{2/1,4/5,6/7,8/9}}{}{{1/0,3/0,5/0}}{{7/10,9/10}}{{0,10}}{10}$. 

\paragraph{Case 2: $\lambda=1$.} We define the map $\phi^{\lambda}_{0,\tau}$ by:
\[\phi^{\lambda}_{0,\tau}:\left\{
\begin{array}{rcl}
\Ch(D) &\longrightarrow &\{\{u,v\} \mbox{, } 0\leq u<v\leq n+1\} \\
\{u,v\} &\longrightarrow & \begin{cases} 
\{u+1,v+1\}& \mbox{ if } u\geq 1 \mbox{ or } (u=0  \mbox{ and } v\leq s-1), \\
\{u,v+1\}& \mbox{ if } u=0 \mbox{ and } v\geq s.
\end{cases}
\end{array}
\right.\] 
In this case the diagram $\tilde{G}^{\lambda}_{D,0,\tau}$ of degree $n+1$ is \emph{open} between the root and the vertex $1$ and  $\Ch(\tilde{G}^{\lambda}_{D,0,\tau})=\phi^{\lambda}_{0,\tau}(\Ch(D))$. 

\Ex{1} We use again the dissection diagram $D=\diagramme{9}{n}{{2/1,4/5,6/7,8/9}}{}{{1/0,3/0,5/0}}{{7/0,9/0}}{{0}}{}$ with $\tau=2$ but now $\lambda=1$. We obtain: $\tilde{G}^{1}_{D,0,2}=\diagramme{10}{n}{{3/2,5/6,7/8,9/10}}{}{{2/1,4/1,6/1}}{{8/0,10/0}}{{0,1}}{0}$. 

We define now the endomorphism homogeneous of degree 1 coding the insertion  of a chord in the root. To do that we consider a non negative integer (partition of the chords connected with the root) and $\lambda\in\{0,1\}$. We define $\phi^{\lambda}_{0,\tau}$ by:
\[\Phi^{\lambda}_{0,\tau}:\left\{
\begin{array}{rcl}
(\D)_{n}&\longrightarrow&(\D)_{n+1} \\
D&\longrightarrow& \begin{cases}
G^{\lambda}_{0,\tau}(D) & \mbox{ if } \tau\leq f_{D}(0),\\
0 &\mbox{ else} 
\end{cases}
\end{array}
\right.\]
where the set of the chords of $G^{\lambda}_{0,\tau}(D)$ is:
\[\Ch(G^{\lambda}_{0,\tau}(D))=\phi^{\lambda}_{0,\tau}(\Ch(D))\cup\{\{0,(1-\lambda)n+1\}\}.\]

\Rq{} Let $D$ be a dissection diagram, $\tau$ be an integer in $\llbracket0, f_{D}(0)\rrbracket$ and $\lambda$ be an element in $\{0,1\}$.
\begin{enumerate}
	\item The chord inserted to build $G^{\lambda}_{D,0,\tau}$ is naturally oriented. Indeed, either $\lambda=0$ and we build $\{0,n+1\}$ (chord labeled by $n+1$), either $\lambda=1$ and then we buid $\{0,1\}$ (chord labeled by $1$) in $G^{\lambda}_{D,0,\tau}$.
	\item The sum of all maps $\Phi^{\lambda}_{0,t}$ is called operation 2.
\end{enumerate}

\Ex{} We use the two previous examples. The inserted chord is colored in red.
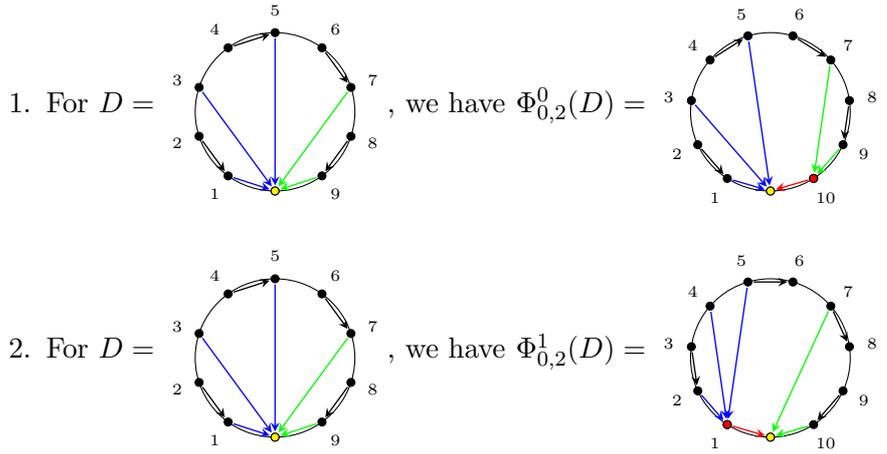
\begin{enumerate}
	\item For $D=\diagramme{9}{n}{{2/1,4/5,6/7,8/9}}{}{{1/0,3/0,5/0}}{{7/0,9/0}}{{0}}{}$, we have $\Phi^{0}_{0,2}(D)=\diagramme{10}{n}{{2/1,4/5,6/7,8/9}}{{10/0}}{{1/0,3/0,5/0}}{{7/10,9/10}}{{0,10}}{}$.
	\item For $D=\diagramme{9}{n}{{2/1,4/5,6/7,8/9}}{}{{1/0,3/0,5/0}}{{7/0,9/0}}{{0}}{}$, we have $\Phi^{1}_{0,2}(D)=\diagramme{10}{n}{{3/2,5/6,7/8,9/10}}{{1/0}}{{2/1,4/1,5/1}}{{7/0,10/0}}{{0,1}}{}$.
\end{enumerate}

\subsubsection{Computation of $L(D)$ where $D$ is a dissection diagram of degree $n\geq 1$.}

\begin{prop}\label{insertionuna}
	Let $D$ be a dissection diagram of degree $n\geq1$, $i$ and $j$ two integers in $\llbracket 1,n\rrbracket$, $t_{i}$ (respectively $t_{j}$) be an integer in $\llbracket 0,f(i)\rrbracket$ (respectively in $\llbracket 0,f(j)\rrbracket$), $\tau_{1}$ and $\tau_{2}$ two integers in $\llbracket 0,f(0)\rrbracket$ and, $\lambda_{1}$ and $\lambda_{2}$ two elements in $\{0,1\}$. We have:
	\begin{enumerate}
		\item $(\Phi_{i,t_{1}}(D),\{i,i+1\})=(\Phi_{j,t_{2}}(D),\{j,j+1\})\iff (i=j \text{ and } t_{1}=t_{2})$,
		\item $(\Phi^{\lambda_{1}}_{0,\tau_{1}}(D),\{0,(1-\lambda_{1})n+1\})=(\Phi^{\lambda_{2}}_{0,\tau_{2}}(D),\{0,(1-\lambda_{2})n+1\})\iff (\lambda_{1}=\lambda_{2} \text{ and } \tau_{1}=\tau_{2})$,
		\item $(\Phi_{i,t_{1}}(D),\{i,i+1\})\neq (\Phi^{\lambda_{1}}_{0,\tau_{1}}(D),\{0,(1-\lambda_{1})n+1\})$.		
	\end{enumerate}
\end{prop} 

\begin{proof}
Let $D$ be a dissection diagram of degree $n\geq1$, $i$ and $j$ two integers in $\llbracket 1,n\rrbracket$, $t_{i}$ (respectively $t_{j}$) be an integer in $\llbracket 0,f(i)\rrbracket$ (respectively in $\llbracket 0,f(j)\rrbracket$), $\tau_{1}$ and $\tau_{2}$ two integers in $\llbracket 0,f(0)\rrbracket$ and, $\lambda_{1}$ and $\lambda_{2}$ two elements in $\{0,1\}$. 
	\begin{enumerate}
		\item It is sufficient to prove the implication:
		\[(\Phi_{i,t_{1}}(D),\{i,i+1\})=(\Phi_{i,t_{2}}(D),\{i,i+1\})\Rightarrow t_{1}=t_{2}.\]
		We denote by $p_{1}$ (respectively by $p_{2}$) the vertex $i$ of the dissection diagram $\Phi_{i,t_{1}}(D)$ (respectively the dissection diagram $\Phi_{i,t_{2}}(D)$) and we consider its fertility $f(p_{1})$ (respectively $f(p_{2})$). We obtain $f(p_{1})=t_{1}+1$ and $f(p_{2})=t_{2}+1$. If we assume that the two couples $(\Phi_{i,t_{1}}(D),\{i,i+1\})$ and $(\Phi_{i,t_{2}}(D),\{i,i+1\})$ are equal then the implication becomes trivial.
		\item We put $\lambda=\lambda_{1}$. To prove the equivalence, it is sufficient to consider the implication: \[(\Phi^{\lambda}_{0,\tau_{1}}(D),\{0,(1-\lambda)n+1\})=(\Phi^{\lambda}_{0,\tau_{2}}(D),\{0,(1-\lambda)n+1\})\Rightarrow  \tau_{1}=\tau_{2}.\] We use the same process as before.
		\item As $i$ is positive, it is trivial.
	\end{enumerate}
\end{proof}

\begin{cor}
	Let $D$ be a dissection diagram of degree $n\geq1$. We call $\sigma_{D}$ the number of different couples $(G,a)$, where $G$ is a dissection diagram of degree $n+1$ and $a$ is a chord $a=\{u,v\}$ with $0\leq u < v\leq n+1$, obtained by operations 1 and 2 applied to $D$. We have:
	\[\sigma_{D}=3n+2+f_{D}(0).\]
\end{cor}

\begin{proof}
	Let $D$ be a dissection diagram of degree $n\geq1$. By direct computation, we have:
	\begin{align*}\sigma(D)&=\overbrace{\sum\limits_{i=1}^{n}(f_{D}(i)+1)}^{\text{operation 1}}+\overbrace{2(f_{D}(0)+1)}^{\text{operation 2}}\\
		&=(3n-f_{D}(0))+2(f_{D}(0)+1)\\
		&=3n+2+f_{D}(0).\end{align*}
\end{proof}

\begin{prop}\label{valLundiagramme}
	Let $D$ be a dissection diagram of degree $n\in\N^{*}$. We have
	\[L(D)=\sum\limits_{\mathclap{\substack{i\in\llbracket1,n\rrbracket \\ t\in\llbracket0,f_{D}(i)\rrbracket}}}x^{k_{\{a_{i}\}}(\Phi_{i,t}(D))}\Phi_{i,t}(D)+\sum\limits_{\mathclap{\tau\in\llbracket 0,f_{D}(0)\rrbracket}}\left[\Phi^{0}_{0,\tau}(D)+\Phi^{1}_{0,\tau}(D)\right],\]
	with $a_{i}=\{i,i+1\}$.
	In other words,
	\[L(D)=\sum\limits_{\mathclap{\substack{i\in\llbracket1,n \rrbracket \\ t\in\llbracket0,l(i)-1 \rrbracket}}}\Phi_{i,t}(D)+x\sum\limits_{\mathclap{\substack{i\in\llbracket1,n \rrbracket \\ t\in\llbracket l(i),f_{D}(i) \rrbracket}}}\Phi_{i,t}(D)+\sum\limits_{\mathclap{\tau\in\llbracket 0,f_{D}(0)\rrbracket}}\left[\Phi^{0}_{0,\tau}(D)+\Phi^{1}_{0,\tau}(D)\right],\]
	where for all integer $i\in\llbracket1,n \rrbracket$, $l(i)$  is the unique integer in $\llbracket1,f_{D}(i)\rrbracket$ such that the chord $a_{D,l(i)}^{i}$ of $D$ is labeled by $i$.
\end{prop}

\begin{proof}
	The fact that dissection diagrams obtained with operations 1 and 2 are elements of $L(D)$ is trivial.
	
	Let $(G,a)\in (\D)_{n+1}\times\Ch(G)$ such that $q_{\{a\}}=D$. We write $a$ as $a=\{i,j\}$ with $\{i,j\}\in\{\{u,v\}, 0\leq u<v\leq n+1\}$. \\
	\ul{Case 1 : $i\neq 0$.} $q_{\{a\}}(G)=D$ (only one diagram) so $j=i+1$. We have then 
	\[(G,a)=(\Phi_{i,t}(D),\{i,i+1\}) \mbox{ with } t=f_{G}(i)-1.\] 
	\ul{Case 2 : $i=0$.} $q_{\{a\}}(G)=D$ (only one diagram) so $j\in\{1,n+1\}$. We have then  
	\[(G,a)=\begin{cases}(\Phi^{1}_{0,\tau}(D),\{0,1\}) \mbox{ with } \tau=f_{G}(0)-1& \mbox{if } j=1,\\
	(\Phi^{0}_{0,\tau}(D),\{0,n+1\}) \mbox{ with } \tau=f_{G}(n+1)-1 & \mbox{if } j=n+1. \end{cases}\] 
\end{proof}

\begin{cor} We assume $x\in\N$. Let $D$ be a dissection diagram of degree $n\in\N^{*}$. We denote by $\sigma_{L(D)}$ the number of terms in $L(D)$ counted with multiplicity. We have:
	\[\sigma_{L(D)}=3nx + (1-x)\sum_{i=1}^{n}l(i) + (2-x)f_{D}(0) + 2\]
	where for all integer $i\in\llbracket1,n \rrbracket$, $l(i)$ is the unique integer in $\llbracket1,f_{D}(i)\rrbracket$ such that the chord $a_{l(i)}^{i}$ of $D$ is labeled by $i$.
\end{cor}

\begin{proof}
	Direct computation.
\end{proof}

\begin{prop}\label{etapeun}
	Let $x$ be a scalar in $\K$. Let $D_{1}$ and $D_{2}$ be two nonempty dissection diagrams. We have the equivalence:
	\[L(D_{1})= L(D_{2})\iff D_{1}=D_{2}.\]
\end{prop}

\ud{0.7}
\begin{proof}
	Let $x$ be a scalar in $\K$. Let $D_{1}$ and $D_{2}$ be two nonempty dissection diagrams.  We recall that $f_{D_{1}}(0)$ (respectively $f_{D_{2}}(0)$) is the root fertility of $D_{1}$ (respectively $D_{2}$). We know that, for all positive integer $n$, the projection of $L(D_{1})$ (respectively  $L(D_{2})$) on the linear space of dissection diagrams with the root fertility equal to $n$ is positive if $n\in\llbracket1,f_{D_{1}}(0)+1\rrbracket\setminus\{f_{D_{1}}(0)\}$ (respectively $n\in\llbracket1,f_{D_{2}}(0)+1\rrbracket\setminus\{f_{D_{2}}(0)\}$), and equals zero if $n\geq f_{D_{1}}(0)+2$ (respectively $n\geq f_{D_{2}}(0)+2$). As a conclusion, if $f_{D_{1}}(0)\neq f_{D_{2}}(0)$ the two elements $L(D_{1})$ and $L(D_{2})$ are different.
	We consider then two dissection diagrams $D_{1}$ and $D_{2}$ such that $f_{D_{1}}(0)=f_{D_{2}}(0)$. We assume $L(D_{1})$ and $L(D_{2})$ equal. We write $D_{1}$ and $D_{2}$ as $D_{1}=\defDiag{{6/0}}{{2/6/A_{1},10/6/B_{1},9/0/C_{1}}}$ and $D_{2}=\defDiag{{6/0}}{{2/6/A_{2},10/6/B_{2},9/0/C_{2}}}$. We define $S_{1}=\defDiag{{9/0}}{{4/0/D_{1}}}$, $S_{2}=\defDiag{{4/0}}{{9/0/D_{1}}}$, $P_{1}=\defDiag{{9/0}}{{4/0/D_{2}}}$, $P_{2}=\defDiag{{4/0}}{{9/0/D_{2}}}$. The sets $\{S_{1},S_{2}\}$ and $\{P_{1},P_{2}\}$ are equal. There are different cases.
	
	If $S_{1}=S_{2}$ or $S_{1}\neq S_{2}$ with $S_{1}=P_{1}$ the result is trivial.
	
	If $S_{1}\neq S_{2}$ and $S_{1}=P_{2}$ then the dissection diagrams $A_{1}$, $A_{2}$, $B_{1}$, $B_{2}$ are all empty. Furthermore, we have   $\defDiag{{9/0}}{{4/0/C_{1}}}=\defDiag{{4/0}}{{9/0/C_{2}}}$ and $\defDiag{{9/0}}{{4/0/C_{2}}}=\defDiag{{4/0}}{{9/0/C_{1}}}$. By induction on their root fertility, we obtain that $C_{1}=C_{2}$ and the proposition is proved. 
\end{proof}
\ud{0.7}

\begin{prop}
	In this proposition, we assume that $\R\subset\K$ and $x\in\R_{+}^{*}$. Let $D$ be a dissection diagram of degree $n$ greater than or equal to $2$. We define the pairing  \[\langle-,-\rangle:\left\{\begin{array}{rcl}\Hd\ot\Hd &\longrightarrow &\K \\ D_{1}\dots D_{k}\ot G_{1} \dots G_{l} &\longrightarrow & Z_{D_{1}\dots D_{k}}(G_{1} \dots G_{l}). \end{array}\right.\] For all dissection diagram $D$ of degree $n\geq2$, there exists a dissection diagram $G$ of degree $n-1$ such that $\langle D,L(G)\rangle\neq0$.
\end{prop}

\begin{proof}
	Let $x$ be positive scalar and $D$ a dissection diagram of degree $n$ greater than or equal to $2$. As there is not any intersection between its chords there exists a chord $a$ such that $a$ connects two consecutive vertices. We call $G$ the dissection diagram $G=q_{\{a\}}(D)$. It has a degree equal to $n-1$ and answers the question.
\end{proof}

\begin{prop}
	In this proposition, we assume that $\R\subset\K$ and $x\in\R_{+}^{*}$. Let $n$ be a non negative integer and $e_{n}\in\TR$ the ladder of degree $n$. By using the previous pairing we have: for all non negative integer $n$ and all dissection diagram $D$ of degree $n$ the scalar  $\langle D,\varphi(e_{n})\rangle$ is positive.
\end{prop}

\ud{0.5}
\begin{proof} We give a recursive proof. The result is trivial for $e_{1}$. We assume the result true for a particular positive rank $n$. We write  $\varphi(e_{n})=\sum\limits_{\substack{G\in(\D)_{n} \\ \text{diagram}}}a_{G}G$ where each $a_{G}$ is positive. Then,
	\[\varphi(e_{n+1})=\varphi(e_{n})\circ \diagramme{1}{}{{1/0}}{}{}{}{}{}=\sum_{\substack{G\in(\D)_{n} \\ \text{diagram}}}a_{G}L(G)=\sum_{\substack{D\in(\D)_{n+1} \\ \text{diagram}}}b_{D}D\]
	and, according to the previous proposition, all $b_{D}$ are positive scalars.
	The result is true for the rank $n+1$ so the propostion is proved.
\end{proof}

\ud{0.7}

\subsection{Insertion of a chord in two dissection diagrams}
Let $D_{1}$ and $D_{2}$ be two dissection diagrams of positive degree $n_{1}$ and $n_{2}$. We choose a vertex $i$ of $D_{1}$, a vertex $j$ of $D_{2}$, integers $t\in\llbracket 0,f_{D_{1}}(i)\rrbracket$ and $\tau\in\llbracket 0,f_{D_{2}}(j)\rrbracket$. We want to insert a new chord by using the following steps. Thank to the opening maps defined in section \ref{fonctionouverture}, we obtain two \emph{open} diagrams. Then we change the vertices labels and we insert a new chord to have a dissection diagram of degree $n_{1}+n_{2}+1$. As a dissection diagram has only one root, we open at least one of the two dissection diagrams $D_{1}$ and $D_{2}$ in the root. We introduce an integer $\lambda\in\{0,1\}$, useful for the rest of the work.

\subsubsection{Open $D_{1}$ in the root and $D_{2}$ in another vertex.}
We consider the root of $D_{1}$, an integer $t\in\llbracket 0,f_{D_{1}}(0)\rrbracket$ (partition of the chords of $D_{1}$ connected with the root), an integer $\lambda\in\{0,1\}$ (location of the root of the \emph{open} diagram $\tilde{G}_{0,t}^{\lambda}$ obtained with $D_{1}$), a vertex of $D_{2}$ different from the root (\emph{i.e.} an integer $j\in\llbracket 1,n_{2}\rrbracket$) and an integer $\tau\in\llbracket 0,f_{D_{2}}(j)\rrbracket$ (partition of the chords of $D_{2}$ with the vertex $j$ in common). The root of the dissection diagram $G_{D_{1},D_{2},t,j,\tau}^{\lambda}$ built after the insertion is given by the root of $D_{2}$.

According to paragraph \ref{fonctionouverture}, thank to the map $\phi_{0,t}^{\lambda}$ (respectively $\phi_{j,\tau}$), we can consider the  \emph{open} diagram $\tilde{G}_{D_{1},0,t}^{\lambda}$ equipped with the set $\phi_{0,t}^{\lambda}(\Ch(D_{1}))$ (respectively  $\tilde{G}_{D_{2},j,\tau}$ equipped with the set $\phi_{j,\tau}(\Ch(D_{2}))$). 

To change the labels of the chords, as the root of $G_{D_{1},D_{2},t,j,\tau}^{\lambda}$ comes from the root of $D_{2}$, we consider the maps:
\[\gamma_{D_{1},t}^{\lambda}:\left\{
\begin{array}{rcl}
\phi_{t,\lambda}^{0}(\Ch(G_{1})) &\longrightarrow &\{\{u,v\} \mbox{, } 0\leq u<v\leq n_{1}+n_{2}+1\} \\
\{u,v\} &\longrightarrow & \begin{cases} 
\{u+j-\lambda,v+j-\lambda\}& \mbox{ if } u\geq 1, \\
\{j+\lambda(n_{1}+1),v+j-\lambda\}& \mbox{ if } u=0
\end{cases}
\end{array}
\right.\]
and
\[\gamma_{D_{2},j,\tau}:\left\{
\begin{array}{rcl}
\phi_{j,\tau}(\Ch(G_{2})) &\longrightarrow &\{\{u,v\} \mbox{, } 0\leq u<v\leq n_{1}+n_{2}+1\} \\
\{u,v\} &\longrightarrow & \begin{cases} 
\{u,v\}& \mbox{ if } u\leq j  \mbox{ and } v\leq j, \\
\{u,v+n_{1}\}& \mbox{ if } u\leq j \mbox{ and } v\geq j+1,\\
\{u+n_{1},v+n_{1}\}& \mbox{ if } u\geq j+1.
\end{cases}
\end{array}
\right.\] 
With these functions, we have a diagram of degree $n_{1}+n_{2}+1$ which is not a dissection diagram since there are just $n_{1}+n_{2}$ chords. We just have to build the chord $\{j,j+n_{1}+1\}$.

We define now a morphism from  $\D\otimes\D$ to $\D$ homogeneous of degree $1$. We consider two non negative integers $t$ and $\tau$, an element $\lambda$ in $\{0,1\}$, a positive integer $j$ and the following map: 
\[\Gamma_{t,j,\tau}^{\lambda}:\left\{
\begin{array}{rcl}
(\D)_{n_{1}}\ot(\D)_{n_{2}}  &\longrightarrow &(\D)_{n_{1}+n_{2}+1} \\
D_{1}\ot D_{2} &\longrightarrow & \begin{cases}
G_{D_{1},D_{2},t,j,\tau}^{\lambda} & \mbox{ if } t\leq f_{D_{1}}(0)\text{, } j\leq n_{2} \text{ and } \tau\leq f_{D_{2}}(j),\\
0 &\mbox{ else,}
\end{cases} 
\end{array}
\right.\] 
where the set $\Ch(G_{D_{1},D_{2},t,j,\tau}^{\lambda})$ is given by: \[\Ch(G_{D_{1},D_{2},t,j,\tau}^{\lambda})=\gamma_{D_{1},t}^{\lambda}(\phi_{0,t}^{\lambda}(\Ch(D_{1})))\cup\gamma_{D_{2},j,\tau}(\phi_{j,\tau}(\Ch(G_{2})))\cup\{\{j,j+n_{1}+1\}\}.\]

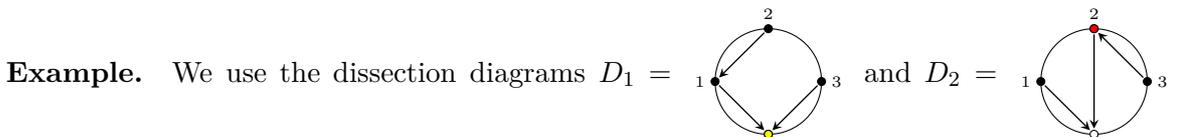
\Ex{1} We use the dissection diagrams $D_{1}=\diagramme{3}{n}{{1/0,2/1,3/0}}{}{}{}{{0}}{}$ and $D_{2}=\diagramme{3}{n}{{1/0,2/0,3/2}}{}{}{}{{2}}{}$ with $\lambda=1$, $t=1$, $j=2$ and $\tau=1$. We build first $\tilde{G}_{D_{1},0,1}^{1}$ and $\tilde{G}_{D_{2},2,1}$ by keeping the color code explained in section \ref{fonctionouverture}. Finally, we make the insertion. The old root of $D_{1}$ becomes white and the new chord is colored in red. We have:
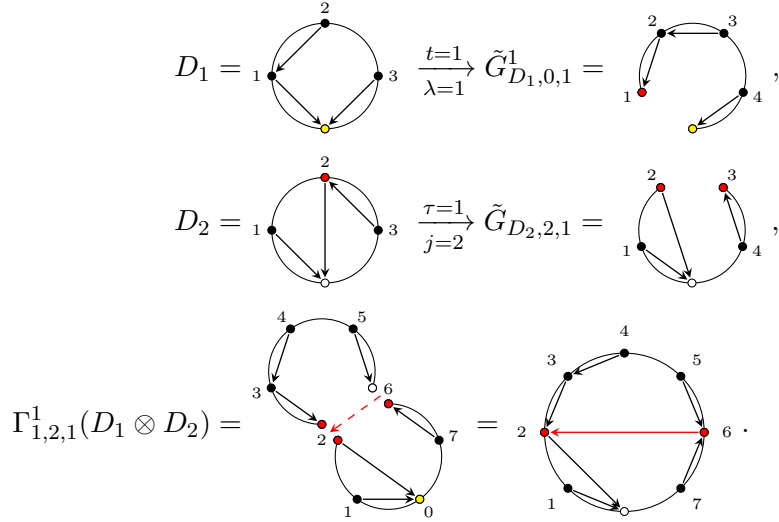
\begin{align*}
	D_{1}=&\diagramme{3}{n}{{1/0,2/1,3/0}}{}{}{}{{0}}{} \xrightarrow[\lambda=1]{t=1} \tilde{G}_{D_{1},0,1}^{1}=\diagramme{4}{n}{{2/1,3/2,4/0}}{}{}{}{{0,1}}{0}, \\
	D_{2}=&\diagramme{3}{n}{{1/0,2/0,3/2}}{}{}{}{{2}}{} \xrightarrow[j=2]{\tau=1} \tilde{G}_{D_{2},2,1}=\diagramme{4}{n}{{1/0,2/0,4/3}}{}{}{}{{2,3}}{2}, \\
	\Gamma_{1,2,1}^{1}(D_{1}\ot D_{2})=&\colldiag{1}=\diagramme{7}{n}{{1/0,2/0,3/2,4/3,5/6,7/6}}{{6/2}}{}{}{{2,6}}{}.
\end{align*}

\Rq{} Let $D_{1}$ and $D_{2}$ be two dissection diagrams of respective  positive degree $n_{1}$ and $n_{2}$ and $\lambda\in\{0,1\}$, $t\in\llbracket 0,f_{D_{1}}(0)\rrbracket$, $j\in\llbracket 1,n_{2}\rrbracket$ and $\tau\in\llbracket 0,f_{D_{2}}(j)\rrbracket$ be integers.
\begin{enumerate}
	\item It is not necessary to consider the two possible values of $\lambda$. Indeed by direct computation we have $\Gamma_{t,j,\tau}^{\lambda}(D_{1}\ot D_{2})=\Gamma_{t,j,\tau}^{1-\lambda}(D_{1}\ot D_{2})$.
	\item We easily know the orientation of the inserted chord $\{j,j+n_{1}+1\}$. Indeed there exists an unique integer $l(j)$ such that the chord $a^{j}_{D_{2},l(j)}$ of $D_{2}$ is labeled by $j$. If $\tau\leq l(j)-1$ then the chord $\{j,j+n_{1}+1\}$ of $\Gamma_{t,j,\tau}^{\lambda}(D_{1}\ot D_{2})$ is labeled by $j$ else $\tau\geq l(j)$ and the chord $\{j,j+n_{1}+1\}$ of $\Gamma_{t,j,\tau}^{\lambda}(D_{1}\ot D_{2})$ is labeled by $j+n_{1}+1$.
	\item We call $\nu$ the map sending $D_{1}\ot D_{2}$ to $D_{2}\ot D_{1}$. The sum of maps of types $\Gamma_{t,j,\tau}^{\lambda}$ or $\Gamma_{t,j,\tau}^{\lambda}\circ\nu$ is called operation $3$.
\end{enumerate}

\subsubsection{Open $D_{1}$ and $D_{2}$ in the root.}
We decide now to work with the root of the two dissection diagrams. We consider an integer $t\in\llbracket0,f_{D_{1}}(0)\rrbracket$ (partition of the chords of $D_{1}$ connected with the root), an integer $\tau\in\llbracket0,f_{D_{2}}(0)\rrbracket$ (partition of the chords of $D_{2}$ connected with the root) and two integers  $\lambda_{1},\lambda_{2}\in\{0,1\}$ (location of the root of the two diagrams  $\tilde{G}_{D_{1},0,t}^{\lambda_{1}}$ and $\tilde{G}_{D_{2},0,\tau}^{\lambda_{2}}$, built with opening maps defined in section \ref{fonctionouverture}). The root of the dissection diagram $G_{D_{1},D_{2},t,\tau}^{\lambda_{1},\lambda_{2}}$, built after the insertion, is again given by the root of $D_{2}$. There are two cases.

\paragraph{Case 1: $\lambda=\lambda_{2}=0$ and $\lambda_{1}=1-\lambda=1$.}
According to paragraph \ref{fonctionouverture}, thank to the map $\phi_{0,t}^{1}$ (respectively $\phi_{0,\tau}^{0}$), we can consider the  \emph{open} diagram $\tilde{G}_{D_{1},0,t}^{1}$ equipped with the set $\phi_{0,t}^{1}(\Ch(D_{1}))$ (respectively  $\tilde{G}_{D_{2},0,\tau}^{0}$ with the set $\phi_{0,\tau}^{0}(\Ch(D_{2}))$).

To change labels, as the root of $G_{D_{1},D_{2},t,\tau}^{1,1}$ is given by the root of $D_{2}$, we use the following maps:
\[\gamma^{1}_{D_{1},t}:\left\{
\begin{array}{rcl}
\phi_{0,t}^{1}(\Ch(D_{1})) &\longrightarrow &\{\{u,v\} \mbox{, } 0\leq u<v\leq n_{1}+n_{2}+1\} \\
\{u,v\} &\longrightarrow & \begin{cases} 
\{u+n_{2},v+n_{2}\}& \mbox{ if } u\geq 1, \\
\{u,v+n_{2}\}& \mbox{ if } u=0
\end{cases}
\end{array}
\right.\] 
and
\[\gamma^{0}_{D_{2},\tau}:\left\{
\begin{array}{rcl}
\phi_{0,\tau}^{0}(\Ch(D_{2})) &\longrightarrow &\{\{u,v\} \mbox{, } 0\leq u<v\leq n_{1}+n_{2}+1\} \\
\{u,v\} &\longrightarrow & \{u,v\}. 
\end{array}
\right.\]  
The diagram of degree $n_{1}+n_{2}+1$ built with the previous map is not a dissection diagram since there are just $n_{1}+n_{2}$ chords. We just have to add the chord $\{0,n_{2}+1\}$. 

We define a morphism from $\D\otimes\D$ to $\D$ homogeneous of degree $1$. Let $t$ and $\tau$ be two non negative integers and we consider the map: 
\[\Gamma_{t,\tau}^{0}:\left\{
\begin{array}{rcl}
(\D)_{n_{1}}\ot(\D)_{n_{2}}  &\longrightarrow &(\D)_{n_{1}+n_{2}+1} \\
D_{1}\ot D_{2} &\longrightarrow &  \begin{cases}
G_{D_{1},D_{2},t,\tau}^{1,0} & \mbox{ si } t\leq f_{D_{1}}(0) \text{ and } \tau\leq f_{D_{2}}(0),\\
0 &\mbox{ sinon,}
\end{cases} 
\end{array}
\right.\] 
with \[\Ch(G_{D_{1},D_{2},t,\tau}^{1,0})=\gamma^{1}_{D_{1},t}(\phi_{0,t}^{1}(\Ch(D_{1})))\cup\gamma^{0}_{D_{2},\tau}(\phi_{0,\tau}^{0}(\Ch(D_{2})))\cup\{\{0,n_{2}+1\}\}.\]

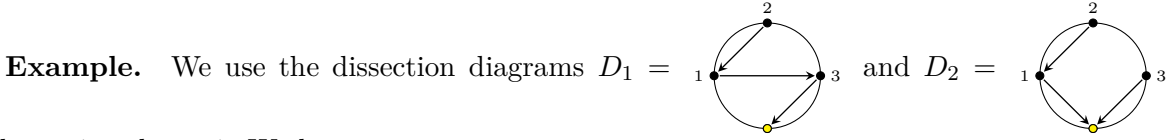
\Ex{1} We use the dissection diagrams $D_{1}=\diagramme{3}{n}{{1/3,2/1,3/0}}{}{}{}{{0}}{}$ and $D_{2}=\diagramme{3}{n}{{1/0,2/1,3/0}}{}{}{}{{0}}{}$ with $t=1$ and $\tau=1$. We have:
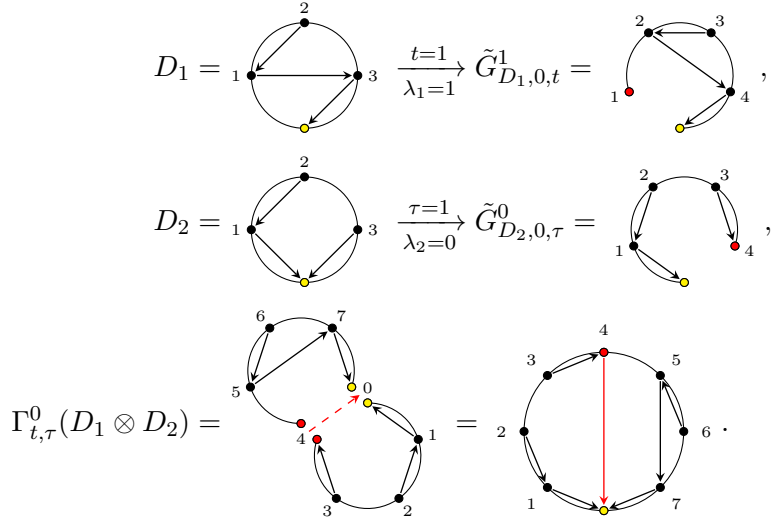
\begin{align*}
	D_{1}=&\diagramme{3}{n}{{1/3,2/1,3/0}}{}{}{}{{0}}{}  \xrightarrow[\lambda_{1}=1]{t=1} \tilde{G}_{D_{1},0,t}^{1}=\diagramme{4}{n}{{2/4,3/2,4/0}}{}{}{}{{0,1}}{0}, \\
	D_{2}=&\diagramme{3}{n}{{1/0,2/1,3/0}}{}{}{}{{0}}{}  \xrightarrow[\lambda_{2}=0]{\tau=1} \tilde{G}_{D_{2},0,\tau}^{0}=\diagramme{4}{n}{{1/0,2/1,3/4}}{}{}{}{{0,4}}{4}, \\
	\Gamma_{t,\tau}^{0}(D_{1}\ot D_{2})=&\colldiag{3}=\diagramme{7}{n}{{1/0,2/1,3/4,5/7,6/5,7/0}}{{4/0}}{}{}{{0,4}}{}.
\end{align*}

\paragraph{Case 2: $\lambda=\lambda_{2}=1$ and $\lambda_{1}=1-\lambda=0$.}
We consider now the \emph{open} diagram $\tilde{G}_{D_{1},0,t}^{0}$ equipped with the set $\phi_{0,t}^{0}(\Ch(D_{1}))$  (respectively  $\tilde{G}_{D_{2},0,\tau}^{1}$ with the set $\phi_{0,\tau}^{1}(\Ch(D_{2}))$).

To change labels we use:
\[\gamma^{0}_{D_{1},t}:\left\{
\begin{array}{rcl}
\phi_{0,t}^{0}(\Ch(D_{1})) &\longrightarrow &\{\{u,v\} \mbox{, } 0\leq u<v\leq n_{1}+n_{2}+1\} \\
\{u,v\} &\longrightarrow & \{u,v\}
\end{array}
\right.\] 
and
\[\gamma^{1}_{D_{2},\tau}:\left\{
\begin{array}{rcl}
\phi_{0,\tau}^{1}(\Ch(D_{2})) &\longrightarrow &\{\{u,v\} \mbox{, } 0\leq u<v\leq n_{1}+n_{2}+1\} \\
\{u,v\} &\longrightarrow & \begin{cases} 
\{u+n_{1},v+n_{1}\}& \mbox{ if } u\geq 1, \\
\{u,v+n_{1}\}& \mbox{ if } u=0.
\end{cases}
\end{array}
\right.\] 
To have a dissection diagram of degree $n_{1}+n_{2}+1$ we just add the chord $\{0,n_{1}+1\}$.

Let $t$ and $\tau$ be two non negative integers and we define the morphism $\Gamma_{t,\tau}^{1}$ homogeneous of degree $1$ by:
\[\Gamma_{t,\tau}^{1}:\left\{
\begin{array}{rcl}
(\D)_{n_{1}}\ot(\D)_{n_{2}}  &\longrightarrow &(\D)_{n_{1}+n_{2}+1} \\
D_{1}\ot D_{2} &\longrightarrow &  \begin{cases}
G_{D_{1},D_{2},t,\tau}^{0,1} & \mbox{ if } t\leq f_{D_{1}}(0) \text{ and } \tau\leq f_{D_{2}}(0),\\
0 &\mbox{ else,}
\end{cases} 
\end{array}
\right.\] 
with \[\Ch(G_{D_{1},D_{2},t,\tau}^{0,1})=\gamma^{0}_{D_{1},t}(\phi_{0,t}^{0}(\Ch(D_{1})))\cup\gamma^{1}_{D_{2},\tau}(\phi_{0,\tau}^{1}(\Ch(D_{2})))\cup\{\{0,n_{1}+1\}\}.\]

\Ex{1} With $D_{1}=\diagramme{3}{n}{{1/3,2/1,3/0}}{}{}{}{{0}}{}$, $D_{2}=\diagramme{3}{n}{{1/0,2/1,3/0}}{}{}{}{{0}}{}$, $t=1$ and $\tau=1$ we have:
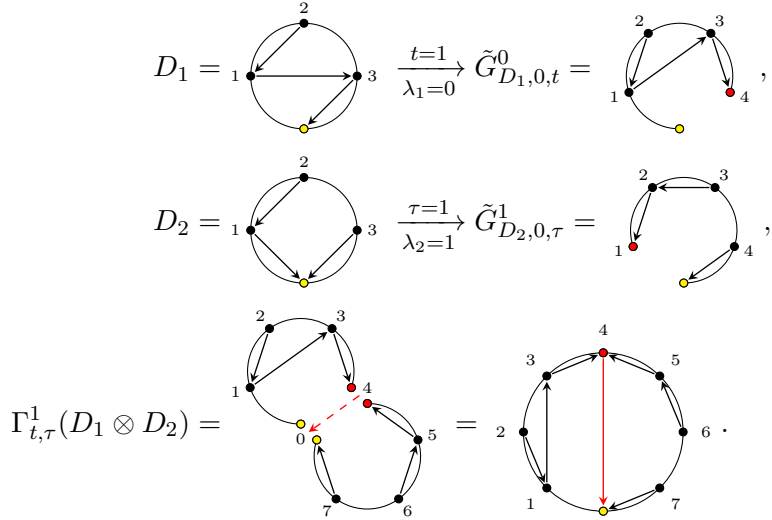
\begin{align*}
D_{1}=&\diagramme{3}{n}{{1/3,2/1,3/0}}{}{}{}{{0}}{}  \xrightarrow[\lambda_{1}=0]{t=1} \tilde{G}_{D_{1},0,t}^{0}=\diagramme{4}{n}{{1/3,2/1,3/4}}{}{}{}{{0,4}}{4}, \\
D_{2}=&\diagramme{3}{n}{{1/0,2/1,3/0}}{}{}{}{{0}}{}  \xrightarrow[\lambda_{2}=1]{\tau=1} \tilde{G}_{D_{2},0,\tau}^{1}=\diagramme{4}{n}{{2/1,3/2,4/0}}{}{}{}{{0,1}}{0}, \\
\Gamma_{t,\tau}^{1}(D_{1}\ot D_{2})=&\colldiag{2}=\diagramme{7}{n}{{1/3,2/1,3/4,5/4,6/5,7/0}}{{4/0}}{}{}{{0,4}}{}.
\end{align*}

\Rq{} Let $D_{1}$ and $D_{2}$ be two dissection diagrams of respective positive degree $n_{1}$ and $n_{2}$ and  $\lambda_{1},\lambda_{2}\in\{0,1\}$, $t\in\llbracket 0,f_{D_{1}}(0)\rrbracket$ and $\tau\in\llbracket 0,f_{D_{2}}(0)\rrbracket$ be integers.
\begin{enumerate}
	\item By direct computation we have: 
	\[\Gamma_{t,\tau}^{0}(D_{1}\ot D_{2})=\Gamma_{\tau,t}^{1}(D_{2}\ot D_{1}) \text{ and }
	\Gamma_{t,\tau}^{1}(D_{1}\ot D_{2})=\Gamma_{\tau,t}^{0}(D_{2}\ot D_{1}).\]
	\item The two cases $(\lambda_{1}=0,\lambda_{2}=0)$ and $(\lambda_{1}=1,\lambda_{2}=1)$ are useless. Indeed, it is sufficient to use the tow maps
	\[\tilde{\gamma}^{0}_{D_{1},t}:\left\{
	\begin{array}{rcl}
	\phi_{0,t}^{0}(\Ch(D_{1})) &\longrightarrow &\{\{u,v\} \mbox{, } 0\leq u<v\leq n_{1}+n_{2}+1\} \\
	\{u,v\} &\longrightarrow & \begin{cases}
	\{u+n_{2}+1,v+n_{2}+1\}& \mbox{si } v\leq n_{1}\\
	\{0,u+n_{2}+1\}& \mbox{si } v=n_{1}+1
	\end{cases}
	\end{array}
	\right.\]
	and
	\[\tilde{\gamma}^{1}_{D_{1},t}:\left\{
	\begin{array}{rcl}
	\phi_{0,t}^{1}(\Ch(D_{1})) &\longrightarrow &\{\{u,v\} \mbox{, } 0\leq u<v\leq n_{1}+n_{2}+1\} \\
	\{u,v\} &\longrightarrow & \begin{cases} 
	\{u-1,v-1\}& \mbox{si } u\geq 1, \\
	\{v-1,n_{1}+1\}& \mbox{si } u=0,
	\end{cases}
	\end{array}
	\right.\] 
	
	and the dissection diagrammes $G_{D_{1},D_{2},t,\tau}^{0,0}$ and $G_{D_{1},D_{2},t,\tau}^{1,1}$ of degree $n_{1}+n_{2}+1$ with
	\begin{align*}
	\Ch(G_{D_{1},D_{2},t,\tau}^{0,0})&=\tilde{\gamma}^{0}_{D_{1},t}(\phi_{0,t}^{0}(\Ch(D_{1})))\cup\gamma^{0}_{D_{2},\tau}(\phi_{0,\tau}^{0}(\Ch(D_{2})))\cup\{\{0,n_{2}+1\}\},\\ \Ch(G_{D_{1},D_{2},t,\tau}^{1,1})&=\tilde{\gamma}^{1}_{D_{1},t}(\phi_{0,t}^{1}(\Ch(D_{1})))\cup\gamma^{1}_{D_{2},\tau}(\phi_{0,\tau}^{1}(\Ch(D_{2})))\cup\{\{0,n_{1}+1\}\}.
	\end{align*}
	The we have: $G_{D_{1},D_{2},t,\tau}^{0,0}=G_{D_{1},D_{2},t,\tau}^{1,0}$ and $G_{D_{1},D_{2},t,\tau}^{1,1}=G_{D_{1},D_{2},t,\tau}^{0,1}$.
	\item The sum of maps of type $\Gamma_{t,\tau}^{\lambda}$ is called operation 4.
\end{enumerate}

\subsubsection{Computation of $L(D_{1}D_{2})$ where $(D_{1},D_{2})\in(\D)_{n_{1}}\times(\D)_{n_{2}}$ and $n_{1},n_{2}\geq1$.}

\begin{prop}\label{insertionunb}
	Let $D_{1}$ and $D_{2}$ two dissection diagrams of positive degree $n_{1}$ and $n_{2}$, $i$ in $\llbracket 1,n_{1}\rrbracket$, $\varrho$ in $\llbracket 0,f_{D_{1}}(i)\rrbracket$, $p$ in $\llbracket 0,f_{D_{2}}(0)\rrbracket$, $j$, $j_{1}$ and $j_{2}$ in $\llbracket 1,n_{2}\rrbracket$, $\tau$, respectively $\tau_{1}$, respectively $\tau_{2}$ in $\llbracket 0,f_{D_{2}}(j)\rrbracket$, respectivement $\llbracket 0,f_{D_{2}}(j_{1})\rrbracket$, respectively $\llbracket 0,f_{D_{2}}(j_{2})\rrbracket$, $\omega$, $\omega_{1}$ and $\omega_{2}$ in $\llbracket0,f_{D_{2}}(0)\rrbracket$, $t$, $t_{1}$ and $t_{2}$ in $\llbracket 0,f_{D_{1}}(0)\rrbracket$) in  $\lambda$, $\lambda_{1}$, $\lambda_{2}$ in $\{0,1\}$ be integers.
	
	We assume first that $D_{1}$ and $D_{2}$ are two different dissection diagrams. We have the following statements:
	\begin{enumerate}
		\item $(\Gamma_{t_{1},j_{1},\tau_{1}}^{1}(D_{1}\ot D_{2}),\{j_{1},j_{1}+n_{1}+1\})=(\Gamma_{t_{2},j_{2},\tau_{2}}^{1}(D_{1}\ot D_{2}),\{j_{2},j_{2}+n_{1}+1\})$ is equivalent to $(j_{1}=j_{2},~ t_{1}=t_{2} \text{ and } \tau_{1}=\tau_{2})$,
		\item $(\Gamma_{t_{1},\omega_{1}}^{\lambda_{1}}(D_{1}\ot D_{2}),\{0,\lambda_{1}n_{1}+(1-\lambda_{1})n_{2}+1\})=(\Gamma_{t_{2},\omega_{2}}^{\lambda_{2}}(D_{1}\ot D_{2}),\{0,\lambda_{2}n_{1}+(1-\lambda_{2})n_{2}+1\})$ is equivalent to $(t_{1}=t_{2},~ \omega_{1}=\omega_{2} \text{ and } \lambda_{1}=\lambda_{2})$,
		\item $(\Gamma_{t_{1},j_{1},\tau_{1}}^{1}(D_{1}\ot D_{2}),\{j,j+n_{1}+1\})\neq(\Gamma_{t_{2},\omega_{2}}^{\lambda}(D_{1}\ot D_{2}),\{0,\lambda_{2}n_{1}+(1-\lambda_{2})n_{2}+1\})$,
		\item $(\Gamma_{t,j,\tau}^{1}(D_{1}\ot D_{2}),\{j,j+n_{1}+1\})\neq(\Gamma_{p,i,\varrho}^{1}(D_{2}\ot D_{1}),\{i,i+n_{2}+1\})$.
	\end{enumerate}
	
	We assume that $D_{1}$ equals $D_{2}$ and we denote $D=D_{1}=D_{2}$. We have the following statements:
	\begin{enumerate}[resume]
		\item $(\Gamma_{t_{1},j_{1},\tau_{1}}^{1}(D\ot D),\{j_{1},j_{1}+n_{1}+1\})=(\Gamma_{t_{2},j_{2},\tau_{2}}^{1}(D\ot D),\{j_{2},j_{2}+n_{1}+1\})$ is equivalent to $(j_{1}=j_{2},~ t_{1}=t_{2} \text{ and } \tau_{1}=\tau_{2})$,
		\item $(\Gamma_{t,\omega}^{\lambda}(D\ot D),\{0,n_{1}+1\})=(\Gamma_{\omega,t}^{1-\lambda}(D\ot D),\{0,n_{1}+1\})$,
		\item $(\Gamma_{t_{1},j_{1},\tau_{1}}^{1}(D\ot D),\{j,j+n_{1}+1\})\neq(\Gamma_{t_{2},\omega_{2}}^{\lambda}(D\ot D),\{0,n_{1}+1\})$.
	\end{enumerate}	
\end{prop}

\begin{proof}
	We assume that $D_{1}$ and $D_{2}$ are two different dissection diagrams.
	\begin{enumerate}
		\item\label{Op3unicite} It is sufficient to prove:
		\[\Gamma_{t_{1},j,\tau_{1}}^{1}(D_{1}\ot D_{2})=\Gamma_{t_{2},j,\tau_{2}}^{1}(D_{1}\ot D_{2})\implies (t_{1}=t_{2} \text{ and } \tau_{1}=\tau_{2}).\]
		We assume $K_{1}=\Gamma_{t_{1},j,\tau_{1}}^{1}(D_{1}\ot D_{2})$ equals $K_{2}=\Gamma_{t_{2},j,\tau_{2}}^{1}(D_{1}\ot D_{2})$. We denote by $p_{1}(j)$ (respectively by $p_{2}(j)$) the vertex $j$ of the dissection diagram $K_{1}$ (respectively of the dissection diagram $K_{2}$)
		and we consider its fertility $f_{1}(j)$ (respectively $f_{2}(j)$). We have:
		\[f_{1}(j)=f_{D_{1}}(0)-t_{1}+\tau_{1}+1=f_{D_{1}}(0)-t_{2}+\tau_{2}+1=f_{2}(j).\]
		We now use the subset of chords $A_{K_{1},j}$ (respectively $A_{K_{2},j}$) and the chord $a=\{j,j+n_{1}+1\}$ too. We know that $a=a_{K_{1},\tau_{1}+1}^{j}=a_{K_{2},\tau_{2}+1}^{j}$ so we have $t_{1}=t_{2}$ and $\tau_{1}=\tau_{2}$.
		\item We assume that the two disssection diagrams $\Gamma_{t_{1},\omega_{1}}^{\lambda_{1}}(D_{1}\ot D_{2})$ and  $\Gamma_{t_{2},\omega_{2}}^{\lambda_{2}}(D_{1}\ot D_{2})$ are equal, and that the two chords $\{0,\lambda_{1}n_{1}+(1-\lambda_{1})n_{2}+1\}$ and $\{0,\lambda_{2}n_{1}+(1-\lambda_{2})n_{2}+1\}$ are equal too, but we assume $\lambda_{1}$ and $\lambda_{2}$ different. We obtain then $n_{1}=n_{2}$. Without any lack of generality we can use the case $\lambda_{1}=0$ and $\lambda_{2}=1$. We denote $K_{1}=\Gamma_{t_{1},\omega_{1}}^{\lambda_{1}}(D_{1}\ot D_{2})$ and $K_{2}=\Gamma_{t_{2},\omega_{2}}^{\lambda_{2}}(D_{1}\ot D_{2})$, we consider the subgraph $K_{1}$ (respectively the subgraph $K_{2}$) created with vertices from the set $\{1,\dots,n_{1}+1\}$ and we contract $\{0,n_{1}+1\}$. As $K_{1}$ and $K_{2}$ are equal, the obtained dissection diagrams are equal too. And yet, with $K_{1}$  (respectively $K_{2}$) we obtain $D_{2}$ (respectively $D_{1}$) so $D_{1}=D_{2}$. We can conclude that $\lambda_{1}$ and $\lambda_{2}$ are equal. Without any lack of generality, we assume $\lambda_{1}=\lambda_{2}=0$. Let $p_{1}$ (respectively $p_{2}$) be the root of $K_{1}$ (respectively $K_{2}$). We consider their fertility $f_{1}$ (respectively $f_{2}$). We have: 
		\[f_{1}=f_{D_{2}}(0)-\omega_{1}+t_{1}+1=f_{D_{2}}(0)-\omega_{2}+t_{2}+1=f_{2}.\]
		We now use the subset of chords  $A_{K_{1},0}$ (respectively $A_{K_{2},0}$) and the chord $a=\{0,n_{2}+1\}$ too. We have $a=a_{K_{1},t_{1}+1}^{0}=a_{K_{2},t_{2}+1}^{0}$ so $t_{1}=t_{2}$ and $\omega_{1}=\omega_{2}$.
		\item As $j$ is positive, this point is trivial.
		\item We assume there exists integers $j\in\llbracket1,n_{2}\rrbracket$ (respectively $i\in\llbracket1,n_{1}\rrbracket$), $t\in\llbracket0,f_{D_{1}}(0)\rrbracket$ and $\tau\in\llbracket0,f_{D_{2}}(j)\rrbracket$ (respectively $p\in\llbracket 0,f_{D_{2}}(0)\rrbracket$ and $\varrho\in\llbracket 0,f_{D_{1}}(i)\rrbracket$) such that 
		\[(\Gamma_{t,j,\tau}^{1}(D_{1}\ot D_{2}),\{j,j+n_{1}+1\})=(\Gamma_{p,i,\varrho}^{1}(D_{2}\ot D_{1}),\{i,i+n_{2}+1\}).\]
		We obtain $i=j$ and $n_{1}=n_{2}$. We denote $K_{1}=\Gamma_{t,j,\tau}^{1}(D_{1}\ot D_{2})$ and $K_{2}=\Gamma_{p,i,\varrho}^{1}(D_{2}\ot D_{1})$, we consider the subgraph of $K_{1}$ (respectively $K_{2}$) formed by vertices in $\{1,\dots,j\}\cup\{j+n_{1}+1,\dots,j+2n_{1}+1\}$ and we contract $\{j,j+n_{1}+1\}$. We obtain then $D_{1}=D_{2}$.
	\end{enumerate}
	
	We assume now $D=D_{1}=D_{2}$.
	\begin{enumerate}[resume]
		\item It is sufficient to consider the proof of \ref{Op3unicite}.
		\item It is true by definition.
		\item It is trivial because $j$ is positive.
	\end{enumerate}
\end{proof}

\begin{cor}
	Let $D_{1}$ and $D_{2}$ be two dissection diagrams of positive degree $n_{1}$ and $n_{2}$. We define $\sigma_{D}$ as the numbers of different couples $(G,a)$ where $G$ is a dissection diagram of degree $n_{1}+n_{2}+1$ and $a$ a chord $a=\{u,v\}$ with $0\leq u < v\leq n_{1}+n_{2}+1$ obtained by applying operations 3 and 4 on $D_{1}$ and $D_{2}$. We have:
	\begin{align*}
		\sigma(D_{1},D_{2})&=\begin{cases} 
			(f_{D_{2}}(0)+1)\sum\limits_{i=0}^{n_{1}}(f_{D_{1}}(i)+1)+(f_{D_{1}}(0)+1)\sum\limits_{j=0}^{n_{2}}(f_{D_{2}}(j)+1)& \mbox{if } D_{1}\neq D_{2}, \\
			(f_{D_{2}}(0)+1)\sum\limits_{i=0}^{n_{1}}(f_{D_{1}}(i)+1)& \mbox{if } D_{1}=D_{2},
		\end{cases} \\
		&\mbox{ }=\begin{cases} 
			(f_{D_{2}}(0)+1)(3n_{1}+1)+(f_{D_{1}}(0)+1)(3n_{2}+1)& \mbox{if } D_{1}\neq D_{2}, \\
			(f_{D_{1}}(0)+1)(3n_{1}+1)& \mbox{if } D_{1}=D_{2}.
		\end{cases}
	\end{align*}
\end{cor}
\begin{proof}
	It is a direct computation.
\end{proof}

\begin{prop}\label{valLdeuxdiagrammes}
	Let $D_{1}$ and $D_{2}$ be two dissection diagrams of positive degree $n_{1}$ and $n_{2}$.	 
	\begin{align*}
		L(D_{1}D_{2})=&\sum\limits_{\mathclap{\substack{t\in\llbracket0,f_{D_{1}}(0)\rrbracket \\ j\in\llbracket1,n_{2}\rrbracket \\ \tau\in\llbracket0,f_{D_{2}}(j)\rrbracket}}}x^{k_{\{a_{j,n_{1}}\}}(\Gamma_{t,j,\tau}^{1}(D_{1}\ot D_{2}))}\Gamma_{t,j,\tau}^{1}(D_{1}\ot D_{2})\\
		+&\sum\limits_{\mathclap{\substack{\tau\in\llbracket0,f_{D_{2}}(0)\rrbracket \\ i\in\llbracket1,n_{1}\rrbracket \\ t\in\llbracket0,f_{D_{1}}(i)\rrbracket}}}x^{k_{\{a_{i,n_{2}}\}}(\Gamma_{\tau,i,t}^{1}(D_{2}\ot D_{1}))}\Gamma_{\tau,i,t}^{1}(D_{2}\ot D_{1})
		+\sum\limits_{\mathclap{\substack{\lambda\in\{0,1\} \\ t\in\llbracket0,f_{D_{1}}(0)\rrbracket \\ \tau\in\llbracket0,f_{D_{2}}(0)\rrbracket}}}\Gamma_{t,\tau}^{\lambda}(D_{1}\ot D_{2})
	\end{align*}
	with
	\[a_{j,n_{1}}=\{j,j+n_{1}+1\} \text{ and } a_{i,n_{2}}=\{i,i+n_{2}+1\}.\]

	In other words, 
	\begin{align*}
		L(D_{1}D_{2})=&\sum\limits_{\mathclap{\substack{t\in\llbracket0,f_{D_{1}}(0)\rrbracket \\ j\in\llbracket1,n_{2}\rrbracket \\ \tau\in\llbracket0,l_{2}(j)-1\rrbracket}}}\Gamma_{t,j,\tau}^{1}(D_{1}\ot D_{2})+x\sum\limits_{\mathclap{\substack{t\in\llbracket0,f_{D_{1}}(0)\rrbracket \\ j\in\llbracket1,n_{2}\rrbracket \\ \tau\in\llbracket l_{2}(j),f_{D_{2}}(j)\rrbracket}}}\Gamma_{t,j,\tau}^{1}(D_{1}\ot D_{2})\\
		+&\sum\limits_{\mathclap{\substack{\tau\in\rrbracket0,f_{D_{2}}(0)\rrbracket \\ i\in\llbracket1,n_{1}\rrbracket \\ t\in\llbracket0,l_{1}(i)-1\rrbracket}}}\Gamma_{\tau,i,t}^{1}(D_{2}\ot D_{1})+x\sum\limits_{\mathclap{\substack{\tau\in\llbracket0,f_{D_{2}}(0)\rrbracket \\ i\in\llbracket1,n_{1}\rrbracket \\ t\in\llbracket l_{1}(i),f_{D_{1}}(i)\rrbracket}}}\Gamma_{\tau,i,t}^{1}(D_{2}\ot D_{1})+\sum\limits_{\mathclap{\substack{\lambda\in\{0,1\} \\ t\in\llbracket0,f_{D_{1}}(0)\rrbracket \\ \tau\in\llbracket0,f_{D_{2}}(0)\rrbracket}}}\Gamma_{t,\tau}^{\lambda}(D_{1}\ot D_{2}) 
	\end{align*}
	where for all couples $(i,j)\in\llbracket1,n_{1} \rrbracket\times\llbracket1,n_{2} \rrbracket$, $(l_{1}(i),l_{2}(j))$ is the unique couple of integers in  $\llbracket0,f_{D_{1}}(i)\rrbracket\times\llbracket0,f_{D_{2}}(j)\rrbracket$ such that $a_{D_{1},l_{1}(i)}^{i}$ in $\Ch(D_{1})$ is labeled by $i$ and $a_{D_{2},l_{2}(j)}^{j}$ in $\Ch(D_{2})$ is labeled by $j$.
\end{prop}

\begin{proof} We recall that, if $Id$ is the identity morphism and $l$ is the projection on $\D^{+}$, we have:
	\ud{0.5}
%	\begin{align*}
%		L(D_{1}D_{2})=&\sum\limits_{\substack{G\in\D \\ G \text{ diagram}}}\bigg(Z_{D_{1}D_{2}\circ \diagramme{1}{}{{1/0}}{}{}{}{}{}}\bigg)(G)G\\
%		=&\sum\limits_{\substack{G\in\D \\ G \text{ diagram}}}\bigg(Z_{D_{1}\circ (D_{2}\circ \diagramme{1}{}{{1/0}}{}{}{}{}{})}-Z_{(D_{1}\circ D_{2})\circ \diagramme{1}{}{{1/0}}{}{}{}{}{}}\bigg)(G)G\\
%		=&\sum\limits_{\substack{G\in\D \\ G \text{ diagram}}}\bigg((Z_{D_{1}}\ot Z_{D_{2}}\ot Z_{\diagramme{1}{}{{1/0}}{}{}{}{}{}}) \circ(Id\ot \De)\circ\De\bigg)(G)G\\
%		-&\sum\limits_{\substack{G\in\D \\ G \text{ diagram}}}\bigg((Z_{D_{1}}\ot Z_{D_{2}}\ot Z_{\diagramme{1}{}{{1/0}}{}{}{}{}{}}) \circ(\De\ot Id)\circ(l\ot Id)\circ\De\bigg)(G)G\\
%		=&\begin{cases}
%			\sum\limits_{\substack{G\in\D \\ G \text{ diagram}}}\bigg(Z_{D_{1}D_{2}}\ot Z_{\diagramme{1}{}{{1/0}}{}{}{}{}{}}\bigg)\bigg(\De(G)\bigg)G& \mbox{if } D_{1}\neq D_{2},\\ 
%			2\sum\limits_{\substack{G\in\D \\ G \text{ diagram}}}\bigg(Z_{D_{1}D_{2}}\ot Z_{\diagramme{1}{}{{1/0}}{}{}{}{}{}}\bigg)\bigg(\De(G)\bigg)G& \mbox{else.}
%		\end{cases}
%	\end{align*}
	\begin{align*}
			L(D_{1}D_{2})=&\begin{cases}
				\sum\limits_{\substack{G\in\D \\ G \text{ diagram}}}\bigg(Z_{D_{1}D_{2}}\ot Z_{\diagramme{1}{}{{1/0}}{}{}{}{}{}}\bigg)\bigg(\De(G)\bigg)G& \mbox{if } D_{1}\neq D_{2},\\ 
				2\sum\limits_{\substack{G\in\D \\ G \text{ diagram}}}\bigg(Z_{D_{1}D_{2}}\ot Z_{\diagramme{1}{}{{1/0}}{}{}{}{}{}}\bigg)\bigg(\De(G)\bigg)G& \mbox{else.}
			\end{cases}
	\end{align*}	
	\ud{0.7}
	
	Let $(G,a)\in (\D)_{n_{1}+n_{2}+1}\times\Ch(G)$ be a couple diagram-chord such that $q_{\{a\}}(G)=D_{1}D_{2}$. We write $a$ as $a=\{i,j\}$ with $0\leq i<j\leq n_{1}+n_{2}+1$. There exists $\nu\in\llbracket1,f_{G}(i)\rrbracket$ such that $a^{i}_{\nu}=\{i,j\}$. $q_{a}(G)=D_{1}D_{2}$ so $j=i+n+1$ with $n\in\{n_{1},n_{2}\}$.\\
	\ul{Case 1 : $i\neq 0$.} We consider the subgraph $S$ of $G$ with vertices \[\{0,\dots,i\}\cup\{j,\dots,n_{1}+n_{2}+1\}.\]
	We have:
	\begin{align*}
		(&G,a)=\\
		&\begin{cases}(\Gamma_{t,i,\tau}^{1}(D_{1}\ot D_{2}),\{i,i+n_{1}+1\}) \mbox{ with } t=f_{D_{1}}(0)-f_{G}(i)+\nu \mbox{ and } \tau=\nu-1  & \mbox{if } q_{\{a\}}(S)=D_{2},\\
			(\Gamma_{\tau,i,t,}^{1}(D_{2}\ot D_{1}),\{i,i+n_{2}+1\}) \mbox{ with } t=\nu-1 \mbox{ and } \tau=f_{D_{2}}(0)-f_{G}(i)+\nu& \mbox{if } q_{\{a\}}(S)=D_{1}. \end{cases}
	\end{align*}
	\ul{Case 2 : $i=0$.} We use the subgraph $S$ of $G$ with vertices $\{0,\dots,j\}$. We have:
	\begin{align*}
		(&G,a)=\\
		&\begin{cases}
			(\Gamma_{t,\tau}^{1}(D_{1}\ot ,D_{2}),\{0,n_{2}+1\}) \mbox{ with } t=f_{D_{1}}(0)-f_{G}(0)+\nu \mbox{ and } \tau=\nu-1& \mbox{if } q_{\{a\}}(S)=D_{2},\\
			(\Gamma_{t,\tau}^{0}(D_{1}\ot D_{2}),\{0,n_{1}+1\}) \mbox{ with } t=\nu-1 \mbox{ and } \tau=f_{D_{2}}(0)-f_{G}(0)+\nu& \mbox{if } q_{\{a\}}(S)=D_{1}. \end{cases}
	\end{align*} 
\end{proof}

\begin{cor}
	We assume that $x\in\N$. Let $D_{1}$ and $D_{2}$ be two dissection diagrams of positive degree $n_{1}$ and $n_{2}$. We call $\sigma_{L(D_{1}D_{2})}$ the number of terms in $L(D_{1}D_{2})$ counted with multiplicity .
	\begin{align*}
		\sigma_{L(D_{1}D_{2})}=&(f_{D_{2}}(0)+1)\left[x(3n_{1} + 1) + (1-x)\sum\limits_{i=1}^{n_{1}}l_{1}(i)+(1-x)(f_{D_{1}}(0)+1)\right]\\
		+&(f_{D_{1}}(0)+1)\left[x(3n_{2} + 1) + (1-x)\sum\limits_{j=1}^{n_{2}}l_{2}(j)+(1-x)(f_{D_{2}}(0)+1)\right]
	\end{align*}
	where, for each $(i,j)\in\llbracket1,n_{1} \rrbracket\times\llbracket1,n_{2} \rrbracket$, $(l_{1}(i),l_{2}(j))$ is the unique element in  $\llbracket0,f_{D_{1}}(i)\rrbracket\times\llbracket0,f_{D_{2}}(j)\rrbracket$ such that the chord $a_{D_{1},l_{1}(i)}^{i}$ of $D_{1}$ is labeled by $i$ and the chord $a_{D_{2},l_{2}(j)}^{j}$ of $D_{2}$ is labeled by $j$.
\end{cor}

\begin{prop}\label{etapedeux}
	Let $x$ be a scalar. Let $D$, $D_{1}$, $D_{2}$, $G_{1}$ and $G_{2}$ be five nonempty dissection diagrams. Then,
	\begin{enumerate}
		\item $L(D)\neq L(G_{1}G_{2})$,
		\item $L(D_{1}D_{2})= L(G_{1}G_{2})\iff \left((D_{1},D_{2})=(G_{1},G_{2}) \text{ or } (D_{1},D_{2})=(G_{2},G_{1})\right)$.
	\end{enumerate} 
\end{prop}

\ud{0.7}
\begin{proof}
	Let $x$ be a scalar in $\K$. Let $D$, $D_{1}$, $D_{2}$, $G_{1}$ and $G_{2}$ five dissection diagrams. We consider their root fertility.  As in the proof of proposition \ref{etapeun}, if $f_{D}(0)\neq f_{G_{1}}(0)+f_{G_{2}}(0)$ and if  $f_{D_{1}}(0)+f_{D_{2}}(0)\neq f_{G_{1}}(0)+f_{G_{2}}(0)$ then $L(D)\neq L(G_{1}G_{2})$ and $L(D_{1}D_{2})\neq L(G_{1}G_{2})$. We consider the case where  $f_{D}(0)=f_{G_{1}}(0)+f_{G_{2}}(0)$ and $f_{D_{1}}(0)+f_{D_{2}}(0)=f_{G_{1}}(0)+f_{G_{2}}(0)$. We write the dissection diagrams as  $D=\defDiag{{6/0}}{{2/6/A,10/6/B,9/0/C}}$, $D_{1}=\defDiag{{6/0}}{{2/6/A_{1},10/6/B_{1},9/0/C_{1}}}$, $D_{2}=\defDiag{{6/0}}{{2/6/A_{2},10/6/B_{2},9/0/C_{2}}}$ $G_{1}=\defDiag{{6/0}}{{2/6/\alpha_{1},10/6/\beta_{1},9/0/\gamma_{1}}}$ and $G_{2}=\defDiag{{6/0}}{{2/6/\alpha_{2},10/6/\beta_{2},9/0/\gamma_{2}}}$. We will use $T_{1}=\defDiag{{4/0}}{{9/0/D}}$, $T_{2}=\defDiag{{9/0}}{{4/0/D}}$, $S_{1}=\defDiag{{13/0}}{{4/0/D_{1},9/0/D_{2}}}$, $S_{2}=\defDiag{{13/0}}{{4/0/D_{2},9/0/D_{1}}}$, $P_{1}=\defDiag{{13/0}}{{4/0/G_{1},9/0/G_{2}}}$, $P_{2}=\defDiag{{13/0}}{{4/0/G_{2},9/0/G_{1}}}$, $H_{1}=\defDiag{{13/0}}{{3/13/D_{1},10/13/D_{2}}}$, $H_{2}=\defDiag{{13/0}}{{3/13/D_{2},10/13/D_{1}}}$, $K_{1}=\defDiag{{13/0}}{{3/13/G_{1},10/13/G_{2}}}$, $K_{2}=\defDiag{{13/0}}{{3/13/G_{2},10/13/G_{1}}}$. We recall that $X_{n}$ is the corolla of degree $n$.
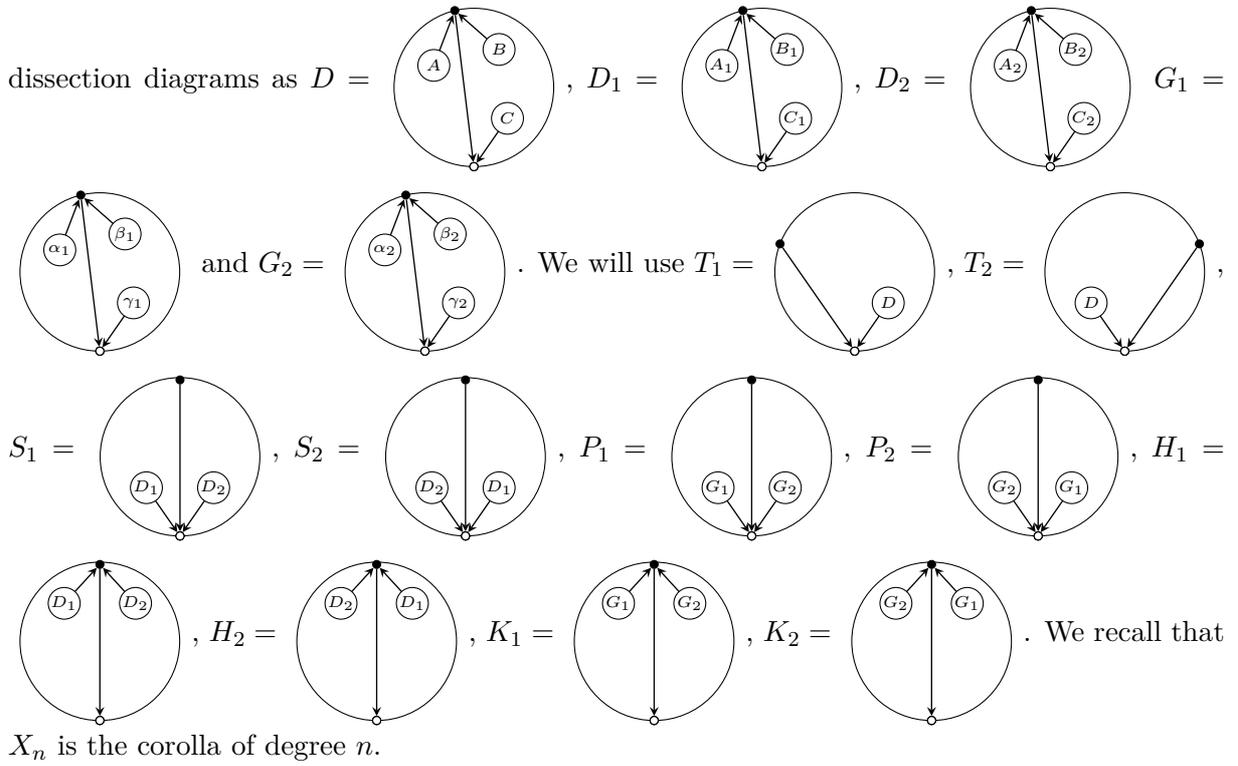
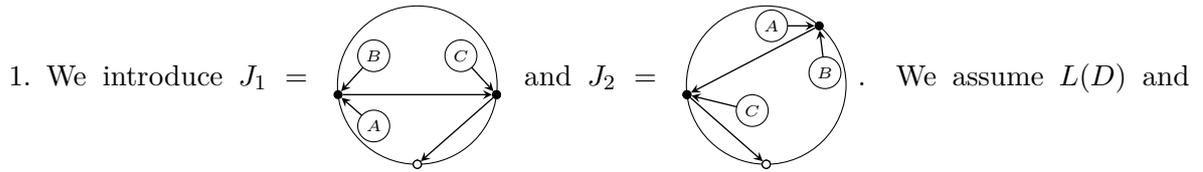
	\begin{enumerate}
		\item We introduce $J_{1}=\defDiag{{10/0,3/10}}{{0/3/A,13/3/B,13/10/C}}$ and $J_{2}=\defDiag{{3/0,8/3}}{{5/8/A,11/8/B,11/3/C}}$. We assume $L(D)$ and $L(G_{1}G_{2})$ equal so $f_{D}(0)\geq2$. There are two cases:
		\begin{enumerate}
			\item If $f_{G_{1}}(0)$ and $f_{G_{2}}(0)$ are both different from $1$, then there exists an integer $i\in\{1,2\}$ such that $J_{1}=K_{i}$. It is impossible.
			\item We assume now there exists at least an integer $i\in\{1,2\}$ such that $f_{G_{i}}(0)=1$. Without any lack of generality, we can consider $i=1$ so $\gamma_{1}$ is empty. We use know the fact that $\{T_{1},T_{2}\}$ and $\{P_{1},P_{2}\}$ be equal. We have two cases.
			\begin{enumerate}
				\item In the first one, $(T_{1},T_{2})$ equals $(P_{1},P_{2})$ so $\alpha_{1}$ and $\beta_{1}$ are empty, $G_{1}=X_{1}$, $A=\alpha_{2}$, $B=\beta_{2}$ and we have more 
				\[\defDiag{{5/0,11/0}}{{2/5/A,8/5/B,9/0/\gamma_{2}}}=D=\defDiag{{13/0,3/0}}{{3/13/A,10/13/B,10/0/\gamma_{2}}}.\]
				Thereafter $A$ and $B$ are empty and we have the equality \[\defDiag{{9/0}}{{4/0/\gamma_{2}}}=\defDiag{{4/0}}{{9/0/\gamma_{2}}}.\]
				So, there exists a non negative integer $n$ such that $\gamma_{2}=X_{n}$ and then $D=X_{n+2}$, $G_{1}=X_{1}$ and $X_{n+1}$. It is impossible.
				\item In the second one, $(T_{1},T_{2})$ equals $(P_{2},P_{1})$ so $\alpha_{2}$ and $\beta_{2}$ are empty, $A=\alpha_{1}$, $B=\beta_{1}$ and we have two equalities 
				\[D=\defDiag{{5/0,8/0}}{{5/8/A,11/8/B,3/0/\gamma_{2}}} \text{ and } \defDiag{{9/0}}{{4/0/C}}=\defDiag{{13/0,3/0}}{{10/0/\gamma_{2}}}.\]
				So, we obtain \[\defDiag{{5/0,11/0,8/0}}{{3/0/\gamma_{2},5/8/A,11/8/B}}=T_{2}=\defDiag{{5/0,7/0,8/0}}{{2/5/A,8/5/B,10/0/\gamma_{2}}}.\] If $\gamma_{2}$ is empty then $A$ and $B$ are empty too and we obtain $D=X_{2}$ and $G_{1}=G_{2}=X_{1}$. It is impossible. The diagram $\gamma_{2}$ is not empty.  As for all integers $i$ and $j$ in $\{1,2\}$, the dissection diagrams $J_{i}$ and $K_{j}$ are differents, there exists two nonempty dissection diagrams $W_{1}$ and $W_{2}$ such that $J_{1}=\defDiag{{13/0}}{{3/13/W_{1},10/13/B}}$ and $J_{2}=\defDiag{{13/0}}{{3/13/A,10/13/W_{2}}}$. Then $A$ and $B$ are empty and we obtain that $D$, $G_{1}$ and $G_{2}$ are corollas. It is impossible.
			\end{enumerate}
		\end{enumerate}	
		Thus, $L(D)$ and $L(G_{1}G_{2})$ are different.
		
		\item We assume $L(D_{1}D_{2})$ and $L(G_{1}G_{2})$ equal. We have to consider three cases.
		\begin{enumerate}
			\item In the first case, $f_{D_{1}}(0)$, $f_{D_{2}}(0)$, $f_{G_{1}}(0)$ and $f_{G_{2}}(0)$ are all equal to $1$. As the two sets  $\{S_{1},S_{2}\}$ and $\{P_{1},P_{2}\}$ are equal the result is trivial.
			\item In the second one, $f_{D_{1}}(0)$, $f_{D_{2}}(0)$, $f_{G_{1}}(0)$ and $f_{G_{2}}(0)$ are all positive integers different from $1$. As the two sets  $\{H_{1},H_{2}\}$ and $\{K_{1},K_{2}\}$ are equal the result is trivial.
			\item In the third one, $f_{D_{1}}(0)=1$ and $f_{D_{2}}(0)\neq1$. The sets $\{S_{1},S_{2}\}$ and $\{P_{1},P_{2}\}$ are equal so we assume without any lack of generality $S_{1}=P_{1}$ and $S_{2}=P_{2}$. For all $i\in\{1,2\}$, we have $A_{i}=\alpha_{i}$ and $B_{i}=\beta_{i}$. 
			
			We start by proving there exists an unique integer $i\in\{1,2\}$ such that $f_{G_{i}}(0)=1$. We just have to prove the existence.
			We assume $f_{G_{1}}(0)$ and $f_{G_{2}}(0)$ different from $1$. We consider $\Omega_{1}$ (respectively $\Omega_{2}$) the set of dissection diagrams of fertility $f_{D_{1}}(0)+f_{D_{2}}(0)$ in the root built to obtain $L(D_{1}D_{2})$ (respectively  $L(G_{1}G_{2})$). By writing $\gamma_{2}=\defDiag{{}}{{4/0/\tilde{\gamma_{2}},9/0/\overline{\gamma_{2}}}}$ with $\overline{\gamma_{2}}$ of fertility $1$ in the root we have two different possiblities; either
			\ud{0.5}
			$\Omega_{1,1}=\InsDeuxDiag{{3/8,8/0,11/0}}{{0/3/A_{1},6/3/B_{1},8/11/A_{2},14/11/B_{2},13/0/C_{2}}}$ equals
			$\Omega_{2,1}=\InsDeuxDiag{{5/0,7/0,12/7}}{{2/5/\alpha_{2},8/5/\beta_{2},6/0/\gamma_{2},9/12/\alpha_{1},0/12/\beta_{1},12/0/\gamma_{1}}}$ either  $\Omega_{2,2}=\InsDeuxDiag{{5/0,8/0,11/0}}{{2/5/\alpha_{2},8/5/\beta_{2},7/0/\tilde{\gamma_{2}},5/8/\overline{\gamma_{2}},8/11/\alpha_{1},14/11/\beta_{1},13/0/\gamma_{1}}}$. \ud{0.7} So, $B_{2}=\beta_{2}$ is empty and $A_{2}(=\alpha_{2})$ and $D_{1}$ have the same degree. We consider $K_{1}$ and $K_{2}$. If there exists $i\in\{1,2\}$ such that $K_{1}=H_{i}$ or $K_{2}=H_{i}$ the one of the two couples $(G_{1},G_{2})$ and $(G_{2},G_{1})$ equals $(D_{i},D_{-i+3})$. Is is not consistant. Thus there exists a nonemty dissection diagram $W$ such that $K_{1}=\defDiag{{13/0}}{{3/13/W,10/13/B_{1}}}$ and we obtain \[\deg(D_{1})>\deg(B_{1})=\deg(G_{2})>\deg(\alpha_{2})+1>\deg(D_{1})\]  which is impossible.
			
			Now we know that there exists an unique integer $i\in\{1,2\}$ such that $f_{G_{1}}(0)=1$. If $i=1$ then the result is trivial. If $i=2$, then we have  	
			\ud{0.5}
			$\Omega_{1,1}=\InsDeuxDiag{{3/8,8/0,11/0}}{{0/3/A_{1},6/3/B_{1},8/11/A_{2},14/11/B_{2},13/0/C_{2}}}$ equals 
			$\Omega_{2,1}=\InsDeuxDiag{{5/0,7/0,12/7}}{{2/5/\alpha_{2},8/5/\beta_{2},9/12/\alpha_{1},0/12/\beta_{1},12/0/\gamma_{1}}}$ or  $\Omega_{2,2}=\InsDeuxDiag{{3/8,8/0,11/0}}{{0/3/\alpha_{2},6/3/\beta_{2},8/11/\alpha_{1},14/11/\beta_{1},13/0/\gamma_{1}}}$.
			Furthermore $f_{\gamma_{1}}(0)=f_{C_{2}}(0)$ so $\gamma_{1}$ and $C_{2}$ are equal. If $\Omega_{1,1}=\Omega_{2,2}$ then $A_{1}=A_{2}=\alpha_{1}=\alpha_{2}$ and $B_{1}=B_{2}=\beta_{1}=\beta_{2}$ so $(D_{1},D_{2})=(G_{2},G_{1})$. If $\Omega_{1,1}\neq\Omega_{2,2}$ then $\Omega_{1,1}=\Omega_{2,1}$ and $\Omega_{2,2}=\InsDeuxDiag{{5/0,7/0,12/7}}{{2/5/A_{1},8/5/B_{1},9/12/A_{2},0/12/B_{2},12/0/C_{2}}}$ which is impossible.
			
			We have proved the statement \[L(D_{1}D_{2})= L(G_{1}G_{2})\iff \left((D_{1},D_{2})=(G_{1},G_{2}) \text{ or } (D_{1},D_{2})=(G_{2},G_{1})\right).\]
		\end{enumerate}
	\end{enumerate}
\end{proof}
\ud{0.7}

\begin{prop}\label{triangulaire}
	Let $n$ be a positive integer. There exists a basis $\mathcal{B}_{1}$ of sub-binary forests and a basis $\mathcal{B}_{2}$ of the dissection diagrams algebra such that the matrix in those basis of the restriction of $\varphi$ on homogeneous elements of degree $n$ is triangular by blocks.
\end{prop}
\begin{proof}
	Let $n$ be a positive integer and $t$ a sub-binary tree of degree $n$. We call $m(t)$ the number of vertices of $t$ with two children. For each disjoint union of dissection diagramme $U=D_{1}\dots D_{k}$ the integer $k$ is called length of $U$ and is denoted by $\mu(U)$. Let $D$ be a dissection diagram of degree $n$. We call $m(D)=\max\{\mu(q_{C}(D)),~C\in\Ch(D)\}$. By definition of $L$, if $m(D)<m(t)$ then $Z_{D}(\varphi(t))=0$.
\end{proof}

\begin{cor}
	Let $n$ be a positive integer. We recall that $e_{n}\in\TR$ is the ladder of degree $n$, $Y_{n}\in\D$ is the path tree of degree $n$ and $X_{n}\in\D$ is the corolla of degree $n$. Let $t$ a sub-binary tree. We have $Z_{Y_{n}}(\varphi(t))=\begin{cases} n! \mbox{ if } t=e_{n},\\0  \mbox{ else} \end{cases}$ and $Z_{X_{n}}(\varphi(t))\neq0$. Actually $Z_{X_{n}}(\varphi(t))=2^{\interne(t)}$ where $\interne(t)$ is the number of internal vertices \emph{ i.e.} the number of vertices with at least one child.
\end{cor}

\textbf{Acknowledgments.} I would like to thank the University of the Littoral Opal Coast and the région Hauts-de-France for their financial support. I would like to thank my supervisor for his support too.

\bibliographystyle{siam}
\bibliography{biblio_dissection_diagrams_mammez}
\end{document}